\documentclass[12pt]{article}
\usepackage{amsmath,latexsym,amssymb,amsthm,enumerate,amsmath,amscd}
\usepackage[mathscr]{eucal}
\usepackage{color}
\definecolor{med-gray}{gray}{0.5}
\definecolor{gray1}{gray}{0.85}
\definecolor{gray2}{gray}{0.90}
\definecolor{gray3}{gray}{0.64}
\definecolor{gray4}{gray}{0.55}
\definecolor{verylight-yellow}{rgb}{1,1,0.7}
\definecolor{yellow}{rgb}{1,1,0.2}
\definecolor{vivid-blue}{rgb}{0.2,0,1}
\definecolor{light-pink}{rgb}{1,0.8,1}
\definecolor{med-pink}{rgb}{1,0.6,1}
\definecolor{aqua}{rgb}{0.0, 1.0, 1.0}
\definecolor{light-gray}{rgb}{0.5, 0.9, 0.5}
\usepackage[usenames,dvipsnames]{xcolor}
\usepackage{colortbl}
\usepackage[all,cmtip]{xy}
\usepackage{comment}
\usepackage{array}
\usepackage{makeidx}
\makeindex
\theoremstyle{plain}
\newtheorem{theorem}{Theorem}[section]
\newtheorem{proposition}[theorem]{Proposition}
\newtheorem{corollary}[theorem]{Corollary}
\newtheorem{lemma}[theorem]{Lemma}
\theoremstyle{definition}
\newtheorem{definition}[theorem]{Definition}

\newtheorem{remark}[theorem]{Remark}
\newtheorem{principle}[theorem]{Principle}
\newtheorem{example}[theorem]{Example}

\newtheorem{question}[theorem]{Question}

\newtheorem{thm}{Theorem}

\newtheorem{rem}[thm]{Remark}

\newtheorem*{ack}{Acknowledgment}

\numberwithin{equation}{section}
\numberwithin{table}{section} 
\DeclareMathOperator{\Gor}{Gor}
 
\DeclareMathOperator{\Grass}{\rm{Grass}}

\newcommand{\Aut}{\ensuremath{\mathit{Aut}}}

\newcommand{\df}{\ensuremath{\mathrm{Diff}}}

\newcommand{\Hom}{\ensuremath{\mathit{Hom}}}

\newcommand{\kk}{\ensuremath{\mathsf{k}}}

\newcommand{\maxA}{\ensuremath{\mathfrak{m}_A}}
\newcommand{\maxR}{\ensuremath{\mathsf{m}_R}}

\DeclareMathOperator{\Ann}{\mathrm{Ann}}

\DeclareMathOperator{\lt}{lt}
\DeclareMathOperator{\ord}{ord}

\def\Gr{\mathrm{Gr}}
\def\Gor{\mathrm{Gor}}

\def\cha{\mathrm{char}\ }

\def\Hom{\mathrm{Hom}}

\def\\k{{\mathsf{k}}}

\def\Hilb{\mathrm{Hilb}}

\def\Grass{\mathrm{Grass}}

\def\<{\left<}
\def\>{\right>}
\def\lt{\mathrm{lt}}

\def\ns{\footnotesize \it}
\def\m{\mathfrak{m}}

\def\max{\mathrm{max}}

\title{Symmetric Decomposition of  the Associated Graded Algebra of an Artinian Gorenstein Algebra\footnote{ Email addresses: \texttt{a.iarrobino@northeastern.edu}, \texttt{pmm@uevora.pt}}}
\author{
Anthony Iarrobino\\[.05in]
{\ns Department of Mathematics, Northeastern University, Boston, MA 02115,
USA.
}\\[.2in] Pedro Macias Marques\\[.05in]
{\ns Departamento de Matem\'{a}tica, Escola de Ci\^{e}ncias e Tecnologia, Centro de Investiga\c{c}\~{a}o}\\[-.05in]
{\ns  em Matem\'{a}tica e Aplica\c{c}\~{o}es, Instituto de Investiga\c{c}\~{a}o e Forma\c{c}\~{a}o Avan\c{c}ada,}\\[-.05in]
{\ns Universidade de \'{E}vora, Rua Rom\~{a}o Ramalho, 59, P--7000--671 \'{E}vora, Portugal}\\
}

\date{December 8, 2018, revised July 28, 2020}

\begin{document}

\newcounter{previouscounter}

\newenvironment{previous}{\stepcounter{previouscounter}\vspace{.5em}\texttt{Previous version (\thepreviouscounter):}\par 
\begin{tabular}{p{.97\textwidth}}\hline}{\\\hline\end{tabular}}

\maketitle
\begin{abstract}
We study the symmetric subquotient decomposition of the associated graded algebras $A^\ast$  of a non-homogeneous commutative  Artinian Gorenstein (AG) algebra $A$. This decomposition arises from the stratification of $A^*$ by a sequence of ideals $A^*=C_A(0)\supset C_A(1)\supset\cdots $ whose successive quotients ${Q(a)=C(a)/C(a+1)}$ are reflexive $A^*$ modules. These were introduced by the first author \cite{I6,I5}, developed in the Memoir \cite{I1}, and have been used more recently by several groups, especially those interested in short Gorenstein algebras, and in the scheme length (cactus rank) of forms.\par For us a \emph{Gorenstein sequence} is an integer sequence $H$ occurring as the Hilbert function ${H=H(A)}$ for an AG algebra $A$, that is not necessarily homogeneous.   Such a Hilbert function $H(A)$ is the sum of symmetric non-negative sequences ${H_{A}(a)=H\bigl(Q_A(a)\bigr)}$, each having center of symmetry ${(j-a)/2}$ where $j$ is the socle degree of $A$: we call these the 
\emph{symmetry conditions}, and the decomposition ${\mathcal{D}(A)=\bigl( H_A(0), H_A(1),\ldots \bigr)}$ the \emph{symmetric decomposition} of $H(A)$ (Theorem \ref{mainoldthm}).
We here study which sequences may occur as the summands $H_A(a)$: in particular we construct in a systematic way examples of AG algebras $A$ for which $H_A(a)$ can have interior zeroes, as ${H_A(a)=(0,s,0,\ldots ,0,s,0)}$. We also study the symmetric decomposition sets $\mathcal{D}({A})$, and in particular determine which sequences $H_A(a)$ can be non-zero when the dual generator is linear in a subset of the variables (Theorem \ref{DArestrictthm}). \par
Several groups have studied ``exotic summands'' of the Macaulay dual generator $F$: these are summands that involve more successive variables than would be expected from the symmetric decomposition of the Hilbert function $H(A)$. Studying these, we recall a normal form for the Macaulay dual generator of an AG algebra that has no ``exotic'' summands (Theorem~\ref{normalform5thm}). We apply this to Gorenstein algebras that are connected sums (Section~\ref{connectedsumsec}). \par
We give throughout many examples and counterexamples, and conclude with some open questions about symmetric decomposition.
\end{abstract}
\tableofcontents
\section{The quotients $R/\mathcal{C}_A(a)$ determined by an Artinian Gorenstein algebra,  and Macaulay inverse systems.}\label{intro1sec}
Let $\kk$ be an arbitrary field, $R$ the completed local ring $R=\kk\{x_1,\ldots ,x_r\}$ in $r$ variables, and denote by $\mathfrak{D}=\kk_{DP}[X_1,\ldots, X_r]$ the divided power algebra in $X_1,\ldots ,X_r$.  Here $R$ acts on $\mathfrak{D}$ by contraction.
F.H.S. Macaulay \cite{Mac1} showed that giving an ideal $I$ of $R$ defining an Artinian quotient ${A=R/I}$ of length ${\dim_{\kk}A=n}$ is  equivalent to giving a length-$n$ $R$-submodule  $\hat {A}$ of the divided power algebra $\mathfrak{D}$.  Given the ideal $I$, we call the $R$-module $\hat {A}$ the \emph{Macaulay inverse system}\index{Macaulay inverse system}	\index{Macaulay inverse system|see {dual generator}} ${\hat{A}=I^\perp}$ of $I$; since $I$ acts as zero, $\hat {A}$ is a module over  $A$. An $R$-closed length-$n$ submodule $M$ of $\mathfrak{D}$ is a module over the Artinian algebra ${A=R/\Ann_RM}$ (see Lemma~\ref{Macduallem}). 
The Artin algebra $A$ is \emph{Gorenstein}\index{Artinian Gorenstein algebra} when the inverse system $\hat{A}$ has a single generator $f_{A}$: then $f_A$ is unique up to multiple by a differential unit (i.e.\ action by a unit in $R$). Let $\maxA$ denote the maximum ideal of $A$. We define the \emph{socle degree}\index{socle degree} $j_A$ of $A$ as
\begin{equation}
j_A=\max\{ i\mid \maxA^{\,i} \neq0\}, 
\end{equation}
Then the degree ${\deg f_A=j_A}$ and $f_A$ is a generator of the cyclic $A$-module $\hat{A}$:  we will call $f_A$ a \emph{dual generator} of $A$.
The first author showed in \cite{I6,I5,I1} that the associated graded algebra ${A^\ast=\Gr_{\maxA}(A)}$ of an Artinian Gorenstein algebra $A$ has a canonical stratification by ideals ${C(a)=C_A(a)}$ whose successive quotients ${Q(a)=Q_A(a)\cong C(a)/C(a+1)}$ are reflexive $A^*$ modules: there is an exact pairing 
\begin{equation}\label{reflexiveeq}
Q(a)_i\times Q(a)_{j_A-a-i}\to \kk,
\end{equation}
induced by the exact pairing $A\times A\to {\kk}$, $(h, h')\mapsto (hh'\circ f_A)_{0}$
(\cite[Theorem 1.5]{I1}).
We term these $Q(a)$ the \emph{symmetric subquotients} of  $A^\ast$.
\index{subquotient $Q_A(a)$}
The Hilbert function $H(A)$ may be written as a sum of symmetric sequences ${H_A(a)=H\bigl(Q_A(a)\bigr)}$, each with a center of symmetry at  $(j_A-a)/2$: we term these sequences the \emph{symmetric components} of the  Hilbert function. The \emph{symmetric decomposition}\index{symmetric decomposition of Hilbert function} $\mathcal{D}(A)$ of $H(A)$ is the sequence (of sequences)
\begin{equation}\label{symdecomp1eq}
\mathcal{D}(A)=\bigl(H_A(0),H_A(1),\ldots ,H_A(j_A)\bigr).
\end{equation}
The components $H(a)$ satisfy the \emph{symmetry conditions}
\index{symmetry conditions on $H(a)$}
\begin{equation}\label{symmetrycondeq}
H(a)_i=H(a)_{j-a-i}\text { for } 0\le i\le j-a,\ 0\le a\le j-2.
\end{equation}
and as well the \emph{Macaulay conditions}\index{Macaulay conditions} that for each $a$, $0\le a\le j-2$,
\begin{equation}\label{Macaulaycondeq}
\sum_{i=0}^a H_A(i) \text { is an O-sequence}\footnote{An O-sequence $H$ is one that occurs as the Hilbert function $H=H(R/M)$ where $M$ is a monomial ideal, which may be assumed to be lex-initial. This is equivalent to $H$ satisfying certain well-known growth conditions shown by Macaulay in \cite{Mac2} -- see also \cite{Wh,BrHe}}
\end{equation}
since it occurs as the Hilbert function of the Artinian algebra $R/C_A(a+1)$. 
\index{O-sequence}\par
Furthermore, writing the dual generator $f_A=f_j+f_{j-1}+\cdots$ the sequence $H_A(0)$ satisfies
\begin{equation}\label{HA0eq}
H_A(0)=H(R/\Ann f_j),
\end{equation}
so is the Hilbert function of a graded Gorenstein algebra of socle degree $j_A$.\par 
The stratification $C_A(a)$ of $A^\ast$ is determined by the AG algebra $A$, but is not in general determined solely by the associated graded algebra $A^\ast$, except in two variables, or when $H(A)=H_A(0)$ (this equality is equivalent to $H(A)$ being symmetric about $j/2$ by Lemma~\ref{Watanabelem}). That is, two different AG algebras $A$, $B$ may have the same associated graded algebra $A^\ast=B^\ast$, but determine different stratifications of $A^\ast$ with $C_A(a)\ncong C_B(a)$ for some positive integer $a$. Also, the same Hilbert function $H(A)=H( B)$ for two Artinian Gorenstein algebras $A$, $B$ may have different symmetric decompositions $\mathcal{D}(A)\not=\mathcal{D}(B)$
(Remark~\ref{diffdecomprem}, Examples~\ref{AG=ex}, \ref{simple=AG}).\par
The symmetric decomposition structure has very recently been studied or used in articles by  A. Bernardi and K.~Ranestad \cite{BR},  J. Elias and M. Rossi \cite{ER1,ER}, J. Jelisiejew \cite{Je},  G. Casnati and R.~Notari \cite{CN},  Y. H. Cho with the first author \cite{ChI}, and others, as \cite{BJMR,BJJM,CJN,MR}.\vskip 0.2cm
It is natural to ask if there are further conditions on the symmetric decompositions $\mathcal{D}(A)$ so on the component Hilbert functions $H(A)$ for AG algebras, besides the symmetry conditions \eqref{symmetrycondeq}, the Macaulay conditions \eqref{Macaulaycondeq}, and the condition \eqref{HA0eq} that $H_A(0)$ be a graded Gorenstein sequence. S. Masuti and M. Rossi show that all socle degree four decomposition sequences $\mathcal{D}$ satisfying the three conditions are realizable -- occur as some $\mathcal{D}(A)$ for an AG algebra $A$ \cite[\S 3.6,3.7]{MR}. Lists of realizable decompositions are proposed in \cite[\S 5F]{I1} for lengths $n\le 16$, and for lengths $n\le 21$ in embedding dimension at most three.\footnote{The author's intent as stated in \cite{I1} was that these lists be complete, and discussion is given in \cite[\S 5F]{I1} to justify that the ones listed can be realized, but they may need further checking for completeness.} \par 
When $A$ is assumed homogeneous, the sequences $H$ that occur as the Hilbert function $H(A)$ are known for codimension ${r\le 3}$, the case ${r=3}$ being a consequence of the D.~Buchsbaum and D.~Eisenbud Pfaffian structure theorem \cite{BuEi,D,IKa,CoVa}.  Also the specialization behavior of generator-relations strata of  codimension three graded AG algebras of fixed Hilbert function $H$ is understood (\cite{D}, \cite[\S 4]{IKa}). However, when ${r=4}$ the longstanding conjecture that for graded AG, algebras, the Hilbert functions $H(A)$ satisfy an additional SI condition\footnote{A sequence $H=(h_0,\ldots,h_j)$ of positive integers is SI if the first differences of the first half $(h_0,h_1,\ldots, h_{\lfloor j/2\rfloor})$ of $H$ is an O-sequence. See \cite{Har,MNZ}.} 
is still open \cite{MNZ,ElKSr}.\par
For non-homogeneous AG algebras, the decomposition $\mathcal{D}(A)$ is completely understood for codimension ${r=2}$ (where it depends only on $H(A)$); but in codimension ${r\ge 3}$, even for the case of a complete intersection, $\mathcal{D}(A)$ and $H(A)$ are not well understood.  In this paper our goal is to shed some further light on the possibilities for $\mathcal{D}(A)$, especially when ${r=3}$, the first open case, or when the dual generator $f$ is linear in some of the variables.\vskip 0.2cm
Given a partial symmetric decomposition $\mathcal{D}_{<a}=\bigl(H_A(0),H_A(1),\ldots ,H_A(a-1)\bigr)$ that occurs for some AG algebra $A$, and a specified codimension $r$ (so specifying the ring $R$) there is a sharp upper bound $M(a,{\mathcal{D}}_{<a})$ (see \cite[Theorem 3.2(A), 3.2(B)]{I1}), such that $H(A)_a\le M(a,{\mathcal{D}}_{<a})$ termwise and such that we have the termwise inequality\footnote{We say sequences $H,H'$ satisfy $H\le H'$ if $H(i)\le H'(i)$ for each integer $i\ge 0$ (termwise inequality).}
\begin{equation}\label{Maupperbdeq}
H(A)\le \left( \sum_{u=1}^{a-1} H_A(u)\right)+M(a,{\mathcal{D}}_{<a})
\end{equation}
If $H(A)$ achieves this upper bound of \eqref{Maupperbdeq} then $H_A(u)=0$ for $u>a$. Furthermore, the upper bound can be achieved simply by adding a general enough or generic\footnote{The phrase \textit{general enough} $h$, will denote for us an element $h$ in an open dense subset of an irreducible parameter space. The phrase \textit{generic} refers to using variables as parameters.} degree ${j-a}$ form ${h_{j-a}\in \mathfrak{D}_{j-a}}$ to the dual generator $f_A$ of $A$, obtaining a new dual generator ${F_{A'}=f_A+h_{j-a}}$ and a new Gorenstein algebra ${A'=R/\Ann F_{A'}}$.   We restate this result for convenience in Section~\ref{amodsec} (Proposition~\ref{maxprop}). \par
In \cite[Theorem 2.2]{I1} it is shown that in codimension two ($R={\kk}\{x,y\})$ each $Q(a)$ is a  reflexive $A^\ast$ module having a single generator, so is cyclic.  However, in \cite[Ex.\ 1.6, 4.7]{I1} a codimension-three complete intersection was given with $H(2)=(0,1,0,1,0)$, so $Q_A(2)$ is non-cyclic. In this example $H(0)=(1,1,1,1,1,1)$, $H(1)=(0,1,2,2,1,0) $ and $H(A)=(1,3,3,4,2,1,1)$.\footnote{The AG algebras of Hilbert function $H=(1,3,3,4,2,1,1)$ are determined below in Proposition \ref{prop1.24}. }
When $H_A(a)$ has interior zeroes, then $Q_{A}(a)$ must be generated in several different degrees, so $Q_A(a)$ is a non-cyclic $A^\ast$ module, in contrast to the situation for two variables.\par
The result that there is a termwise upper bound $M(a,{\mathcal{D}}_{<a})$ to $H_A(a)$ given a partial decomposition $\mathcal{D}_{< a}(A)$ (e.g.\ Equation \eqref{Maupperbdeq} and Proposition \ref{maxprop}), and the above example where $H(2)=(0,1,0,1,0)$, suggest the following more general question. Let $R={\kk}\{x_1,x_2,\ldots, x_r\}$. 
\vskip 0.2cm\noindent
\textbf{Question 1}. Given an AG algebra $A=R/I$ such that $H(A)$ has the given partial decomposition ${\mathcal{D}}_{<a}(A)$, what symmetric sequences are possible as the next component Hilbert function $H_{A'}(a)$ for a deformation $A'$ of $A$ with $\mathcal{D}'_{<a}(A')= \mathcal{D}_{<a}(A)$? In particular,\par
(a) Given a graded Gorenstein sequence $H(0)$ and an integer ${a>0}$ can we always find an AG algebra $A$ such that ${H_A(0)=H(0)}$ and $H_{A}(a)=(0,\ldots,0,s,0\ldots,0, s,0,\ldots ,0)$  has a subsequence of interior zeroes? What other sequences $H_A(a)$ can we construct?\par
(b) What restrictions does $\mathcal{D}_{< a}(A)$ impose on $\mathcal{D}_{\ge a}=(H_A(a),H_A(a+1),\ldots $)?Œ\vskip 0.2cm
In Section \ref{constructsec} we explain how we use the dual generator to determine AG algebras. We define subquotients $Q_A^\vee (a)_{\mathfrak{D}}$ of the associated graded module $\Gr\bigl(\Hom(A,\kk)\bigr)$ of the dual to $A$ (Lemma \ref{Qdualdescription} and Definition \ref{Wuvdef}). We then explain how we use the dual generator to choose $Q^\vee(a)_{\mathfrak{D}}$ and hence to determine AG algebras, in Examples \ref{stdex} and \ref{varydualex}.\par
In Section \ref{Hasec} we consider deformations of ${A=R/\Ann f}$ to ${A'=R/\Ann F}$ where $F=f+Zh$ with ${f,h\in \mathfrak{D}}$ and $Z$ is a new variable. We give a general way to construct such examples of AG algebras for which $H_{A'}(a)=(0,1,0, \ldots, 0,1,0)$. Also, using $F=f+\sum_{i=1}^s Z_ih_i$, we give examples for which $H_{A'}(a)=(0,s,0, \ldots, 0,s,0)$
(Proposition~\ref{noncyclic1prop}). Our process explains and greatly generalizes \cite[Examples\ 1.6 and 4.7]{I1} which, at the time, seemed quite mysterious. 
Understanding better this mystery was a main motivation for us in writing this paper. In Proposition~\ref{noncyclic1prop} the forms $h_1,\ldots ,h_s$ in $F$ all have the same degree $k$, which is one greater than the maximum degree $i=k-1$ for which $(H_A)_{i}=r_{i}$ where $r_i=\dim_{\kk}R_i$: that is, $(R\circ f)_{k-1}=\mathfrak{D}_{k-1}$, so the partials of $f$ fill  the available space. This determines the special form of the Hilbert function $H_{A^\prime}(a)$ of the deformation $F$ of $f$. \par
Because of this restriction on the degrees of $h_1,\ldots, h_s$, those examples are still quite special.  In Section~\ref{restrictsymdecompsec} we show conversely that for dual generators $F=f_j+\cdots+f_{j-a+1}+f_{j-a}+\cdots $ with ${f_j+\cdots+f_{j-a+1}\in \mathfrak{D}}$ (variables  $X_1,\ldots ,X_r$) but $f_{j-a}\in \mathfrak{E}=\mathfrak{D}_{DP}[Z_1,\ldots ,Z_s]$ with ${a\ge 1}$ then in order for $H_A(a)$ to have interior zeroes $f_{j-a}$ must be linear in the new variables $Z_1,\ldots ,Z_s$ (Lemma~\ref{linearZlem}). These results together give a positive answer to Question~1(a).\par
We then study more general AG algebras $A_F$ whose dual generators satisfy  $F=f+\sum_{i=1}^s Z_ih_i$, where the $Z_i$ are (new) variables, and where $f,h_1,\ldots, h_s$ are homogeneous elements of the dual module $\mathfrak{D}$ of $R$ that have specified -- possibly different -- degrees. In one of our main results, we characterize in terms of the degrees of $h_1,\ldots,h_s$, which component symmetric decomposition Hilbert functions $H_{A_F}(a)$ can be non-zero; and we as well determine the non-zero graded symmetric modules $Q^\vee (a)$. Equivalently, we determine the possible symmetric subquotients $Q(a)$ of $A_F^\ast$ (Theorem~\ref{DArestrictthm}).   \vskip 0.2cm
In \cite{I1}  examples of non-graded Artinian Gorenstein algebras were constructed in steps by first choosing the top degree form $f_j$ of the dual generator $f$ then $f_{j-1}, f_{j-2},\ldots $ in order to attain a desired Hilbert function decomposition $\mathcal{D}$ (see Definition \ref{Ddef}). 
The method uses the result that $\mathcal{D}_{\le a}(A)$ depends only on $f_j+\cdots +f_{j-a}$ (Corollary \ref{partialdecompcor} and Principle \ref{principlefj-a}). It works particularly well when $H_A(a)$ is to be the maximum possible sequence $M(a,{\mathcal{D}}_{<a})$, given the sequence $H_A(0), H_A(1),\ldots, H_A(a-1)$ and the codimension of $A$. It also works well when $f$ is a sum of powers of suitably general linear forms \cite[Theorem 5.8]{I1}. This step-by-step construction was used to give tables of the actually occurring symmetric Hilbert function decompositions for AG algebras of small enough lengths \cite[Appendix 5F]{I1}. However, at each step the algebra determined by the dual generator $f_j+\cdots +f_{j-a}$ may well have some higher component Hilbert function ${H(u)\neq0}$ for ${u>a}$, even though $f_j+\cdots +f_{j-a}$ determines $\mathcal{D}_{\le a}(A)=\bigl( H_A(0),\ldots ,H_A(a) \bigr)$ uniquely (see Section \ref{cautionsec}, Examples \ref{cautionex1}, \ref{cautionex2}). \par
In Section \ref{nonubiquitycod4sec} we partially answer Question 1(b) by showing that $\mathcal{D}_{\le 1}(B)$ being specified may force $\mathcal{D}_2(B)$ to be non-zero if $B$ is determined by a dual form $G=G_j+G_{j-1}$ (Example~\ref{nonubiq4ex} and Theorem \ref{nonubiq4thm}). We call this partial non-ubiquity as we restrict the algebras $B$ considered. The non-ubiquity of a symmetric decomposition would say that given any pair $\bigl(a,\mathcal{D}(A)\bigr)$ there is no AG algebra $B$ such that $\mathcal{D}_{\le a}(A)=\mathcal{D}(B)$, the full symmetric decomposition for $B$. It is open whether there is a pair $\bigl(a,\mathcal{D}(A)\bigr)$ for an $ a\ge 1$ that is non-ubiquitous.
\index{ubiquity} 
\vskip 0.2cm\noindent  We next consider 
\vskip 0.2cm\noindent
\textbf{Question 2}. Is there a normal or canonical form for the dual generator of an AG algebra $A$, up to isomorphism? Can we parametrize the isomorphism classes of algebras?\vskip 0.2cm
Our work on this question is in Section \ref{tail sec}. We first recall the adjoint linear map on $\mathfrak{D}$ to an automorphism of $A$. In Section \ref{normalformsec} we amplify and make more precise the proof of the structure result \cite[Theorem 5.3]{I1} that gives a weak canonical form for the dual generator of an AG algebra $A$, up to isomorphism. This has also been studied in \cite{ER,Je,CJN}.  ``Exotic summands'' of the dual generator were introduced by A. Bernardi and K. Ranestad in a preprint that led to \cite{BR}; with J.~Jelisejew and the second author they studied these further in \cite{BJMR}. In Section~\ref{exoticsumsec} we show -- as was pointed out by J. Jelisiejew -- that the weak canonical form of the dual generator given  in \cite[Theorem~5.3]{I1} implies that an AG algebra is isomorphic to one whose inverse system has no ``exotic summands''  (Theorem \ref{normalform5thm}).  We then in Section \ref{connectedsumsec} discuss connected sums -- where the dual generator $F$ is the sum of terms in distinct variables. In particular, we view some results of  \cite{ACY}  through the lens of the weak canonical form.\par
By taking up the above two questions, our intent is to deepen the study of symmetric decompositions for Artinian Gorenstein algebras, and to suggest new problems.
Our main tools are  inverse systems (see \cite{Mac1,Em,Mo,MS}) and the particular linear and sometimes non-linear algebra that comes up in studying symmetric decompositions for AG algebras whose dual generators are either $F=f+Zh$ or more generally $F=f+\sum_{i=1}^s Z_ih_i$, where $ f,h, h_i\in \mathfrak{D}$ and $Z, Z_1,\ldots Z_s$ are new variables. We also use the theory of  compressed algebras \cite{Sch,I2,FrLa}.  We include a Section \ref{cautionsec} of cautionary examples, and we provide as well throughout many examples illustrating our work. We hope that not only our results but also our methods, and as well the many open questions and problems we include might be useful to others.\par
\subsubsection{Brief Outline.}
The paper is organized as follows: in Sections \ref{basicssec} to \ref{constructsec}, we review the basic theory and describe the tools we use throughout the paper; we include a list of cautionary examples (Section~\ref{cautionsec}). In Section \ref{Hasec} we present a systematic way of constructing AG algebras whose Hilbert function admits a symmetric summand of type $(0,s,0,\ldots,0,s,0)$. In Section \ref{AGA2Hfsec} we give examples of associated graded algebras having two different Hilbert function decompositions. In Section \ref{restrictsymdecompsec} we explore the relation between the linearity of the dual generator in some variables to the existence of interior zeroes in a symmetric summand of the Hilbert function.  

In Sections \ref{normalformsec} and \ref{exoticsumsec} we recall a normal form theorem from \cite{I1}; we include further details, and we discuss the relation of the normal form with the removal of exotic summands, up to isomorphism. Section \ref{isomclasssec} includes an example where the Hilbert function of $A=R/\Ann f$ depends on the characteristic of the field. In Section \ref{connectedsumsec} we use the Normal Form Theorem to give a different proof of a theorem by H. Ananthnarayan, E. Celikbas, and Z. Yang concerning connected sums. Finally in Section \ref{questionsec} we list questions and open problems related to the topics discussed in the paper.\par

\subsection{Artinian Gorenstein algebras and symmetric decomposition.}\label{basicssec}
Let ${R= {\kk}\{ x_1,\ldots ,x_r\}}$ be a complete local ring over a field ${\kk}$, with maximum ideal $\maxR$, and denote by $\mathfrak{D}={\kk}_{DP}[X_1,\ldots ,X_r]$ the ring of divided powers in $r$ variables $X_1,\ldots ,X_r$, where $X_i^{\,[k]}$ is the divided $k$-th power, and $X_1^{\,k}=k! X_1^{\,[k]}$. The ring $R$ acts on ${\mathfrak{D}} $ by contraction:
\begin{equation}\label{actioneq}
x_i^{\,k}\circ X_i^{\,[K]}=
\begin{cases} 
X_i^{\,[K-k]}&\text { if } K\ge k, \\
0 &\text { if } K<k.
\end{cases}
\end{equation}
Let ${A=R/I}$ be an Artinian Gorenstein (AG) quotient,  with maximum ideal ${\maxA=(x_1,\ldots ,x_r)}$. (we use a slight difference in typography to distinguish the maximum ideals $\maxA$ of $A$ and $\maxR $ of $R$). We have (\cite{Mac1}, \cite[Lemma 1.1]{I1}):
\begin{lemma}[AG algebras and $\kk$-linear maps of $R$]
\label{Macduallem}
There is a 1-1 isomorphism of sets
\begin{align*}
&\{ \text{AG quotients } A \text{ of } R \text{ having socle degree } j \} \Leftrightarrow \\
&\{ {\kk}\text{-linear homomorphisms } \phi: R\to {\kk}, \text{ with } \phi|_{\maxR^{\,j+1}}=0 \text{ but } \phi|_{\maxR^{\,j}}\neq0\}.
\end{align*}
Here ${A=R/I}$ with ${I=\{h\mid \phi (R\cdot h)=0\}}$.
\end{lemma}
\begin{definition}
\label{basicdef}
Recall that the \emph{socle degree} of $A$ is the integer $j$ such that ${\maxA^{\,j}\neq 0}$ but ${\maxA^{\,j+1}=0}$. Consider the associated graded algebra\index{associated graded algebra} ${A^*=\bigoplus_{i\ge0}A_i}$, where ${A_i=\maxA^{\,i}/\maxA^{\,i+1}}$. The \emph{Hilbert function}\index{Hilbert function} of $A$ is the sequence ${H(A)=(h_0,h_1,\ldots )}$, with ${h_i=\dim_{\kk}A_i}$. We denote by $\Gor(R,j)$ the set of Artinian Gorenstein quotients ${A=R/I}$ of $R$ having socle degree\index{socle degree} $j$ and by $\Gor_H(R)$ the parametrized family of those AG quotients of $R$ having Hilbert function $H$, and by $G_H(R)$ the family of graded quotients $A=R/I$, $I$ homogeneous in $R$. Here $H=(h_0,h_1,\cdots,h_j)$ and $\sum h_i=n$. We give $\Gor_H(R)$ and $G_H(R)$ respectively the reduced subscheme structure arising from the inclusions 
\begin{align}
\Gor_H(R)&\subset\Gor(R,j)\subset \Grass(n,R/ {\maxR}^{j+1});\notag\\
G_H(R)&\subset \prod_{1\le i\le j_H} \mathrm{Grass}(h_i, r_i).
\end{align}
Here $r_i=\dim_{\kk} R_i={\binom{r+i-1}{i}}$ and the map $\pi: \Gor_H(R)\to G_H(R): \pi (A)=A^\ast=Gr_{\maxA}A$ is a morphism.
\end{definition}
F.H.S. Macaulay showed  
\index{dual generator of A@dual generator of $A$}%
\index{apolar generator@apolar generator of $\mathcal{A}$}%
\index{apolar generator@apolar generator of $\mathcal{A}$|see {dual generator}}%
\begin{lemma}[Macaulay duality for AG local algebras]{\rm \cite{Mac1}}
\label{maclem1}
There is an isomorphism $\beta$ of sets from $\Gor(R,j)$ to the set of principal inverse systems $\{R\circ f\mid f\in \mathfrak{D}, \deg f=j\}$. Here 
\begin{equation}
\beta (A) =\{Q\in \mathfrak{D}\mid I\circ Q=0\}, \text { and } \beta^{-1}(R\circ f) = R/\Ann f.
\end{equation}
\end{lemma}\noindent
\index{Macaulay duality}%
\index{Macaulay duality|see {dual generator}}%
We call such an $f$ a \emph{dual generator} or \emph{apolar generator} of the AG algebra $A$. Given ${A=R/I}$ as a quotient of $R$, $f$ is unique only up to multiplication by a differential unit:
${\Ann f=\Ann(u\circ f)}$, where $u$ is any unit of $R$. \par
Given an Artinian Gorenstein quotient $A$ of socle degree $j$, and a ${\kk}$-linear map $ \phi: A\to {\kk}$ that is surjective on the socle $A_j$, F.H.S.~Macaulay showed also that we have an exact pairing,
\begin{equation}\label{exactpair2eq}
\langle \cdot,\cdot\rangle_\phi: A\times A\to {\kk} \text{ by }\langle h,h'\rangle_\phi =\phi (hh').
\end{equation}\label{frakCaeq}\noindent
Evidently, in this pairing we have that the annihilator of $\maxA^{\,i}$ is ${0:\maxA^{\,i}=\{a\in A\mid \maxA^{\,i}\cdot a=0\}}$; although we may regard this as a perpendicular space of $\maxA^{\,i}$ for the pairing, we will reserve the notation $I^\perp\subset \mathfrak{D}$ for the Macaulay inverse system $I^\perp=\{f\in \mathfrak{D}\mid I\circ f=0\}$. 
\par
The first author in \cite{I1} defined a filtration of the associated graded algebra  $A^\ast=\Gr_{{\maxA}}(A)$ by the graded ideals $C_A(0), C_A(1),\ldots$\index{Ca ideal@$C_A(a)$ ideal of $A^\ast$} (we suppress the $A$ when it is understood) 
\begin{equation}\label{Ceq}
A^\ast = C(0)\supset C(1)\supset \cdots \supset C(j-1)=0,
\end{equation}
where 
\begin{equation}\label{Caieqn}
C(a)_i=\varrho \Bigl( \maxA^{\,i}\cap (0:\maxA^{\,j+1-a-i})/\bigl( \maxA^{\,i+1}\cap 
(0:\maxA^{\,j+1-a-i})\bigr)\Bigr)
\end{equation}
where $\varrho$ denotes the projection to $A_i$. The $A^\ast$ module structure comes most naturally from working with the quotients on the right, before applying $\varrho$. Note that since $(0:\maxA^{\,j-a-i})\subseteq(0:\maxA^{\,j+1-a-i})$, we have
\index{subquotient $Q_A(a)$}%
\begin{equation*}
C(a+1)_i=\varrho \left(\frac{\maxA^{\,i}\cap (0:\maxA^{\,j-a-i})}{\maxA^{\,i+1}
\cap (0:\maxA^{\,j-a-i})}\right)
\cong\frac{\maxA^{\,i}\cap (0:\maxA^{\,j-a-i})+
\maxA^{\,i+1}\cap (0:\maxA^{\,j+1-a-i})}{\maxA^{\,i+1}
\cap (0:\maxA^{\,j+1-a-i})},
\end{equation*}
and the quotients $Q _A(a)$ (we shorten to $Q(a)$) satisfy
\begin{equation}\label{Qaeqn}
Q(a)_i=\frac{C(a)_i}{C(a+1)_i}\cong
\frac{\maxA^{\,i}\cap (0:\maxA^{\,j+1-a-i})}
{\maxA^{\,i}\cap (0:\maxA^{\,j-a-i})+
\maxA^{\,i+1}\cap (0:\maxA^{\,j+1-a-i})}.
\end{equation}
The subquotients ${Q(a)=C(a)/C(a+1)}$ are reflexive $A^\ast$ modules. An alternative, and perhaps more natural way of viewing the structure of $Q(a)$ as $A^\ast $ modules is given in \cite[\S 1F]{I1}. The multiplication in $C(a)$ as ideal of $A^\ast$ is not simply defined in $A^\ast$ but uses the quotient in Equation \eqref{Caieqn}: see Example \ref{usevarrho}.
We now state the symmetric decomposition theorem from \cite{I1} that underlies our work.\footnote{We thank Larry Smith, who pointed out to us that W. Gr\"{o}bner's \cite{Gro} has some related material: that article includes decomposing ideals of an Artinian algebra into the intersection of irreducible (Gorenstein) ideals, but does not appear to include a result concerning the symmetric decomposition of the AG algebra of an irreducible algebra.} Recall that $j=j_A$ is the socle degree of the AG algebra $A$. The \emph{Hilbert function}\index{Hilbert function!decomposition} $H(A)$ of an Artinian algebra $A=R/I$ is the sequence\index{symmetric decomposition of Hilbert function}\index{symmetric subquotient $Q_A(a)$}
\begin{equation}\label{HFeqn}
H(A)=(h_0,h_1,\ldots ) \text { where } h_i=\dim_{\kk} \bigl( ( I\cap \maxR^{\,i}+\maxR^{\,i+1})/ \maxR^{\,i+1} \bigr).
\end{equation}
\index{symmetry conditions on $H(a)$!Theorem \ref{symdecompthm}}%
\begin{theorem}[\textsc{Symmetric decomposition}]
\label{symdecompthm} 
{\rm \cite[Theorem 1]{I5}, \cite[Theorem 1.5]{I1}\label{mainoldthm} } 
Let $A$ be an Artinian Gorenstein algebra. The exact pairing $A\times A\to {\kk}$ determines an exact pairing  
\begin{equation}\label{exactpaireq}
\phi_a: \, Q(a)_i\times Q(a)_{j-a-i}\to {\kk}.
\end{equation}
The Hilbert function $H(A)=\sum_a H\bigl(Q(a)\bigr)$, and each $H\bigl(Q(a)\bigr)$ is symmetric with center of symmetry
$(j-a)/2.$  Let $f=f_j+f_{j-1}+\cdots $ be a dual generator for $A$ (Lemma \ref{maclem1}). Then $Q(0)$ is a graded Artinian Gorenstein algebra whose dual generator is $f_j$. Also, $Q(0)$ is a maximum-length graded AG quotient of $A^*$.
\end{theorem}\noindent
\index{Q(a)decomposition@$Q(a)$ decomposition}%
We denote by $n=n(A)$ the length $n=\dim_{\kk}A$.
\begin{definition}\label{Ddef}  
A \emph{Gorenstein sequence}\index{Gorenstein sequence} is one that occurs as the Hilbert function of an Artinian Gorenstein algebra, not necessarily homogeneous. We term the sequence\index{DA@$\mathcal{D}(A)$ Gorenstein decomposition sequence}
\begin{equation}
\mathcal{D}(A)=\bigl(H(0), H(1),\ldots, H(j-2)\bigr),\text{ where }H(u)=H\bigl(Q_A(u)\bigr),
\end{equation}
the \emph{Hilbert function decomposition} of $H(A)$, and term the collection of sequences\par\noindent $\mathcal{D}=\bigl(H(0),H(1),\ldots , H(j-2)\bigr)$ that occur for an Artinian Gorenstein algebra a \emph{Gorenstein decomposition} sequence.  We will sometimes write $H_A(a)$ or $H(a)$ for $H\bigl(Q_A(a)\bigr)$.

\end{definition}
The following is an immediate consequence of Theorem \ref{mainoldthm}.\footnote{We omit stating here the ``shell formula''   \cite[Proposition 1.9]{I1}, which gives the difference between the two sides of Equation \ref{overweighteqn}.}

\index{shell formula}%
\begin{corollary}
\label{symdeccor} 
Let $A$ be an AG algebra of socle degree $j$, and let $a\in \mathbb N$ satisfy $0\le a\le j-2$. Then the sequence $H(A)-\sum_{i=0}^a H_A(i)=\sum_{i=a+1}^{j-2} H_A(i)$ must be either symmetric with center $(j-a-1)/2$, in which case it is just $H_A(a+1)$ (and $H_A(i)=0$ for $i\ge a+2$) or else the sequence is overweighted in degrees less than $(j-a-1)/2$: that is (in the overweighted case)
\begin{equation}\label{overweighteqn}
\sum_{k< (j-a-1)/2}\,\sum_{\, i=a+1}^{j-2} H_A(i)_k> \sum_{k> (j-a-1)/2}\,\sum_{\, i=a+1}^{j-2} H_A(i)_k.
\end{equation}
\end{corollary}
We recall the following result, due to Junzo Watanabe \cite{W} that is also a consequence of the symmetric decomposition Theorem \ref{mainoldthm} (\cite[Proposition 1.7]{I1}): it follows from 
Corollary~\ref{symdeccor} applied to ${a=0}$. 

\index{associated graded algebra!when $H(A)$ is symmetric}%
\begin{lemma}
\label{Watanabelem} 
Assume that the Gorenstein sequence ${H=H(A)}$ is itself symmetric, so ${h_i=h_{j-i}}$ for ${0\le i\le j/2}$.
Then the associated graded algebra $A^*$ is itself Artinian Gorenstein, $H_A(0)=H(A)$ and $H_A(i)=0 $ for $i\ge 1$. 
\end{lemma}
For example, $H_1=(1,2,3,2,1)$  and $H_2=(1,2,3,3,2,2,2,1,1)$ are both Gorenstein sequences, here $H_1$ occurs for a graded AG algebra $A=R/\Ann \bigl(X^{[4]}+Y^{[4]}+(X+Y)^{[4]}\bigr)$. But $H_2=H(B)$, $B=R/\Ann \bigl(X^{[8]}+Y^{[7]}+(X+Y)^{[4]}\bigr)$, and $B^\ast $ cannot be Gorenstein.\footnote{Many authors have used ``Gorenstein sequence'' $H(A)$ for just the case $A$ is homogeneous, which by Lemma~\ref{Watanabelem} is the case that $H(A)$ itself is symmetric about $j/2$.} \par
\index{symmetric decomposition of Hilbert function!and ``magic square''}%
\index{Hilbert function!of $(0:\mathfrak{m}^b)$}%
\begin{example}[Symmetric decomposition of $H$ and ``magic square'']
\label{symdecompex}
Consider the dual generator $F=X^{[5]}+XY^{[2]}Z+W^{[2]}$, in the divided power ring $\kk_{DP}[X,Y,Z,W]$ and the corresponding AG algebra $A=R/\Ann F$, $R=\kk\{x,y,z,w\}$. The Hilbert function $H(A)=(1,4,5,3,1,1)$ and we write the symmetric decomposition $\mathcal{D}$ of $H$ in a standard form, one row for each sequence $H(a)=H\bigl(Q(a)\bigr)$, in Table~\ref{decomptable}. Note that the rising diagonals can be summed from the right, and give back the Hilbert function of $A$. This is because for an AG algebra $A$ the Loewy $(0:\maxA^{\,b})$ filtration is dual to the $\mathfrak{m}$-adic filtration. Thus ${H(0:\maxA^{\,3})=(0,0,0,\begin{array}{ccc}\cellcolor{yellow}5,&\cellcolor{med-pink}4,&\cellcolor{gray1}1\end{array})}$, reverses $H(R/\maxA^{\,3})=(1,4,5)$. Here $I=\Ann F$ satisfies $I=(xw,\,yw,\,zw,\,z^2,\,w^2-xy^2z,\,x^2y,\,x^2z,\,y^2z-x^4)$ and $I^\ast=(xw,\,yw,\,zw,\,z^2,\,w^2,\,x^2y,\,x^2z,\,y^2z,\,x^6)$.  The dual modules to the $Q(a)$ decomposition (Definition~\ref{Wuvdef}) satisfy $Q_A^\vee(3)_{\mathfrak{D}}=\langle W\rangle$ and $Q_A^\vee(0)_{\mathfrak{D}}=\langle 1;X;X^{[2]};X^{[3]};X^{[4]};X^{[5]}\rangle$, $Q_A^\vee(1)_{\mathfrak{D}}=\langle Y,\,Z;\,XY,\,XZ,\,YZ,\,Y^{[2]};\,XYZ,\,XY^{[2]}\rangle$.\par
Here each $Q^\vee (a)$ arises from the degree $j-a$ term of $F$. This does not always happen.\par
There is a second decomposition $\mathcal{D}_2$ possible for $H=(1,4,5,3,1,1)$, namely 
\begin{equation*}
\mathcal{D}_2=\bigl(H(0)=(1,1,1,1,1,1),\, H(1)=(0,2,3,2,0),\, H(2)=(0,1,1,0)\bigr),
\end{equation*}
that occurs for $B=R/J$, $J=\Ann G$, $G=X^{[5]}+Y^{[2]}Z^{[2]}+W^{[3]}$, $J=(xy,\,xz,\,xw,\,yw,\,zw,\,y^3,$  $z^3,\, w^3-y^2z^2,\,y^2z^2-x^5)$, and $J^\ast=(xy,\,xz,\,xw,\,yw,\,zw,\,y^3,\,z^3,\,w^3,\,y^2z^2,\,x^6)$. Also $Q_B^\vee (0)_{\mathfrak{D}}=\langle R\circ X^{[5]}\rangle$, $ Q_B^\vee(1)_{\mathfrak{D}}=\langle Y,\,Z;\,Y^{[2]},\,YZ,\,Z^{[2]};\,Y^{[2]}Z,\,YZ^{[2]}\rangle$, $Q_B(2)_{\mathfrak{D}}=\langle W;\,W^{[2]}\rangle$. Computations for this example and throughout the paper were either made or confirmed with the help of the software system Macaulay2 \cite{GS}.
\begin{table}
\qquad \qquad\qquad \qquad\qquad\qquad$ \begin{array}{c|ccccccc}
H(A^\vee)&&\cellcolor{gray2}1&\cellcolor{light-gray}1&\cellcolor{aqua}{3}&\cellcolor{yellow}{5}&\cellcolor{med-pink}{4}&\cellcolor{gray1}1\\
\hline\hline
H(0)&\cellcolor{gray2}1&\cellcolor{light-gray}1&\cellcolor{aqua}{1}&\cellcolor{yellow}{1}&\cellcolor{med-pink}{1}&\cellcolor{gray1}1&\\
H(1)&\cellcolor{light-gray}0&\cellcolor{aqua}{2}&\cellcolor{yellow}{4}&\cellcolor{med-pink}{2}&\cellcolor{gray1}0&&\\
H(2)&\cellcolor{aqua}0&\cellcolor{yellow}0&\cellcolor{med-pink}0&\cellcolor{gray1}0&&\\
H(3)&\cellcolor{yellow}0&\cellcolor{med-pink}{1}&\cellcolor{gray1}0&\\
\hline\hline
H(A)&1&4&5&3&1&1&\\
\end{array}$
\caption{A decomposition $\mathcal{D}$ for $H(A)=(1,4,5,3,1,1)$ (Example \ref{symdecompex}).}\label{decomptable}
\end{table}
\end{example}

\begin{definition}[Notation]\label{ordernotation}
\begin{enumerate}[(a)]
\item Given a completed local ring ${R=\kk\{x_1,\ldots,x_r\}}$ we denote by ${r_i=\dim_{\kk} R_i={\binom{r+i-1}{i}}}$; and, given an integer ${j\ge1}$, we denote by ${{c_{i,j}:=\min\left\{ r_i, r_{j-i}\right\}}}$; here $c_{ij}$ are the values of the Hilbert function of a compressed AG algebra ${A=R/I}$ with socle degree $j$.
\item Given an AG algebra $A$, with socle degree $j_A$, we denote by $H(A)$ or $H_A$ its Hilbert function, and by ${\mathcal{D}(A)=\mathcal{D}_A=\bigl(H_A(0),\ldots,H_A(j_A-2)\bigr)}$ its symmetric decomposition, where ${H_A(a)=H\bigl(Q_A(a)\bigr)}$ is the Hilbert function of the module $Q_A(a)$. When $R$ is unsderstood, we may denote ${H_f=H(R/\Ann f)}$, for ${f\in\mathfrak{D}}$, and use the analogous notations $\mathcal{D}_f$ and $Q_f(a)$.
\item The \emph{leading term}\index{leading term} $\lt(f)$ of $f\in \mathfrak{D}=\kk_{DP}[X_1,\ldots, X_r]$ (the divided power ring) is its highest degree term; recall that the \emph{initial term}\index{initial term}  $\mathrm{in}(\varphi)$ of $\varphi\in R$ is its lowest degree term.
\item  The \emph{order}\index{order} $\ord(\varphi)$ of $\varphi\in R$ is the degree of $\mathrm{in}(\varphi)$ so  $\varphi\in\maxR^{\,\ord \varphi}$ but $\varphi\notin\maxR^{\,\ord{\varphi}+1}$.  
\item
Given a dual generator $f\in \mathfrak{D}$, by the \emph{order} $\mathfrak{o}(g)$ of $g\in R\circ f\subset \mathfrak{D}$ we mean the highest power of the maximum ideal ${\maxR}$ of $R$ such that ${g\in \maxR^{\,\mathfrak{o}(g)}\circ f}$, but is not in ${\maxR^{\,\mathfrak{o}(g)+1}\circ f}$.
\item
For an ideal $I$ of $R$, we denote by ${I_i=\langle \mathrm{in} (\varphi) \mid \varphi \in I\cap {\maxR}^i\setminus I\cap {\maxR}^{i+1}\rangle}$, the degree $i$ component of $\Gr_{{\maxR}}(I)$.
\item For an ideal $I$ of $R$ we denote by $I^\perp=\{f\in \mathfrak{D}\mid I\circ f=0\}$. For an element $F\in \mathfrak{D}$ we denote by $\Ann F=\{h\in R\mid h\circ F=0\}$.
\item For a vector subspace ${B\subseteq R}$, an element ${f\in \mathfrak{D}}$, and  a vector subspace ${V\subseteq \mathfrak{D}}$ we denote by\footnote{Some of these notations are different from what we can find in other papers. For instance in \cite{BJMR}, $(R\circ f)_{\le t}$ is denoted as $\df(f)_t$ and $(\maxA^{\,s}\circ f)_{\le t}$ is $\df(f)_t^{j-s-t}$; while in \cite{CN} $(R\circ f)_{t}^\ast$ and $(R\circ V)_{t}^\ast$ are denoted by ${\mathrm{tdf}(f)_t}$ and ${\mathrm{tdf}(V)_t}$, respectively.}  
\begin{equation}\label{Diffeq}
\begin{split}
R\circ f &\text{ the vector space of partials of }f;\\
(R\circ f)_{\le t} &\text{ the partials of }f\text{ of degree at most }t;\\
\maxR^{\,s}\circ f &\text{ the partials of }f\text{ of order at least }s;\\
B\circ V &= \{\varphi\circ g \mid \varphi \in B, g\in V\};\\
(\maxR^{\,s}\circ f)_{\le t}&=(R\circ f)_{\le t}\cap\maxR^{\,s}\circ f;\\
(R\circ f)_{t}^\ast &=\langle \lt (g)\mid g\in (R\circ f)_{\le t}\rangle\cap \mathfrak{D}_t;\\
(R\circ f)^\ast &={\textstyle\bigoplus_{t\ge0} (R\circ f)_{t}^\ast =
\langle \lt(g)\mid g\in (R\circ f)\rangle};\\
(R\circ V)_{t}^\ast &=\langle (R\circ g)_{t}^\ast, g\in V \rangle.
\end{split}
\end{equation}
\end{enumerate}
\end{definition}
\subsection{The $a$-modifications of an AG ideal.}\label{amodsec}
We have defined the ideal $C_A(a)\subset A^\ast$ for $A$ an AG quotient $A=R/I$ (Equations \eqref{Ceq} and \eqref{Caieqn}). We let $\mathcal{C}_A(a)$ denote the pull back of $C_A(a)$ to $R$. 
We call attention to the identification of $\mathcal{C}_A(a)$ of $R$ in Lemma~\ref{modificationlem}(c) below, as the union of associated graded ideals of $a$-modifications of $A$: this can be a tool in constructing Gorenstein algebras, or in showing that certain symmetric decompositions of the Hilbert function $H(A)$ are impossible. We will term an ideal of $R$ an ``AG ideal''\index{AG ideal} if it defines an Artinian Gorenstein (AG) quotient $A=R/I$ and we will denote by $\mathfrak{D}=R^\vee$ the divided power algebra corresponding
to $R$ (dual algebra to $R$). We write $f=f_j+f_{j-1}+\cdots +f_0$ where $f_i$ is a homogeneous form in $\mathfrak{D}_i$.  Here $f_0$ is irrelevant, and if the embedding dimension is $r$ we may assume $f_1=0$.\par 
Two natural questions that arise in using a dual generator $f\in \mathfrak{D}$ to define the AG algebra $A=R/\Ann f$ are
\vskip 0.2cm\noindent
{\bf Question} A. What does a partial sum $f_{\ge j-a}=f_j+f_{j-1}+\cdots +f_{j-a}$ and the AG algebra $B=R/\Ann ( f_{\ge j-a})$ determine about $A=R/\Ann f$? \vskip 0.15cm\noindent
{\bf Question} B.  What symmetric decomposition results when we fix  $f=f_j+f_{j-1}+\cdots +f_{j+1-a}$ then choose $f_{j-a}$ a generic form of degree $j-a$, and consider $B=R/\Ann f'$, $f'=f+f_{j-a}$?  \vskip 0.2cm
We next introduce the concept of $a$-modification which we use to answer Question A in Lemma \ref{modificationlem} and Corollary \ref{partialdecompcor}.  Then we introduce the concept of \emph{relatively compressed $a$-modification} ($a$-RCM) and state Proposition \ref{maxprop}, that answers Question~B. We will be considering proper ideals $I,J\subset\maxR$.
\index{a-modification@$a$-modification}
\begin{definition}\label{amoddef}
We say that the Artinian Gorenstein (AG) ideal $J$ in $R$ is an $a$-\emph{modification} of the AG ideal $I$ (here $j$ is the socle degree of ${A=R/I}$)  if \footnote{This definition is consistent with p.\ 31 of \cite{I1}, but the usage there is primarily of ``relatively compressed $a$-modification'' (here $a$-RCM, see Definition \ref{compressdef} below).}
\begin{equation}\label{amoddefeq} 
I\cap \maxR^{\,j+1-a}=J\cap \maxR^{\,j+1-a}.
\end{equation}
For $f$ a dual generator of $J$ and $g$ a dual generator of $I$ above, we term $f$ an $a$-modification of $g$.
\end{definition}
Evidently this is a symmetric relation between $f$, $g$. Also, if $f$ is an $a$-modification of $g$ for $a\ge 1$ then $f$ is an $(a-1)$-modification of $g$. Any $f$, $g$ of degree $j$ satisfy $f$ is a $0$-modification of $g$; being a $j$-modification is equivalent to $f=u\circ g$ for some unit $u\in R$. \par
Recall that with $A$ a given AG algebra we denote by $\mathcal{C}(a)\subset R$ the graded ideal $\mathcal{C}(a)=(\pi^\ast)^{-1}\bigl(C(a)\bigr)$, where $\pi^\ast: R\to A^\ast$ is the natural projection:  that is, $\mathcal{C}(a)$ is the pull back to $R$ of the ideal $C(a)\subset A^\ast$. \par
For non-zero $h\in R$ we will denote by $\mathrm{in}(h)$ the lowest degree term. For $J$ an ideal of $R$ defining the quotient $A=R/J$ we denote by $J^\ast$ the graded ideal $J^\ast=\sum_{i=0}^\infty  J_i$, where $J_i=(\maxR^{\,i}\cap J+\maxR^{\,i+1})/\maxR^{\,i+1}$;  so $A^\ast\cong  R/J^\ast$. Part (b.ii) below is a consequence of (b.i) not explicitly stated in \cite{I1}. \par
\index{ideal C(a) of R@ideal $\mathcal{C}(a)$ of $R$}%
\index{a-modification@$a$-modification}%
For an ideal $I$ or $R$ defining an Arinian quotient, the Macaulay inverse system $I^\perp\subset \mathfrak{D}=\{f\in \mathfrak{D}\mid I\circ f=0\}$.\par

\index{ideal C(a) of R@ideal $\mathcal{C}(a)$ of $R$}%
\begin{lemma}{\rm \cite[Lemma 1.10, Theorem 3.6]{I1}}
\label{modificationlem} 
Let $f,g\in \mathfrak{D}$ and suppose ${\deg f=\deg g= j}$. Let $A=R/\Ann f$, $A'=R/\Ann g$.
\begin{enumerate}[(a)]
\item\label{modlemi} i. $I=\Ann f$ is an $a$-modification of $J$ if and only if we can write $J=\Ann (f+w)$ for some $ w\in \mathfrak{D}_{\le j-a}.$ \par
ii. Equivalently, ${I=\Ann f}$ is an $a$-modification of ${J=\Ann g}$ if and only if there is a unit $u$ in $R$ such that ${f= u\circ  g\mod  \mathfrak{D}_{\le j-a}}$.  
\item\label{modlemii} Let $f$ be an $a$-modification of $g$. Then we have \par
i. $\mathcal{C}_A(k)=\mathcal{C}_{A'}(k)$ for $k\le a$, and 
\par
ii. $Q_A(k)\cong Q_{A'}(k)$ as $R/{\mathcal{C}}_A(k+1)$ module for $k\le a-1$.
\item \label{modlemiii} The ideal $\mathcal{C}_A(a)$ of $R$ is the union of the associated graded ideals of all AG $a$-modifications of ${I=\Ann f}$. 
\end{enumerate}
\end{lemma}
\begin{proof} 
The proof is a bit scattered in \cite{I1} so we give some indications. \par\noindent
Proof of (a.ii). Suppose that ${I\cap \maxR^{\,j+1-a}=J\cap \maxR^{\,j+1-a}}$ with $J=\Ann g$. Then (see \cite[Eqn.~3.13]{I1})
\begin{equation}\label{amod2eq} 
R\circ f+\mathfrak{D}_{\le j-a}=I^\perp+(\maxR^{\,j+1-a})^\perp = (I\cap \maxR^{\,j+1-a})^\perp =(J\cap \maxR^{\,j+1-a})^\perp =
R\circ g+\mathfrak{D}_{\le j-a}.
\end{equation}
This implies there is a unit $u\in R$ such that $f=u\circ g$ mod $\mathfrak{D}_{\le j-a}$. The converse is evident.\par\noindent
 Proof of (a.i).  Taking $u=1$ above, we have $f+w$, $w\in \mathfrak{D}_{\le j-a}$, determines an $a$-modification of $A=R/\Ann f$. Conversely, if $J=\Ann g$ where $f=u\circ g +w'$, $w' \in  \mathfrak{D}_{\le j-a}$, we have also $J=\Ann(u\circ g) =\Ann (f+w)$ where $w=-u\circ w' \in \mathfrak{D}_{\le j-a}$.\par\noindent
Proof of (b). Suppose now ${f=u\circ g \mod \mathfrak{D}_{\le j-a}}$. Let $f=u\circ g+c$, with 
$c\in\mathfrak{D}_{\le j-a}$. We denote by $A-B$ (or $A\backslash B$) the difference set.
Then (after \cite[Eqn. 1.3]{I1})
\begin{equation}
\begin{aligned}
\mathcal{C}_A(a)_i&=\{ \mathrm{in} (h)\mid h\in \maxR^{\,i}-\maxR^{\,i+1} 
\text{ and } \maxR^{\,j+1-a-i}h\circ f=0\} \\
&=\{ \mathrm{in} (h)\mid h\in \maxR^{\,i}-\maxR^{\,i+1} 
\text{ and }h\circ (u\circ g+c)\in \mathfrak{D}_{\le j-a-i}\} \\
&=\{ \mathrm{in} (h)\mid h\in \maxR^{\,i}-\maxR^{\,i+1} 
\text{ and } h\circ g\in \mathfrak{D}_{\le j-a-i}\} \\
&=\mathcal{C}_{A'}(a)_i.
\end{aligned}
\end{equation}
\index{ideal C(a) of R@ideal $\mathcal{C}(a)$ of $R$}%
This holds also for each $k\le a$, since $f$ is an $a$-modification implies $f$ is an $a-1$ modification. Taking quotients, and noting that $R/\mathcal{C}_A(a)\cong A^\ast/C_A(a)$, we have that the quotient $Q_A(a-1)=C_A(a-1)/C_A(a)\cong \mathcal{C}_{A'}(a-1)/\mathcal{C}_{A'}(a)$, and we have
\begin{equation}
Q_A(a-1)\cong\mathcal{C}_A(a-1)/\mathcal{C}_A(a)=\mathcal{C}_{A'}(a-1)/\mathcal{C}_{A'}(a)\cong Q_{A'}(a-1).
\end{equation}
This implies the second part of \eqref{modlemii}.\vskip 0.2cm\par\noindent
Proof of \eqref{modlemiii}. Part \eqref{modlemii} implies that when $g$ is an $a$-modification of $f$, then $\mathcal{C}_A(a)=\mathcal{C}_{A'}(a)\supset  J^\ast$, so $\mathcal{C}_A(a)$ contains the union of associated graded ideals of all $a$-modifications of $f$. Now
suppose  that $h_t\in\mathcal{C}_A(a)$ is the lowest degree term of  an element
\begin{equation*}
h\in (I: \maxR^{\,j+1-a-t})\cap \maxR^{\,t}.
\end{equation*}
We use the fact ${h_t\circ \mathfrak{D}_k\to \mathfrak{D}_{k-t}}$ is surjective, to bootstrap and find ${g\equiv f \mod \mathfrak{D}_{\le j-a}}$ such that $h\circ g=0$. Then $g$ is an $a$-modification of $f$ and $h_t\in (\Ann g)^\ast$. First, taking $g(0)=f$, since $h\circ f \in \mathfrak{D}_{\le j-a-t} $ there is an element $w_{j-a}\in \mathfrak{D}_{j-a}$ so ${h_t\circ w_{j-a} \equiv h\circ f\mod \mathfrak{D}_{\le j-a-t-1}}$; we set ${g(1)=f-w_{j-a}}$ and note that ${h\circ g(1)\in \mathfrak{D}_{\le j-a-t-1}}$. If we have chosen $g(u)$ so $h\circ g(u)\in \mathfrak{D}_{\le j-a-t-u}$, we choose $w_{j-a-u}\in \mathfrak{D}_{j-a-u}$ such that $h_t\circ w_{j-a-u}\equiv h\circ g(u)\mod \mathfrak{D}_{\le j-a-t-(u+1)}$; we set $g(u+1)=g(u)-w_{j-a-u}$ and note that ${h\circ g(u)\in \mathfrak{D}_{\le j-a-t-(u+1)}}$.   Then $g=g(j-a-t+1)$ satisfies  $h\circ g=0$ and $g\equiv f\mod \mathfrak{D}_{\le j-a}$.
This completes the proof.
\end{proof}\par
Our next example shows that the condition of Lemma \ref{modificationlem}\eqref{modlemii} does not, conversely, imply that $f$ is an $a$-modification of $g$.
\begin{example}[Non-modification]
\label{non-modex} 
Let $f = X^{[4]} + Y^{[2]}$ defining the AG algebra $A=R/\Ann f=R/(xy,\,y^2-x^4)$ and $g= X^{[4]} + X^{[2]}Y$ defining $B=R/(xy-x^3,\,y^2)$; here $\mathcal{D}(A)=\mathcal{D}(B)=\bigl(H(0)=(1,1,1,1,1),\, H(1)=(0,1,0)\bigr)$. Also $A^\ast=B^\ast=R/(xy,\,y^2,\,x^5)$ and the subspace $C(1)=\langle y\rangle\subset A^\ast$ is the same for both algebras. But $(\Ann f)\cap \mathfrak{m}^3=(\langle x,y\rangle \cdot xy,\, y^3-yx^4,\,x^5)$ which is not equal to $(\Ann g)\cap \mathfrak{m}^3=(\langle x,y\rangle \cdot y^2,\,x^2y-x^4,\,x^5)$, so neither algebra $A$, $B$ is a $2$-modification of the other.
\end{example}
\begin{example}[Modification]
\label{modificationlemexample}
Let $R=\kk[x,y,z]$, $f=X^{[6]}+X^{[3]}Y^{[2]}+Z^{[4]}\in\mathfrak{D}=\kk_{DP}[X,Y,Z]$, and
$A=R/\Ann f$ of Hilbert function $H(A)=(1,3,4,3,2,1,1)$ and decomposition
\begin{equation*}
\mathcal{D}(A)=\bigl(H(0)=(1,1,1,1,1,1,1),\, H(1)=(0,1,1,1,1,0),\,
H(2)=(0,1,2,1,0)\bigr).
\end{equation*}
Here $\Ann f=(xz,\, yz,\,xy^2-x^4,\, y^3,\,z^4-x^6)$ and
$A^\ast=R/(xz,\,yz,\,xy^2,\,y^3,\,yx^4,\,z^4,\,x^7)$.
We determine  $\mathcal{C}_A(2)$. In degree $1$, since $x\circ f$ and
$y\circ f$ are partials of degrees $5$ and $4$, respectively, only $z$
satisfies $\maxR^{\,4}z\circ f=0$, since $z\circ f=Z^{[3]}$. So
$\mathcal{C}_A(2)_1=\langle z\rangle$. In degree $2$, all multiples of
$z$ belong to $\mathcal{C}_A(2)_2$, but also $y^2$, since it is the
initial term of $y^2-x^3$ and
$(y^2-x^3)\circ f=-Y^{[2]}$. In a similar manner, we can see that only
multiples of $z$ and $y^2$ belong to $\mathcal{C}_A(2)_3$, and we have
\begin{align*}
\mathcal{C}_A(2)_1 &= \langle z\rangle,\qquad \mathcal{C}_A(2)_2 = \langle R_1z, y^2\rangle,\\
\mathcal{C}_A(2)_4 &=\langle R_2\cdot \mathcal{C}_A(2)_2\rangle,\quad \mathcal{C}_A(2)_5=\langle R_3\cdot\mathcal{C}_A(2)_2,  yx^4\rangle\\
\mathcal{C}_A(2)_7 &=\langle  R_3\cdot \mathcal{C}_A(2)_4, yx^6,x^7\rangle=\langle  R_2\cdot \mathcal{C}_A(2)_5,x^7\rangle.
\end{align*}
Note that $\mathcal{C}_A(2)$ is an ideal of $R$, so we include in it the generators of the ideal defining $A^\ast$.  We have  $\mathcal{C}_A(2)=(z,\,y^2,\,yx^4,\,x^7)$. Now
let $g=X^{[6]}+X^{[3]}Y^{[2]}+Y^{[4]}$. We have $\Ann g=(z,\,y^2-x^3,\,
xy^3)$, with associated graded ideal $(\Ann g)^*=(z,\, y^2,\, x^4y,\,x^7)$.
The Hilbert function $H_g=H(R/\Ann g)=(1,2,2,2,2,1,1)$, and the decomposition is
\begin{equation*}
\mathcal{D}_g=\bigl(H_g(0)=(1,1,1,1,1,1,1),\, H_g(1)=(0,1,1,1,1,0)\bigr).
\end{equation*}
Now according to \eqref{modlemi} in Lemma \ref{modificationlem}, $\Ann
g$ is a $2$-modification of $\Ann f$, which shows that the union
mentioned in \eqref{modlemiii} can in this particular case be obtained
from a single AG $a$-modification.
\end{example}

\begin{remark}
\label{modificationlemrmk}
The previous example shows that in \eqref{modlemiii} of Lemma
\ref{modificationlem} we cannot take the union of only the relatively
compressed $a$-modifications (see Definition \ref{compressdef} below).
We here need all $a$-modifications since if $J$ is a relatively
compressed $2$-modification of $\Ann f$, then $z\notin J^\ast$ (in fact $J^\ast$ can contain no element of order less than three) because the
maximum $H(R/J)(2)=M(2,\mathcal{D}_{<2})=(0,1,4,1,0)$ and
$H(R/J)=(1,3,6,3,2,1,1)$  (see Proposition \ref{maxprop} below).
\end{remark}

\index{dual generator of A@dual generator of $A$!and determining $Q(a)$}%
\index{CaR ideal@$\mathcal{C}_A(a)$ ideal of $R$}%
\begin{corollary}
\label{partialdecompcor} 
Let $f$ be a dual generator of a socle-degree-$j$ AG algebra $A$ and consider $A'=R/\Ann f_{\ge j-a}$ where $f_{\ge j-a}=f_j+\cdots +f_{j-a}$. Then there is equality between the partial decompositions $\mathcal{D}_{\le a}(A)=\mathcal{D}_{\le a}(A')$.  These depend only on  $f\mod\mathfrak{D}_{\le j-a-1}$ up to unit in R action as in Lemma \ref{modificationlem}\ref{modlemi}. Also, $f_{\ge j-a}$ is an $(a+1)$-modification of $f$.
\end{corollary}
\begin{proof}  
Evidently $(\Ann f )\cap \maxR^{\,j-a}=(\Ann f')\cap \maxR^{\,j-a}$ so by Definition \ref{amoddef} $f$ is an $(a+1)$-modification of $g$. By Lemma \ref{modificationlem}\ref{modlemii}ii. $\dim_\kk Q_A(k)=\dim_\kk Q_{A'}(k)$ for $k\le a$, hence  $\mathcal{D}_{\le a}(A)=\mathcal{D}_{\le a}(A')$. 
\end{proof}\par
We have the following immediate consequence of this Corollary.
\index{subquotient $Q_A(a)$!Hilbert function $H_A(a)$ determined by $f_{\ge j-a}$, Principle \ref{principlefj-a}}%
\begin{principle}
\label{principlefj-a}
Let $a\ge 1$.  The term $f_{j-a}$ of a dual generator $f$ of a socle-degree-$j$ AG algebra $A$ can only influence  $H_A(a), H_A(a+1), \ldots $, and cannot influence $H_A(0),\ldots, H_A(a-1)$. That is, \emph{the sequence $H_A(a)$ is determined by the isomorphism class, up to differential unit multiple, of $f_{\ge j-a}=f_j+\cdots +f_{j-a},$ where $f=\sum_i  f_{j-i}$ is the dual generator of $A$}.
\end{principle}
It is important to note that for $a\ge 1$ a change to $f_{j-a}$ may change not only $H(a)$ but also $H(a+1), H(a+2),\ldots ,H(j-2)$.  See Sections \ref{cautionsec} and \ref{Hasec} for examples. Also since $Q_A(a)$ is an $A^\ast$ module, and $A^\ast$ may depend on all of the dual generator $f$, we can't quite say that $Q_A(a)$ is determined by $f_{\ge j-a}$. Rather, the  $A^\ast/C(a+1)$ module $Q_A(a)\cong  C_A(a)/C_A(a+1)$ is determined by $f_{\ge j-a}$.

\index{subquotient $Q_A(a)$!determined by $f_{\ge j-a}$}%
\subsubsection{Relatively compressed $a$-modification}

Fix $R={\kk}\{x_1,\ldots , x_r\}$.  J. Emsalem and the first author showed that there is an upper bound $M(a, \mathcal{D}_{<a})$ for the $H(a)$ continuation of any partial Hilbert function decomposition $\mathcal{D}_{<a}$, and it actually occurs \cite[Theorem 3.3]{I1}. We restate this result and give some consequences and examples.  We will need to assume that $\kk$ is an infinite field for existence results such as Proposition  \ref{maxprop}(b) for relatively compressed or compressed AG algebras.  \par Let $a\ge 1$ be an integer. Recall that $r_i=\dim_{\kk} R_i$.
Given a partial Hilbert function decomposition $\mathcal{D}_{<a}$ that occurs for some AG algebra $A$, let $\sum \mathcal{D}_{<a}=(h_0,h_1,\ldots )$ and let $M(a,\mathcal{D}_{<a})$ be the following sequence, symmetric about $(j-a)/2$:
\begin{equation}\label{Maeq}
M(a,\mathcal{D}_{<a})_i=\begin{cases} r_i-h_i &\text { for } i\le (j-a)/2\\
M(a,\mathcal{D}_{<a})_{j-a-i}& \text { for } i> (j-a)/2.
\end{cases}
\end{equation}
\index{maximum $H_A(a))$ given $\mathcal{D}_{<a}$}%
\index{M(a,)@$M(a,\mathcal{D}_{<a})$, maximum $H(Q(a))$ given $\mathcal{D}_{<a}$}%
When $A=R/\Ann f$ is understood, we may write $M(a)$ or $M_f(a)$ for
$M(a,\mathcal{D}_{<a})$. 
We have defined $a$-modification in Definition \ref{amoddef}. We now define $a$-RCM. Recall that $r_i=\dim_{\kk}R_i$, and we denote by 
\begin{equation}\label{compeqn}
{c_{i,j}:=\min\left\{ r_i, r_{j-i}\right\}},
\end{equation}
the dimension of the vector space of degree $i$ partials of a generic homogeneous form of degree $j$ in $r$ variables (see \cite[p.80]{IKa} for historical references, and a short proof).\par

\index{a-RCM@$a$-RCM}%
\index{a-modification@$a$-modification!relatively compressed}%
\begin{definition}
\label{compressdef}
Fix a codimension $r$ (i.e.\ fix $R$), a socle degree $j$ and an integer $a\ge 1$.
\begin{enumerate}[(a)]
\item An Artinian Gorenstein algebra $A=R/\Ann f_A$ satisfying $f_A=f+h_{j-a}$ where $\mathcal{D}(A)_{<a}$ is fixed and $H_A(a)=M\bigl(a,\mathcal{D}(A)_{<a}\bigr)$, the maximum possible, is termed a \emph{relatively compressed $a$-modification}\index{relatively compressed a-modification@relatively compressed $a$-modification!see{ $a$-RCM}} ($a$-RCM) of $R/\Ann f$. \par

\item
An Artinian Gorenstein algebra $A$ that has the maximum Hilbert function given the socle degree $j$, so $h_i=c_{i,j}$ is called \emph{compressed Gorenstein}\index{compressed Gorenstein}. 
\end{enumerate}
\end{definition} 
Note that a compressed AG algebra need not be homogeneous (nor isomorphic to
a homogeneous AG algebra): however it has a symmetric Hilbert function and by Lemma \ref{Watanabelem} its associated graded algebra is also a compressed AG algebra.\par
RCM's are studied in \cite[\S 3]{I1}. By Proposition \ref{maxprop} below an $a$-RCM $A'$ of a given AG algebra $A$ always exists and the set of all of them form an irreducible family. We will use them in some examples and in our discussion of connected sums. That the symmetric subquotient $Q(a)$ for an $a$-RCM of socle degree $j$ is generated in degrees no greater than $\lceil (j-a)/2\rceil$ is in contrast to many of the examples we will study later. The following result from \cite{I1} was joint with J. Emsalem.
\begin{proposition}{\rm \cite[Theorem 3.3]{I1}.}
\label{maxprop} 
Assume that $\mathcal{D}_{<a}=\bigl(H_A(0),\ldots ,H_A(a-1)\bigr)$ occurs as a partial Hilbert function decomposition for an Artinian Gorenstein quotient ${A=R/I}$, ${I=\Ann f}$.
\begin{enumerate}[(a)]
\item Then the sequence $M(a,\mathcal{D}_{<a})$ is a termwise upper bound for the difference 
\[ 
H(A)-\sum\mathcal{D}_{<a},
\]
hence also for $H(a)$. 
\item Assume $\kk$ is infinite. The Hilbert function decomposition
$\mathcal{D}_{\le a}=\bigl(\mathcal{D}_{<a}, M(a,\mathcal{D}_{<a})\bigr)$ occurs as the complete Hilbert function decomposition of an AG quotient $R/\Ann(f+h)$ when $h\in \mathcal{D}_{\le j-a}$ satisfies $h_{j-a}$ is general enough or is generic. 
\item For such $h$ the subquotient $Q(a)$ of $A^\ast$ has Hilbert function  $H(a)=M(a,\mathcal{D}_{<a})$. Also, $Q(a)$ is generated in degrees no greater than $\lceil (j-a)/2\rceil$.
\item The algebras of the form $B=R/\Ann(f+h)$, $f$ fixed, $h\in \mathfrak{D}_{j-a}$ having this Hilbert function decomposition $\mathcal{D}_{\le a}$ form an irreducible family.
\end{enumerate}
\end{proposition} 
\begin{proof}[Proof of (c): generation of $Q(a)$]  
(This result is not stated in \cite{I1}, so we show it here). The Hilbert function of an $a$-RCM $F=f+h$ of $f$ (here $ h\in \mathfrak{D}_{\le j-a}$) agrees with that of a compressed algebra of socle degree $j-a$, in degrees $i\le (j-a)/2$.  Let $j'=\lfloor (j-a)/2\rfloor$. We may ignore terms of $f$, $h$ in degrees less than $j-a$ since they cannot contribute to $H(a)$, and $H(u)=0$ for $u>a$; for simplicity we assume  $f_{<{j-a}}=h_{<{j-a}}=0$, and we set $I=\Ann f$, $J=\Ann F$, $A'=R/J $. It follows that 
\begin{align}
(R\circ f)^*_{j'}+I_{j-a-j'}\circ h&=\mathfrak{D}_{j'}. \text { Then }\notag\\
R_i\circ \bigl((R\circ f)^*_{j'}+I_{j-a-j'}\circ h\bigr)&=R_i\circ \mathfrak{D}_{j'}=\mathfrak{D}_{j'-i}.
\end{align}
Thus, in the notation of Lemma~\ref{Qdualdescription}, Definition~\ref{Wuvdef} (see also Lemma \ref{dualCalem}), the dual $Q_{A'}^\vee(a)\subset \mathfrak D$ is generated in degrees at least $j'$, which is equivalent to $Q_{A'}(a)$ being generated in degrees less or equal $j-a-j'=\lceil (j-a)/2\rceil$, as claimed.
\end{proof}
\begin{example}[$2$-RCM]
\label{genericmod1ex} 
(a) Let $f=X^{[5]}+X^{[2]}Y^{[2]}+Z^{[3]}$, and $R=\kk\{x,y,z\}$; then $\Ann f=(y^2-z^3,\,xz,\,yz,\,z^4,\,x^6)$, and we have $H_f=(1,3,3,2,1,1)$ and 
\begin{equation}\label{Dfeq}
\mathcal{D}_f=\bigl(H_f(0)=(1,1,1,1,1,1),\, H_f(1)=(0,1,1,1,0),\, H_f(2)=(0,1,1,0)\bigr).
\end{equation} 
Here $f_{\ge 4}=X^{[5]}+X^{[2]}Y^{[2]}$ satisfies $\Ann(f_{\ge 4})=(z,\,y^2-x^3,\,x^6)$, and $H_{f_{\ge 4}}=(1,2,2,2,1,1).$ The maximum possible Hilbert function for any $F=f_{\ge 4}+f'_3$ is $H_F=H_{f_\ge 4}+(0,1,1,0)=H_{f_{\ge 4}}+H_f(2)$. Thus,   $H_f(2)=M(2,\mathcal{D}_{<2})$, so $f$ itself is a $2$-RCM of $f_{\ge 4}$.\vskip 0.2cm
\par\noindent
(b) The RCM need not have a different Hilbert function decomposition. Let $f'= X^{[5]} + X^{[3]}Z + X^{[2]}Y^{[2]} + XY^{[3]}+Y^{[4]} + Z^{[3]}$.\footnote{The Example 4.5 of \cite{I1} is the same AG algebra, but because of a misprint, the dual generator there is missing the term $XY^3.$} Then $H_{f'}=(1,3,3,2,1,1)$, the ideal $\Ann f'=(yz,\,xy-y^2+xz,\,xz+z^2-x^3)$ and $\mathcal{D}_{f'}=\mathcal{D}_f$, which is the same as that for $(f')_{\ge 4}$. So $f'$ is a (somewhat trivial) RCM of $(f')_{\ge 4}$ (this fact is related to $X^{[3]}Z$ being an exotic summand of $(f^\prime)_{\ge 4}$ - see Example \ref{blindexotics}b).\vskip 0.2cm\par\noindent
(c) By removing the $XY^{[3]}$ term from $f'$ we get $g = X^{[5]} + X^{[3]}Z + X^{[2]}Y^{[2]} + Y^{[4]} + Z^{[3]}$, determining the ideal $J=\Ann g=(yz,\,xz^2,\,xz+z^2-x^3,\,xy^2-x^2z,\,x^2y-y^3)$, with Hilbert function $H(R/J)=(1,3,4,2,1,1)$ and Hilbert function decomposition 
\[\mathcal{D}_g=\bigl(H_g(0)=(1,1,1,1,1,1),\ H_g(1)=(0,1,2,1,0),\ H_g(2)=(0,1,1,0)\bigr),\] 
where $H_g(2)=M\bigl(2, (\mathcal{D}_g)_{<2}\bigr)$. Thus, removing the $XY^{[3]}$ term from $f'$ leads to a \emph{larger} component $H_g(1)>H_{f^\prime}(1)$ of $\mathcal{D}_g$ (see cautionary Example \ref{caution3ex}). Here, also, $g$ and $g_{\ge 4}$ have the same Hilbert function decomposition, since $h=z-x^2+y^2-x^3$ satisfies ${-h\circ (g_{\ge 4})=XZ+Z}$ and ${-xh\circ ( g_{\ge 4})=Z}$ yielding $Q^\vee_{g\ge 4}(2)=\langle Z,XZ\rangle$ (for the dual $Q^\vee(a)$ of $Q(a)$, see Section~\ref{constructsec}); thus,
$g$ is an RCM $2$-modification of $g_{\ge 4}$. 
\end{example}
\index{a-RCM@$a$-RCM}%
For numerical reasons, the RCM's in codimension $r\ge 3$ may have a symmetric summand $Q(a)$ that is not generated in a single degree.  
\begin{example}[$2$-RCM with $Q(2)$ generated in degrees $3$, $4$]
\label{Qamax2ex}
Let $R={\kk}\{x,y,z\}$ and consider a graded Artinian Gorenstein algebra $A(0)=R/\Ann f$ of socle degree $11$ and Hilbert function $H(A)=(1,3,6,9,11,13,13,11,9,6,3,1)$. Such $A$ exist by the HF criterion (\cite{BuEi, D,IKa}). We assume $H(1)=0$. Then by \eqref{Maeq} and Proposition \ref{maxprop} the maximum Hilbert function $H_{A'}(2)=M(2,\mathcal{D}_{<2})$ for a $2$-RCM $A'=R/\Ann (f+h)$, $h\in \mathfrak{D}_{9}$ is $H_{A'}(2)=(0,0,0,1,4,4,1,0,0,0)$; then $H_{A'}=(1,3,6,10,15,17,14,11,9,6,3,1)$. However, $Q(2)$ is generated in degrees $3$ and $4$, as $H_{A'}(2)_3=1$ so $\dim_{\kk}R_1\cdot Q(2)_3\le 3$, but $\dim_{\kk}Q(2)_4=4$.
\end{example}
\index{curvilinear algebra}%
\index{a-RCM@$a$-RCM!from a curvilinear algebra}%
\begin{example}[$a$-RCM's from a curvilinear $A$]
\label{a-RCMex} 
Begin with $f_5=X^{[5]}$, of Hilbert function  $H_{f_5}=H(0)=(1,1,1,1,1,1)$ (we call this Hilbert function ``curvilinear'') and assume $R=\kk\{x,y,z,w\}$. Then a 1-RCM $f=f_5+f_4$, with $f_4$ generic in $\mathfrak{E}=R^\vee$ satisfies  $H(1)=(0,3,9,3,0)$ so $\mathcal{D}_f=\bigl(H_f(0), H_f(1)\bigr)$ with $H_f=(1,4,10,4,1,1)$. A $2$-RCM $X^{[5]}+f_3$ with $f_3$ generic would have $H(1)$ zero, and $H(2)=(0,3,3,0)$ so Hilbert function $(1,4,4,1,1,1)$ and a $3$-RCM has $H(1)=H(2)=(0,\ldots, 0)$ and $H(3)=(0,3,0)$ so Hilbert function $(1,4,1,1,1,1)$.\par
If we begin instead with  $F_{\ge 4}=X^{[5]}+XY^{[2]}Z$ from Example \ref{symdecompex}, where 
\begin{equation*}
\mathcal{D}_{F_{\ge 4}}=\bigl(H_{F_{\ge 4}}(0)=(1,1,1,1,1,1),\, H_{F_{\ge 4}}(1)=(0,2,4,2,0)\bigr),
\end{equation*}
then a $2$-RCM defined by $F_{\ge 4}+f_3$ with $f_3$ generic (here $f_3=W^{[3]}$ would suffice) has $H(2)=(0,1,1,0)$, so Hilbert function $(1,4,6,3,1,1)$, and $F=F_{\ge 4}+W^{[2]}$ is itself a $3$-RCM of $F_{\ge 4}$.
\end{example}
\index{a-RCM@$a$-RCM}%
\subsection{Cautionary examples.}
\label{cautionsec}
There are some subtleties to the study of the associated graded algebra of non-homogeneous AG algebras that may not be initially apparent. We give several cautionary examples.\footnote{Several initial arXiv postings in recent years have included -- as Lemmas -- incorrect statements that are contradicted by (i)-(iv).}\par
\index{associated graded algebra!subtleties}%
Recall from Principle \ref{principlefj-a} that for $f $ the dual generator of $A=R/\Ann f$, the Hilbert function $H_A(a)=H\bigl(Q_A(a)\bigr)$ is determined by $f_{\ge j-a}$.
\subsubsection{Cautions:}  
\begin{enumerate}[(a)]
\item Let ${f\in\mathfrak{D}}$ be a polynomial of degree $j$, set ${f'=f_{\ge j-a}=f_j+\cdots  +f_{j-a}}$, let ${A=R/\Ann f}$ and ${A'=R/\Ann f'}$. If ${a\ge 1}$ and ${Q_{f'}(k)\neq0}$ for some $k>a$, \textbf{then $H(A)$ may be smaller termwise than $H(A')$}, although $f$ may have more non-zero terms than $f'$.  (See Examples~\ref{cautionex1}, \ref{cautionex2}, \ref{caution3ex} where $a=1$.) \footnote{We have other examples with higher $a$, and some with no exotic terms in $f$.}
\item  In general the space of leading forms ${(R\circ f_{\ge j-a})_{k}^\ast=\langle \lt (g)\mid g\in (R\circ f_{\ge j-a})_{\le k}\rangle\cap \mathfrak{D}_k}$ \textbf{is not contained in} ${(R\circ f)_{k}^\ast=\langle \lt (g)\mid g\in (R\circ f)_{\le k}\rangle\cap \mathfrak{D}_k}$ (equivalent to (i.)).
\item  And $(\Ann f_{\ge j-a})^*_k$ \textbf{need not contain} $(\Ann f)^*_k$ (same examples as for (i.)). 
\item  We have for $f=f_j+f_{j-1}+\cdots f_{j-a}+\cdots$, the termwise inequality
\begin{equation}\label{HFineq} 
H_{f_j+f_{j-1}+\cdots f_{j-a}}\ge  H_f(0)+\cdots + H_f(a),
\end{equation}
\textbf{but there need not be equality} (Examples \ref{cautionex1}, \ref{cautionex2}).
\item The Hilbert function $H(A)$ \textbf{is not termwise semicontinuous} under deformation of the Artinian algebra $A$ (see \cite{Bri,I8} in codimension two). However, each partial sum $\sum_{i=0}^{k}H(A)_i$ is semicontinuous \cite[\S 4.1]{I1}.
\item The set of sequences possible for $H(A)$, so for symmetric decompositions $\mathcal{D}(A)$ in certain constructions for $A=R/\Ann f$, $f$ fixed, \textbf{may depend upon the characteristic} (Example \ref{chardependex}). We don't know if the set of possible Gorenstein sequences of fixed embedding dimension and socle degree, or the set of possible decompositions of a given Gorenstein sequence might themselves depend upon the characteristic (Question \ref{Frobenius2ques}).
\item  The AG algebra $A=R/\Ann f$ determined by $f$ is evidently the same as that determined by ${\sf u}\circ f$ for any unit ${\sf u} \in R$. Thus, the condition that $f={\sf u}\circ f_{>u}$ for a differential unit is natural, and replaces the notion $f_{\le u}=0$, which is not natural.  Also, the condition ``$f_u$ is general enough'' or ``generic''' is natural in the same sense of  being invariant under any map $f\to {\sf u}\circ f$ for $\sf u$ a unit of $R$.
\item Assume $a\ge 1$ and $A=R/\Ann f$. It may not be possible to attain a certain given partial symmetric decomposition $\mathcal{D}_{\le a}(A)$ of the Hilbert function as the complete symmetric decomposition $\mathcal{D}(B)$ for an algebra $B=R/\Ann G$ using a dual generator $G=G_{\ge j-a}=G_j+G_{j-1}+\cdots +G_{j-a}$, no matter the choice of $G$. We call this partial non-ubiquity of $\mathcal{D}_{\le a}(A)$ (see Example \ref{nonubiq4ex} and Theorem \ref{nonubiq4thm} in Section \ref{nonubiquitycod4sec}). 
\end{enumerate}
\begin{proof}[Proof of \eqref{HFineq}] 
By Lemma \ref{modificationlem}\ref{modlemiii} applied to ${a+1}$, we have  ${\mathcal{C}(a+1)=\bigcup (\Ann g)^\ast}$, $g$ an ${(a+1)}$-modification
of $f$, hence ${\mathcal{C}(a+1)\supset (\Ann f_{\ge j-a}})^\ast$, ${f_{\ge j-a}=f_j+\cdots +f_{j-a}}$ which is itself an ${(a+1)}$-modification of $f$. Hence ${H_{f_{\ge j-a}}\ge H\bigl(R/C(a+1)\bigr)=H_f(0)+\cdots +H_f(a)}$.
\end{proof}\par\noindent
Of course, when $Q_{f_{\ge j-a}}(k)=0$ for all integers $k>a$ -- as when $M_f(a+1)=0$ -- there is equality in Equation \eqref{HFineq}.

\begin{example}[$H(R/\Ann f_{\ge 3})>H(R/\Ann f)$]
\label{cautionex1} 
(c.f.\ Example 7 in \cite{BJMR}).
Let ${R={\kk}\{x,y\}}$, ${f= X^{[4]}+X^{[2]}Y+Y^{[2]}}$ in ${\mathfrak{D}={\kk}_{DP}[X,Y]}$, ${A_f=R/\Ann f}$ and take $a=1$ so $f'=f_{\ge 3}=X^{[4]}+X^{[2]}Y$. Then $H(A_f)=H_{f}(0)=(1,1,1,1,1)$. But $ {A'=R/\Ann f'}$ has  $H(A_{f'})=(1,2,1,1,1)$, with
$H_{f'}(2)=(0,1,0)$. We also have $(R\circ f')^*_1=\langle X,Y\rangle$ but $(R\circ f)^*_1=\langle X\rangle$:  so  $\lt (R\circ f)\subsetneq  \lt (R\circ f')$.\par
Here ${\Ann f=(y-x^2,x^5)}$, ${\Ann f'=(y^2,xy-x^3)}$ and $X^{[2]}Y$ is an ``exotic summand'' of $f$ (Section \ref{exoticsumsec}). 
\index{Hilbert function!decreased by adding a term to dual generator}
\end{example}
In the next example, codimension is preserved.
\begin{example}[$H(R/\Ann f_{\ge 5})>H(R/\Ann f)$]
\label{cautionex2}
Let  $f= X^{[6]} + X^{[4]}Y + X^{[2]}Y^{[2]}$ take $ a=1$ so $ f'=f_{\ge 5}$.  Using contraction, we see that the partials of $f$ and $f'$ are
\begin{equation}
\begin{aligned}
1\circ f &= X^{[6]} + X^{[4]}Y + X^{[2]}Y^{[2]} &&& 1\circ f' &= f_{\ge 5}=X^{[6]} + X^{[4]}Y\\
x\circ f &=X^{[5]} + X^{[3]}Y + XY^{[2]} &&& x\circ f' &= X^{[5]} + X^{[3]}Y\\
x^2\circ f &=X^{[4]} + X^{[2]}Y + Y^{[2]} &&& x^2\circ f' &= X^{[4]} + X^{[2]}Y \\
x^3\circ f &=X^{[3]} + XY &&& x^3\circ f' &= X^{[3]} + XY, \  (x^2-y)\circ f' = X^{[2]}Y\\
x^4\circ f &=X^{[2]} + Y,\  (x^2-y)\circ f = Y^{[2]} &&& x^4\circ f' &= X^{[2]} + Y, \ (x^2-y)x\circ f' = XY\\
x^5\circ f &=X, \  (x^2-y)y\circ f = Y &&& x^5\circ f' &= X, \ (x^2-y)x^2\circ f' = Y\\
x^6\circ f &=1 &&& x^6\circ f' &=1
\end{aligned} 
\end{equation}
Therefore, ${(R\circ f_{\ge 5})^*_3}$ is generated by $X^{[3]}$ and $X^{[2]}Y$, and ${(R\circ f_{\ge 5})^*_2}$ is generated by $X^{[2]}$ and $XY$.  So they are not contained in ${(R\circ f)^*_3}$ and ${(R\circ f)^*_2}$, respectively. Here $H_f=(1,2,2,1,1,1,1)$, $H_{f}(3)=(0,1,1,0)$, and  ${\Ann f=(y^3,\, xy-x^3)}$.
For $f'$ we have $H_{f'}=(1,2,2,2,1,1,1)$, $H_{f'}(2)=(0,1,1,1,0)$, and $\Ann f'=(y^2, \, yx^3-x^5).$  This example also appears in a different role in \cite{BJMR}.
\end{example}
\begin{example}
\label{caution3ex} 
In a somewhat different direction, taking $g= X^{[5]} + X^{[3]}Z + X^{[2]}Y^{[2]} +Y^{[4]} + Z^{[3]}$ and $f^\prime=g+ XY^{[3]}$ from Example \ref{genericmod1ex},  we have $H_{{f^\prime}}=(1,3,3,2,1,1)$ with $H_{{f^\prime}}(1)=(0,1,1,1,0)$, but $H_g=(1,3,4,2,1,1)$ with $H_g(1)=(0,1,2,1,0)$ while $H_{{f^\prime}}(2)=H_g(2)=(0,1,1,0)$. So adding the term $XY^{[3]}$ reduces the Hilbert function from that of $H_g$ to that of $H_{f^\prime}$!
\end{example}
\subsection{Constructing an AG algebra from a dual generator.}\label{constructsec}
We next study the relation between $A=R/\Ann f$ and the dual $A^\vee=\Hom(A,{\kk})\cong R\circ f$ and, in particular, the construction of the ideal $C_A(a)$ of $A^\ast$, its dual $C_A^\vee(a)\subset \bigl(\Hom(A,{\kk})\bigr)^\ast$, and an embodiment of the dual, $C^\vee_A(a)_{\mathfrak{D}}$, in $\mathfrak{D}$; and the construction from these of the symmetric suquotient $Q_A(a)$,  its dual $Q_A^\vee(a)$ and an embodiment  $Q^\vee_A(a)_{\mathfrak{D}}$. We then give examples of constructing an Artinian Gorenstein algebra having expected symmetric decompositions by choosing suitable dual generators $f\in \mathfrak{D}$ (Examples \ref{stdex} and \ref{varydualex}).\par
By definition of $Q_A(a)\cong C_A(a)/C_A(a+1)$ where $C_A(a)$ is an ideal of $A^\ast$, we have
\begin{equation}\label{Qveedef}
Q_A^\vee (a)=\bigl(C_A(a)/C_A(a+1)\bigr)^\vee\cong \Hom_{A^\ast} \bigl(C_A(a)/C_A(a+1),{\kk}\bigr).
\end{equation}
\index{Q(a)dual@$Q_A^\vee (a)$, dual to $Q_A(a)$}%
A main goal of this section is to determine the avatar $Q^\vee_A(a)_{\mathfrak{D}}$ of $Q_A^\vee(a)$ (Lemma~\ref{Qdualdescription}, Definition~\ref{Wuvdef}, Example~\ref{qdualex}), which we will use to construct further examples. We first construct isomorphic copies of $C_A(a)$ and of $Q_A(a)$ based on the $A$-module $R\circ f$, where $A=R/\Ann f$.\par
Following \cite[Section~2]{BJMR}, consider the isomorphism 
\begin{equation}
\label{tauisomorphism}
\iota: R/\Ann f\to R\circ f, \quad \varphi\mapsto \varphi\circ f,
\end{equation}
of $R$-modules and {${\kk}$-vector} spaces. Note that ${\iota(\maxA^{\, s})=\maxA^{\,s}\circ f}$ and ${\iota\bigl((0:\maxA^{\,t})\bigr)=(R\circ f)_{\le t-1}}$, so ${\iota\bigl(\maxA^{\,s}\cap(0:\maxA^{\,t})\bigr)=(\maxA^{\,s}\circ f)_{\le t-1}}$. Therefore the inclusion of ideals 
\[
\maxA^{\,i+1}\cap(0:\maxA^{\,j+1-a-i})\subseteq
\maxA^{\,i}\cap(0:\maxA^{\,j+1-a-i})
\]
corresponds via $\iota$ to the inclusion of vector spaces
\begin{equation*}
(\maxA^{\,i+1}\circ f)_{\le j-a-i}\subseteq(\maxA^{\,i}\circ f)_{\le j-a-i},
\end{equation*}
\index{Ca ideal@$C_A(a)$ ideal of $A^\ast$}%
and we see that $\iota$ induces an isomorphism of the vector space $C_A(a)_i$ of Equation \eqref{Caieqn}
\begin{equation}
\label{Cafeq}
C_A(a)_i=\varrho\left(\frac{\maxA^{\,i}\cap(0:\maxA^{\,j+1-a-i})}
{\maxA^{\,i+1}\cap(0:\maxA^{\,j+1-a-i})}\right)\overset{\,\iota\,}{\cong}
\frac{(\maxA^{\,i}\circ f)_{\le j-a-i}}{(\maxA^{\,i+1}\circ f)_{\le j-a-i}}.
\end{equation}
Likewise, $\iota$ induces the isomorphism
\begin{equation}\label{Caf1eq}
C_A(a+1)_i\overset{\,\iota\,}{\cong}\frac{(\maxA^{\,i}\circ f)_{\le j-a-1-i}}{(\maxA^{\,i+1}\circ f)_{\le j-a-1-i}}.
\end{equation}
Since a partial $g$ that belongs to the intersection ${(\maxA^{\,i}\circ f)_{\le j-a-1-i}\cap(\maxA^{\,i+1}\circ f)_{\le j-a-i}}$ must lie in ${(\maxA^{\,i+1}\circ f)_{\le j-a-1-i}}$ we have 
\begin{equation}\label{Ca+1eq}
C_A(a+1)_i\overset{\,\iota\,}{\cong}\frac{(\maxA^{\,i}\circ f)_{\le j-a-1-i}+(\maxA^{\,i+1}\circ f)_{\le j-a-i}}{(\maxA^{\,i+1}\circ f)_{\le j-a-i}}.
\end{equation}
Recalling that $Q_A(a)=C_A(a)/C_A(a+1)$, so applying $\iota$ to Equation~\eqref{Qaeqn} and recalling also that ${\iota\bigl(\maxA^{\,s}\cap(0:\maxA^{\,t})\bigr)=(\maxA^{\,s}\circ f)_{\le t-1}}$, we have the first Equation \eqref{Qaeq} of the following Lemma.
The exact pairing
$Q_A(a)_i\times Q_A(a)_{ j-a-i}\to \kk$ of Equation \eqref{reflexiveeq}, gives an isomorphism $Q_A^\vee(a)_i\overset{\gamma}{\cong} Q_A(a)_{j-a-i}$.  Applying $\gamma$ and substituting $j-a-i$ for $i$ in Equation~\eqref{Qaeq}, we have 
the second equation \eqref{Qadualeq} of the Lemma.
\index{Q(a)decomposition@$Q(a)$ decomposition}%
\index{Q(a)dual@$Q_A^\vee (a)$, dual to $Q_A(a)$}%
\index{subquotient $Q_A(a)$}%
\begin{lemma} 
\label{Qdualdescription}
We have the following vector space isomorphisms:
\begin{equation}\label{Qaeq}
Q_A(a)_i\overset{\,\iota\,}{\cong}\frac{(\maxA^{\,i}\circ f)_{\le j-a-i}}{(\maxA^{\,i}\circ f)_{\le j-a-1-i}+(\maxA^{\,i+1}\circ f)_{\le j-a-i}},
\end{equation}
and
\begin{equation}\label{Qadualeq}
Q_A^\vee(a)_i\overset{\,\iota\circ\gamma\,}{\cong}\frac{(\maxA^{\,j-a-i}\circ f)_{\le i}}{(\maxA^{\,j-a-i}\circ f)_{\le i-1}+(\maxA^{\,j+1-a-i}\circ f)_{\le i}}
\end{equation}
\end{lemma}
\index{C(a)dual@$C^\vee_A(a)$ dual to $C(a)$}%

\begin{definition}
\label{Wuvdef} 
Let $A=R/I$, $I=\Ann f$, $f\in \mathfrak{D}_{\le j}$, with $f_j\not=0$ be a fixed Artinian Gorenstein quotient of $R$ having socle degree $j_A=j$. Below $\hat{\varrho}$ projects the quotient in Equation \eqref{CAveeDeq} to $(R\circ f)^\ast$, analogously to $\varrho$ in Equation \eqref{Caieqn}. 
\begin{enumerate}[(a)]
\item
We define 
\begin{align}
C^\vee_A(a)_{i,\mathfrak{D}}
&=\hat{\varrho}\left(\frac{(\maxA^{\,i}\circ f)_{\le j-a-i}}{(\maxA^{\,i+1}\circ f)_{\le j-a-i}}\right) \subset (R\circ f)^\ast_i,\label{CAveeDeq}\\
Q_A^\vee(a)_{i,\mathfrak{D}}&=\frac{(\maxA^{\,j-a-i}\circ f)_{\le i}}{(\maxA^{\,j-a-i}\circ f)_{\le i-1}+(\maxA^{\,j+1-a-i}\circ f)_{\le i}}\label{QAveeDeq},
\end{align}
also $C_A^\vee(a)_{\mathfrak{D}}=\oplus_{i=0}^j
C^\vee_A(a)_{i,\mathfrak{D}}$ and $Q_A^\vee(a)_{\mathfrak{D}}=\oplus_{i=0}^j Q_A^\vee(a)_{i,\mathfrak{D}}.$
\item Recall that for a vector subspace $S\subset R_i$ we denote by $S^\perp\subset \mathfrak{D}_i$ the perpendicular space to $S$ in the contraction pairing $R_i\circ \mathfrak{D}_i\to {\kk}$. 
For a finite-dimensional $R$-submodule $M\subset \mathfrak{D}_{\le j}$ we denote by 
\begin{equation}
M^\ast=\oplus_0^jM^\ast_i \text { where }(M^\ast)_i=\langle M\cap \mathfrak{D}_{\le i}\rangle/\langle M\cap \mathfrak{D}_{<i}\rangle,
\end{equation}
and we denote by $H(M)$ the Hilbert function $H(M)=(h_0,h_1,\ldots,h_j)$, $h_i=\dim_{\kk} M^\ast_i$.
\end{enumerate}
\end{definition} 
\begin{lemma}
\label{dualCalem}  
Let $A=R/I$ be Artinian Gorenstein of socle degree $j$, and let $f\in \mathfrak{D}_{\le { j}}$ be a dual generator, so $I=\Ann f$, and suppose $a\in [0,j-1]$. 
\begin{enumerate}[(a)]
\item The subspace $C^\vee_A(a)_{i,\mathfrak{D}}\subset \mathfrak{D}_i$ is independent of the choice of the element $f\in \mathfrak{D}_{\le j}$ defining $A$. We have $\dim_{\kk}\bigl(C^\vee_A(a)_{i,\mathfrak{D}}\bigr)=\dim_{\kk}C_A(a)_i$, and the Hilbert function $H\bigl(C^\vee_A(a)_{\mathfrak{D}}\bigr)=H\bigl(C_A(a)\bigr)$, the Hilbert function of $C_A(a)$ as an ideal of $A^\ast$.
\item\label{Qaduallem} The $A$-module $Q_A^\vee(a)_{\mathfrak{D}}$ satisfies $Q_A^\vee(a)_{\mathfrak{D}}= C_A^\vee(a)_{\mathfrak{D}}/C_A^\vee(a+1)_{\mathfrak{D}}$, and is independent of the choice of $f\in \mathfrak{D}_{\le j}$ defining $A$.
The Hilbert function $H\bigl(Q^\vee_A(a)_{\mathfrak{D}}\bigr)=H\bigl(Q_A(a)\bigr)$.
\end{enumerate}
\end{lemma}
\begin{proof}[Proof of (a)]
If an element ${g\in \mathfrak{D}}$ represents a non-zero class in $C^\vee_A(a)_{i,\mathfrak{D}}$ then $g$ has degree $i$, because it belongs to ${(0:\mathfrak{m}^{i+1})\circ f}$, but not to ${(0:\mathfrak{m}^{i})\circ f}$; and if ${g'\in \mathfrak{D}}$ represents the same element then ${g_i=g'_i}$. So any class $\overline{g}$ in $C^\vee_A(a)_{i,\mathfrak{D}}$ is determined by the top-degree term of $g$, and therefore $C^\vee_A(a)_{i,\mathfrak{D}}\subset \mathfrak{D}_i$ is a homogeneous vector space, for the grading inherited from $\mathfrak{D}$. The vector space $C_A(a)_i\subset A_i$ is independent of the choice of $f$ up to a unit $u\in R$ -- the vector space is the same for $u\circ f$ as it is for $f$.
Evidently, since $C^\vee_A(a)_{i,\mathfrak{D}}$ is a subspace of $\mathfrak{D}_i$ and is a dual vector space to $C_A(a)_i$, we have $\dim_{\kk}\bigl(C^\vee_A(a)_{i,\mathfrak{D}}\bigr)=\dim_{\kk}C_A(a)_i$, implying $H\bigl(C^\vee_A(a)_{\mathfrak{D}}\bigr)=H(C_A(a))$.
\vskip 0.2cm\par\noindent
\textit{Proof of (b)}. This is immediate from the definitions and from part (a). 
\end{proof}\par
\index{Q(a)dual@$Q_A^\vee (a)$, dual to $Q_A(a)$}%
We first give an example to illustrate the need for $\hat{\varrho}$ in Equation \eqref{CAveeDeq}, that is the quotient in the equation is in general not itself in $(R\circ f)^\ast$. It also shows the use of $\varrho$ in Equation~\eqref{Caieqn}.
\begin{example}
\label{usevarrho} 
Let $f=X^{[4]}+X^{[2]}Y$, $A=R/I$, $I=\Ann f=(y^2,\, xy-x^3)$. Then ${(0:\m)=\langle x^4\rangle}$ and $(0:\m^2)=\langle y-x^2,\,x^3,\,x^4\rangle$, $\m^2\cap (0:\m^2)=\langle x^3,x^4\rangle$ so $C_A(2)_1=\varrho\bigl(\m\cap (0:\m^2)/\m^2\cap (0:\m^2)\bigr)$ where the quotient is $\langle y-x^2\rangle$, but $\varrho(y-x^2)=\overline{y}$ in $A^\ast$. Likewise, $C^\vee_A(2)_1=\hat{\varrho}\left(\langle y-x^2\rangle\circ f\right)=\hat{\varrho}\langle -Y\rangle=\langle Y\rangle$, without any need for $\hat{\varrho}$, but 
\[
C^\vee_A(0)_2=\hat{\varrho}\bigl((\m^2\circ f)_{\le 2}/(\m^3\circ f)_{\le 2}\bigr)=
\hat{\varrho}\bigl( \langle X^{[2]}+Y,X,1\rangle/\langle X,1\rangle\bigr)=
\hat{\varrho} (\langle X^{[2]}+Y\rangle)=\langle X^{[2]}\rangle
\]
shows the use of $\hat{\varrho}$ to project to $C^\vee_A(0)_2\subset (R\circ f)^\ast_2$. 
\par
This example also allows us to illustrate that the ideal structure of $C_A(a)\subset A^\ast$ must be defined using the multiplication in the quotients in Equation \eqref{Ceq}, not directly in $A^\ast$. Here, related to $C_A(2)$, note that ${y \circ f = X^{[2]}}$, ${(y-x^2 )\circ f = -Y}$, so ${xy \circ f = X}$, but ${x(y-x^2) \circ f = 0}$; thus, $y=\varrho(y-x^2)\in C_A(2)_1$ so for $x\in A^\ast$ we have $xy=\varrho\bigl(x(y-x^2)\bigr)=\varrho(0)=0$ in $C_A(2)_2$. 
\end{example}
We introduce some notation. We denote by $W(u,v)=W_A(u,v)\subset A$ the intersection
\begin{equation}\label{Weqn}
W(u,v)=W_A(u,v)=\mathfrak{m}^u\cap ( 0:\mathfrak{m}^v),
\end{equation}
an ideal of $A$.  Let $\phi: A\to \kk$ be the homomorphism defined on $A$ taking the socle surjectively to $\kk$ (Lemma \ref{Macduallem}). Evidently, under the pairing $A\times A\to \kk$, $\langle g,h\rangle_\phi=\phi(gh)$ the orthogonal complement of $W(u,v)$ is 
\begin{equation}\label{complemWuveq}
K(u,v)=\mathfrak{m}^v+(0:\mathfrak{m}^u),
\end{equation}
an ideal of $A$.  We denote by $W(u,v)^\perp=K(u,v)\circ f$, and we have
\begin{equation}\label{CAvee2eq}
C^\vee_A(a)_{i,\mathfrak{D}}=\hat{\varrho}\bigl( K(i+1,j+1-a-i)\circ f/K(i,j+1-a-i)\circ f\bigr) \subset \mathfrak{D}_i.
\end{equation}

\begin{example}
\label{qdualex} 
Let $f=X^{[3]}+Y^{[4]}$ and  $R={\kk}\{x,y\}$ then 
\[
A=R/\Ann f=R/(xy,\,x^3-y^4)=\langle \overline{1},\,\overline{x},\,\overline{y},\,\overline{x^2},\,\overline{y^2},\,\overline{y^3},\,\overline{y^4}\rangle \text { as vector space, }
\]
and $Q(0)=R/(x,y^5)=\langle 1,\,y,\,y^2,\,y^3,\,y^4\rangle$, $Q(1)\cong \langle x,\,x^2\rangle$, $H(A)=(1,2,2,1,1)$. Here 
\[
(0:\m)=\langle\overline{y^4}\rangle,\quad (0:\m^2)=\langle\overline{x^2},\overline{y^3},\overline{y^4}\rangle\quad \text { and }\quad (0:\m^3)=\langle\overline{x},\overline{y^2}\rangle+(0:\m^2)=\langle\overline{x},\overline{x^2},\overline{y^2},\overline{y^3},\overline{y^4}\rangle
\] 
So $W(1,3)=(0:\m^3)$, $W(2,3)=\m^2\cap(0:\m^3)=\langle\overline{x^2},\overline{y^2},\overline{y^3},\overline{y^4}\rangle$, $W(1,2)=W(2,2)=(0:\m^2)$. And $W(2,1)=W(3,1)=\langle \overline{y^4}\rangle$, $W(3,2)=\langle\overline{y^3},\overline{y^4}\rangle$.  Now 
\begin{align*}
W(2,3)^\perp&=\bigl(\m^3+ (0:\m^2)\bigr)\circ f=\langle Y,X,1\rangle\\
W(1,3)^\perp&=\bigl(\m^3+( 0:\m)\bigr)\circ f=\langle Y,1\rangle.
\end{align*}
We have $C(a)_i\cong\frac{W(i,5-a-i)}{W(i+1,5-a-i)}$, so
\begin{align*}
C(1)_1&\cong W(1,3)/W(2,3)=\langle \overline{x}\rangle, & 
C(1)_2&\cong W(2,2)/W(3,2)=\langle\overline{x^2}\rangle, \\
C(2)_1&\cong W(1,2)/W(2,2)=0, & C(2)_2&\cong W(2,1)/W(3,1)=0,\\
\intertext{and}
Q(1)_1&=C(1)_1/C(2)_1=\langle\overline{x}\rangle, & 
Q(1)_2&=C(1)_2/C(2)_2=\langle\overline{x^2}\rangle,
\end{align*}
We have $C^\vee(1)_{1,\mathfrak{D}}\subset  (R\circ f)^\ast$ satisfies, by Equation \eqref{CAvee2eq}  
\[
C^\vee(1)_{1,\mathfrak{D}}\cong\frac{(0:\m^2)\circ f+\m^3\circ f}{(0:\m)\circ f+\m^3\circ f}=
  \frac{\langle 1,X,Y\rangle+\langle 1,Y\rangle}{\langle 1\rangle+\langle 1,Y\rangle}\cong
  \langle X\rangle.
\] 
Notice that the $C^\vee(a)_{\mathfrak{D}}$ is naturally by construction in $(R\circ f)^\ast$. By Lemma \ref{dualCalem}(b) we have
\[
Q^\vee(1)_{1,\mathfrak{D}} =C^\vee(1)_{1,\mathfrak{D}}/C^\vee(2)_{1,\mathfrak{D}}=\langle X\rangle.
\]
\end{example}

\begin{remark}  
We have given in Lemma \ref{dualCalem} a way to construct the dual $Q_A^\vee(a)_{i,\mathfrak{D}}$. Henceforth if the algebra $A$ is clear, and we are working in $\mathfrak{D}$ we may omit the subscripts $A$ and $\mathfrak{D}$ and write, simply
$Q^\vee(a)$ for $Q_A^\vee(a)_{\mathfrak{D}}$. In calculating examples from the dual generator $f$ we usually retain the ``tails'' (lower degree portions) of elements that are partials of $f$, because cancellation of higher degree terms in linear combinations yield elements in  $Q^\vee(a)$ for higher $a$:  for example cancellation from elements ostensibly in $Q^\vee(a_1)$, $Q^\vee(a_2)$ may yield an element of $Q^\vee(a_1+a_2)$ in the proof of the key Theorem \ref{DArestrictthm}, where we use Lemma \ref{Qdualdescription}.

\par 
Also, using Lemma \ref{Qdualdescription}, we can check that there is an isomorphism
\begin{equation}\label{bigoeqn}
\bigoplus_{a\ge0}Q^\vee(a)_i \overset{\iota\circ\gamma}{\cong} \frac{(R\circ f)_{\le i}}{(R\circ
f)_{\le i-1}}.
\end{equation}
The dimension  $\dim_{\kk}Q^\vee(a)_i=\dim_{\kk}Q(a)_i$ by Lemma \ref{dualCalem}. We have $Q^\vee(a)_i\cong Q(a)_{j-a-i}$, which is isomorphic as a vector space to $Q(a)_i$. By Theorem \ref{mainoldthm} the dimension of the left-hand side of Equation \eqref{bigoeqn} gives $H(A)_i$. Thus, we can compute the Hilbert function of $A$ by taking
$H(R\circ f)=H(A^\vee)$ as in Definition \ref{Wuvdef}(b): for this we consider the leading (highest degree) term of a partial ${g\in(R\circ f)_{\le i}}$ (Definition \ref{ordernotation}(c)) rather than the order of an element of $R$.
\end{remark}

\subsubsection{Power sum dual generator.}
We now apply the results about duality to constructing Artinian Gorenstein algebras with an expected symmetric decomposition.  For the following example we let $R={\kk}\{x,y\}$, $S={\kk}\{x,y,z\}$, $\mathfrak{D}={\kk}_{DP}[X,Y]$, $\mathfrak{E}={\kk}_{DP}[X,Y,Z]$. For ${U\subseteq R}$ and ${V\subseteq \mathfrak{D}}$ we denote by $U\circ V\subset \mathfrak{D}$ the set $U\circ V=\{u\circ v
\mid u\in U, v\in V\}$ and by $\langle U\circ V\rangle$ the span of $U\circ V$. 
\begin{example}[\textsc{Power sum dual generator}]
\label{stdex}
Let $f=X^{[7]}+Y^{[5]}+(X+Y)^{[5]}+Z^{[4]}\in \mathfrak{E}$, let ${I=\Ann f}\subset S$, $ A=S/I$. Then the ideal $I=(xz,yz,x^2y-xy^2,y^4-2x^3y+x^6,z^4-x^7) $ and $\mathcal{D}(A)=\bigl(H(0),H(1),H(2),H(3)\bigr)$, with
\begin{equation}\label{Deq} 
\begin{tabular}{c|cccccccc|}
$H(0)=$&1&1&1&1&1&1&1&1\\
$H(1)=$&0 &0&\multicolumn{4}{c}{$\cdots$}&0&\\
$H(2)=$&0&1&2&2&1&0&&\\
$H(3)=$&0&1&1&1&0&&&\\
\hline
$H(A)=$&1&3&4&4&2&1&1&1
\end{tabular}
\end{equation}
\index{dual generator of A@dual generator of $A$!power sum}%
In the following table, whose columns are $\{Q(a)\circ f\}$ (meaning we consider the vector spaces $Q(a)_i\circ f$), we show in each entry the basis vectors of $\bigl(\maxA^{i}\cap(0:\maxA^{8-a-i})\bigr)\circ f$ after we mod out by $\bigl(\maxA^{i+1}\cap(0:\maxA^{8-a-i})\bigr)\circ f$ and by $\bigl(\maxA^{i}\cap(0:\maxA^{7-a-i})\bigr)\circ f$: 
\begin{equation}
\label{Ddualeq} 
\begin{array}{c|cccc|}
\text{degree in }\mathfrak{D}&Q(0)\circ f&Q(1)\circ f&Q(2)\circ f&Q(3)\circ f\\\hline
7&f&0&0&0\\
6&X^{[6]}+(X+Y)^{[4]}&0&0&0\\
5&X^{[5]}+(X+Y)^{[3]}&0&0&0\\
4&X^{[4]}+(X+Y)^{[2]}&0&Y^{[4]}+(X+Y)^{[4]}&0\\
3&X^{[3]}+(X+Y)&0&Y^{[3]},(X+Y)^{[3]}&Z^{[3]}\\
2&X^{[2]}+1&0&Y^{[2]},(X+Y)^{[2]}&Z^{[2]}\\
1&X&0&Y&Z\\
0&1&0&0&0
\end{array}
\end{equation}
Here $Y^{[3]}=(y^2-xy)\circ f$ and $(X+Y)^{[3]}=xy\circ f$ are both in 
${\bigl(\maxA^{\,2}\cap (0:\maxA^{\,4})\bigr)\circ f}$.  They are partials of order two (Definition \ref{ordernotation}e) of $f$ hence we expect them to have degree $5$; however, their degrees are three, and two less than would be expected for $\maxA^{\,2}\circ f_7$, thus they are in $Q(2)_2\circ f$. Note, we may reduce the $Q(0)\circ f$ entries by those occuring later, so we can write $X^{[4]}$, $X^{[3]}$, $X^{[2]}$, respectively, in place of $X^{[4]}+(X+Y)^{[2]}$, $X^{[3]}+(X+Y)$, $X^{[2]}+1$, respectively, in the $Q(0)\circ f$ column of \eqref{Ddualeq}. The stratification of $(A^\ast)^\vee=\bigl(\Gr_{{\maxA}}(A)\bigr)^\vee$ corresponding to the $Q(a)$ stratification of ${A}^\ast$ is obtained by taking leading terms in  
\eqref{Ddualeq}. We have $Q^\vee(1)=0$ and, by definition,
\begin{align}
Q^\vee(0)&=\langle X^{[7]},\,X^{[6]},\,X^{[5]},\,X^{[4]},\,X^{[3]},\,X^{[2]},\,X,\,1\rangle =S\circ X^{[7]},\, \notag\\
Q^\vee(2)&=\langle Y^{[4]}+(X+Y)^{[4]},\,Y^{[3]},\,(X+Y)^{[3]},\,Y^{[2]},\,(X+Y)^{[2]},\,Y\rangle =S\circ  \bigl(Y^{[4]}+(X+Y)^{[4]}\bigr),\notag\\
Q^\vee(3)&=\langle Z^{[3]},\,Z^{[2]},\,Z\rangle=S\circ Z^{[3]}.
\label{dualQ(a)moduleeq}
\end{align}
Each $Q(a)$ is a homogeneous $S$-module (always), and here each $Q(a)$ is a cyclic $S$-module (required in 2 variables, but not in 3 variables, as we shall see below in Section \ref{Hasec}).  From \eqref{dualQ(a)moduleeq} we can read off the Hilbert function decomposition $\mathcal{D}(A^\vee)=\mathcal{D}(A)$ in \eqref{Deq}.\par
We also have
\begin{align*} 
\mathcal{C}(0)&=(y,\,z,\,x^8),\ \mathcal{C}(1)=\mathcal{C}(2)= (z,\,x^2y-xy^2,\,y^4-2x^3y,\,x^8),\\
\mathcal{C}(3)&=I^\ast=( xz,\,yz,\,x^2y-xy^2,\,y^4-2x^3y,\,z^4,\,x^8).
\end{align*}
The representatives of $Q(a)$ are
\begin{align*} 
Q(0)&=\langle 1,\,x,\,x^2,\,x^3,\,x^4,\,x^5,\,x^6,\,x^7\rangle\\
Q(2)&=\langle y,\,xy,\,y^2,\,xy^2,\,y^3,\,y^4\rangle\\
Q(3)&=\langle z,\,z^2,\,z^3\rangle .
\end{align*} 
The action of $Q(a)_i$ on $f$ according to Lemma \ref{Qdualdescription}, Equation  \eqref{Qadualeq} gives vector spaces isomorphic to  $Q^\vee(a)_i$: for example, denoting by $\overline{h}$  the class of $h$ in its respective quotient,
\begin{align*}
Q^\vee(0)_4&\cong{\frac{ (\maxA^{\,3}\circ f)_{\le 4}}{(\maxA^{\,3}\circ f)_{\le 3}+(\maxA^{\,4}\circ f)_{\le 4}}}=\big\langle\overline{X^{[4]}+(X+Y)^{[2]}}\big\rangle\\
Q^\vee(2)_4&\cong {\frac{ (\maxA\circ f)_{\le 4}}{(\maxA\circ f)_{\le 3}+(\maxA^{\,2}\circ f)_{\le 4}}}=\big\langle\overline{Y^{[4]}+(X+Y)^{[4]}}\big\rangle\\
Q^\vee(2)_3&\cong {\frac{ (\maxA^{\,2}\circ f)_{\le 3}}{(\maxA^{\,2}\circ f)_{\le 2}+(\maxA^{\,3}\circ f)_{\le 3}}}=\big\langle\overline{Y^{[3]}},\overline{(X+Y)^{[3]}} \big\rangle.
\end{align*}
As we can see, the representatives above need not be homogeneous, but, taking their leading terms we do get the $Q^\vee(a)_i$ specified in \eqref{dualQ(a)moduleeq}. 
\end{example}
\index{dual generator of A@dual generator of $A$!use to construct given $\mathcal{D}(A)$}%
\begin{example}[\textsc{Using the dual generator}]
\label{varydualex}
The dual generator can often be used to simply define AG algebras having certain given symmetric decompositions.\par\noindent
In Example \ref{stdex} the algebra $B=R/\Ann g$, where $g=X^{[7]}+Y^{[5]}+(X+Y)^{[5]}$, is a  $2$-RCM of $A_0=R/\Ann (X^{[7]})$ over $R={\kk}\{x,y\}$: it  has the maximum possible $H_B(2)=(0,1,2,2,1,0)$ and Hilbert function $H(B)=(1,2,3,3,2,1,1,1)$, given $H_B(0)$, $H_B(1)=0$ and the codimension two of $R$. However, the AG algebra $A$ defined by $f=g+Z^{[4]}$ is not a $3$-RCM of $B$ in $S$, as $H_A(3)=(0,1,1,1,0)$ which is not  $M\bigl(3,D_{\le 2}(B)\bigr)=(0,1,3,1,0)$. \par
The dual generator ${F=g+Z^{[2]}XY}$ (or take $F=g+w_4$, $w_4$ generic) yields a 3-RCM $A'=S/\Ann F$ of $B$, with $H_{A'}=(1,3,6,4,2,1,1,1)$. \par
Taking instead $F'=g+Z^{[4]}+Z^{[2]}X^{[2]}$ we have $H_{F'}(3)=(0,1,2,1,0)$ and the Hilbert function
$H(S/\Ann F')=(1,3,5,4,2,1,1,1)$, intermediate between those for $f$ and $F$.
\end{example}
\subsection{Constructing AG algebras having $H(a)=(0,s,0,\ldots, 0,s,0)$.}\label{Hasec}
F.H.S. Macaulay  showed that an AG algebra of codimension two
is a complete intersection (CI) \cite{Mac0,Mac1}.  The AG algebra structure theorem for these algebras shows that when ${r=2}$ each $Q(a)$ is a cyclic module isomorphic to a (shifted) graded CI, so $Q(a)=h(a) R/\bigl(g_1(a),g_2(a)\bigr)$ (see \cite[\S 2]{I1}). But in codimension three, even for complete intersections, $Q(a)$ may not be cyclic 
\cite[Example 1.6]{I1}. At the time this was surprising to the first author.  We here introduce a simple process to construct such AG quotients of $S={\kk}\{x,y,z\}$ with 
\begin{equation*}
{H(a)=H\bigl(Q(a)\bigr)= (0,1,0,\ldots ,0, 1, 0)}.
\end{equation*} More generally, we construct quotients of the ring ${S={\kk} \{x_1, \ldots, x_r; z_1,\ldots, z_s\} }$, having $H(a)=(0,s,0,\ldots ,0,s,0)$ where $Q(a)$ is not even generated in degree $1$ (Proposition \ref{noncyclic1prop}).
We will later build on this process to construct ``designed'' AG algebras with $H(a), H(2a), \ldots $ non-zero, having certain patterns.\par
We let $R={\kk}\{x_1,\ldots ,x_r\}$, $S={\kk}\{x_1,\ldots ,x_r;z_1,\ldots ,z_s\}$, with ${r\ge 2}$, ${s\ge 1}$; let $\mathfrak{D}={\kk}_{DP}[X_1,\ldots ,X_r]$ and ${\mathfrak{E}={\kk}_{DP}[X_1,\ldots ,X_r,Z_1,\ldots Z_s]}$, where $\{X_i\}$, $\{Z_i\}$ are variables. The following gives the basic construction that underlies the then mysterious  \cite[Example 1.6]{I1}, and is intended to highlight our method. \par
The key idea in this construction is to start with a polynomial ${f\in \mathfrak{D}}$ of degree $j$ such that $\mathfrak{D}_{\le k-1}\subset{R\circ f}$: this is equivalent to $\Ann f\subset \maxR^{\,k}$. We assume that ${\dim_{\kk} (\Ann f)_k\ge s}$. We choose linearly independent elements $h_1,\ldots ,h_s\in \mathfrak{D}_{\le k}$ whose leading  (top degree) terms are linearly disjoint from $(R\circ f)^\ast_k $. Then all partials of order at least one of each $h_i$ are partials of $f$ as well: therefore ${\maxR}\circ \langle h_1,\ldots ,h_s\rangle \subset R\circ f$. The dualizing module $S\circ F$ generated by ${F=f+h_1Z_1+\cdots +h_sZ_s\in \mathfrak{E}}$, includes ${R\circ f}$ and as well $2s$ linearly disjoint new basis elements,  the polynomials $h_1,\ldots ,h_s$ and the variables $Z_1,\ldots , Z_s$. When $f\in\mathfrak{D}$ is also homogeneous, we show that there are no further non-zero $Q(u)$ other than  $Q(0)$ and $Q(a)$, $a=j-(k+1)$ -- except in one special case.
\begin{proposition}
\label{noncyclic1prop}
Let ${f\in \mathfrak{D}}$ be a divided-power polynomial of degree $j$. Suppose  $\Ann_R f\subset\maxR^{\,k} $ but $\Ann_R f\nsubseteq\maxR^{\,k+1}$, for an integer ${k\ge2}$. Let ${s\ge1}$  be an integer such that ${\left(H_{f}\right)_k\le r_k-s}$ and let ${h_1,\ldots,h_s\in \mathfrak{D}_{\le k}\setminus \mathfrak{D}_{\le k-1}}$ be polynomials of degree $k$ whose degree $k$ top-degree forms are linearly independent and linearly disjoint from $(R\circ f)^\ast_k$: that is, we assume that the intersection ${\langle \lt(h_1),\ldots , \lt(h_s)\rangle \cap (R\circ f)^\ast_k=0}$.
Let ${a=j-(k+1)}$ and consider the element
\begin{equation}
{F}=f+\sum_{i=1}^s h_i\cdot Z_i \in \mathfrak{E},
\end{equation}
and let $A=S/\Ann F$.
\begin{enumerate}[(a)]
\item Then  ${Q^\vee_F(a)\cong\langle Z_1,\ldots, Z_s; h_1,\ldots ,h_s \rangle\oplus Q^\vee_f(a)}$ and ${Q^\vee_F(u)\cong Q^\vee_f(u)}$ for ${1\le u<a}$.
\item  Also $Q_F(a)\cong \langle z_1,\ldots, z_s; \phi_1,\ldots ,\phi_s\rangle $ where $\phi_t\in (\Ann_R f)_k$ and $\phi_i\circ h_t=\delta_{i,t}$. Furthermore, $Q_F(a)$ is neither cyclic, nor generated in degree $1$.
\item Assume further that  $f$ is homogeneous (${f=f_j}$). Then ${Q_F(u)=0}$ for ${u\notin\{0,a,2a\}}$ and the Hilbert function of ${A}$ satisfies
\begin{equation}
H_{F}(a)=(0,s,0,\ldots,0,s,0) \text { where } {H_{F}(a)_1=H_{F}(a)_k=s},
\end{equation}
where there are $k-2$ central zeroes, beginning in degree two.\par
Also ${Q_F(u)=0}$ for ${u\notin\{0,a\}}$, except in the special case ${j=2(k-1)}$ and ${Q_F(0)=R/\Ann_R f}$ is compressed Gorenstein.  Then ${0\le H_{F}(2a)_2\le rs}$ and ${H_{F}(2a)_i=0}$ for ${i\not=2}$.  
\end{enumerate}
\end{proposition}
\begin{proof}
By the assumption of linear disjointness of ${\langle \lt(h_1),\ldots \lt(h_s)\rangle}$ from $({R\circ f})^\ast_k$, there are elements ${\phi_t\in (\Ann_R f)_k}$ such that ${\phi_t \circ h_i=\delta_{it}}$. So ${\phi_i \circ F=Z_i}$, and evidently ${z_i\circ F=h_i\notin R\circ f}$.  Since ${\maxR\circ h_i\subset R\circ f}$ we conclude that ${Q^\vee_F(a)=\langle Z_1,\ldots, Z_s; h_1,\ldots , h_s\rangle\oplus Q^\vee_f(a)}$, ${a=j-(k+1)}$. Evidently, $\mathsf{m}_S\cdot z_i=0$ in $Q_F(a)$ since $z_iz_j\in \Ann F$ and $(x_jz_i)\circ F=x_j\circ h_i$ is in $R\circ f$ by assumpition, so it is in $Q_F(0)$. So $Q_F(a)$ is neither cyclic nor generated in degree 1.\par
We now assume ${f=f_j}$ is homogeneous, and consider another ${Q_A(u)}$, ${u\notin\{0,a\}}$, and assume that ${Q_F(u)_i\neq0}$, for some $i$.
Let ${g\in (R\circ {F})_{\le i}}$ have degree $i$ and represent a non-zero element of ${Q_F(u)_i}$. Then since ${z_i^{\,2}\circ {F}=z_iz_t\circ F=0}$, ${1\le i,t\le s}$,   we may choose ${\varphi\in{\kk}[X,Z]},$
\begin{equation}\label{varphieq}
\varphi=\varphi_0+\sum_{l=1}^s\varphi_l z_l,
\end{equation}
with ${\varphi_0,\varphi_1\ldots , \varphi_s\in R={\kk}[X]}$  such that ${ g=\varphi\circ {F}}$. Then
\begin{equation}\label{varphi2eq}  
g=\varphi\circ {F}=\varphi_0\circ f+\sum_{l=1}^s(\varphi_0\circ h_l) Z_l+
\sum_{l=1}^s\varphi_l \circ h_l.
\end{equation}
\vskip 0.2cm\noindent
Since, evidently, ${u>a}$, we may mod out ${g}$ by ${Q_F^\vee(a)\supset \langle z_1,\ldots z_s\rangle \circ F}$, so we may assume ${\varphi_1,\ldots ,\varphi_s\in \maxR}$. Since ${f=f_j}$ we can have cancellation between the leading terms of ${\varphi_0 \circ f}$ and ${\sum_{l=1}^s\varphi_l \circ h_l}$ only if ${\varphi_0\in \maxR^{\,j-(k-1)}}$: then for ${1\le l\le s}$
\begin{equation*}
\varphi_0\circ h_l \subset \mathfrak{D}_{\le \kappa}, \kappa=k-\bigl(j-(k-1)\bigr)=2k-1-j
\end{equation*}
But we already have ${Z_l \in Q_F^\vee(a)_1}$ so to have ${\varphi_0\circ h_l \in Q_F^\vee(u)}$, with ${u\notin \{0,a\}}$ we must have ${\deg \phi_0\circ h_l \ge 1}$ implying ${2k-1-j\ge 1}$ so ${ j\le 2(k-1)}$. This can happen only if $Q_F(0)$ is compressed Gorenstein of even socle degree ${j=2(k-1)}$. In that case ${Q_F^\vee(u)_2}$ may include ${ (\varphi_0\circ h_l)\cdot Z_l}$, whence  ${u=2a}$. Since ${\varphi_0\circ h_l \subset \langle X_1,\ldots , X_r\rangle}$ we have ${0\le H_F(2a)_2\le rs}$ and ${H_F(2a)_i=0}$ for ${i\ne2}$. \vskip 0.2cm
\end{proof}\par\noindent
The following is an example for Proposition \ref{noncyclic1prop}(c) in the special case where $j$ is even and $f=f_j$ defines a compressed algebra  quotient of ${\kk}\{x,y\}$,  $a=1$ and $H(2)\not=0$.
\begin{example}
\label{1201ex}   
Let $S={\kk}\{x,y,z\}$, and $F=X^{[3]}\cdot Y^{[3]}+Z(X^{[4]}+Y^{[4]})\in \mathfrak{E}$. Then $H_F=(1,3,5,4,4,2,1)$ the ideal $I=\Ann F = (z^2,\, xyz,\, x^2z-xy^3,\, y^2z-x^3y,\, x^4-y^4)$, $I^\ast=(z^2,\, xyz,\, x^2z,\, y^2z,\, x^4-y^4,\,xy^4,\,yx^4)$, and
\begin{align}\label{geneq1}
\mathcal{D}_F&=\bigl(H_F(0)=(1,2,3,4,3,2,1),\ H_F(1)= (0,1,0,0,1,0),\ H_F(2)=(0,0,2,0,0)\bigr)      
\end{align}
Here 
\begin{align*}
Q_F^\vee(1)=Q_F(1)\circ F&=\langle x^4\circ F,z\circ F\rangle=\langle
Z,X^{[4]}+Y^{[4]}\rangle,\\
Q_F^\vee(2)=Q_F(2)\circ F&= \langle (-zx+y^3)\circ F, (-zy+x^3)\circ
F\rangle=\langle XZ,YZ\rangle .
\end{align*}
\end{example}
The next example shows that when the portion $f\in \mathfrak{D}$ of $F\in \mathfrak{E}$ is not homogeneous (so $f\not=f_j$) the terms $\varphi_0 \circ f_t$ for $t\in\{ k+1,k+2,\ldots\}$ might cancel with $\varphi_i\circ h_i$. Then we may no longer have $H(u)=0$ for the set $\{u>a,{u \not= 2a\}}$.
\begin{example}
\label{fnothomogex}  
Let ${f=X^{[7]}+Y^{[6]}+U^{[6]}+X^{[2]}Y^{[2]}U^{[2]}}\in \mathfrak{D}={\kk}_{DP}[X,Y,U]$ and take $R={\kk}\{x,y,u\}.$  Then $H_f=(1,3,6,10,6,3,1,1)$. Let 
${s=1}$, and ${h=Y^{[3]}U}\in \mathfrak{D}_4$, so using the notation of Proposition \ref{noncyclic1prop}, $k=4$. Then $a=j-(k+1)=2$.  Taking $\mathfrak{E}={\kk}_{DP}[X,Y,U,Z]$, and $S={\kk}\{x,y,u,z\}$, let
\[{F=X^{[7]}+Y^{[6]}+U^{[6]}+X^{[2]}Y^{[2]}U^{[2]}+Y^{[3]}UZ} \in\mathfrak{E},\]
and set $A=S/\Ann F$.
Then $Q_F^\vee(2)=\langle Z,Y^{[3]}U\rangle$ from ${y^3u\circ F}$, ${z\circ F}$. Also $(-uz+y^3)\circ F=UZ$, giving a partial of order two (for $-uz$) and degree two (of $UZ$), so
$Q_F^\vee(3)=\langle  UZ\rangle$. The Hilbert function is
${H_F=(1,4,7,10,7,3,1,1)}$, and it has the symmetric decomposition  $\mathcal{D}_F$ (recall $a=2$)
\begin{align*}
H_F(0)&=(1,1,1,1,1,1,1,1)\\
H_F(1)&=(0,2,5,9,5,2,0)\\
H_F(2)&=(0,1,0,0,1,0)\\
H_F(3)&=(0,0,1,0,0),
\end{align*}
while ${\mathcal{D}_f=\bigl(H_f(0),H_f(1)\bigr)=\bigl(H_F(0),H_F(1)\bigr)}$. Were $f$ homogeneous and $a=2$, $H(3)\not=0$ could not occur.
\end{example}As mentioned earlier the Proposition \ref{noncyclic1prop} generalizes \cite[Examples 1.6, 4.7]{I1} referred to above for which $H(A)=(1,3,3,4,2,1,1)$. That example had a quite complicated dual generator. Using the Proposition below we can find examples having rather simpler dual generators. In part A of Proposition~\ref{prop1.24} we describe a large subfamily of all AG
algebras having this Hilbert function.\footnote{The example $f=X^{[6]} + X^{[4]}Z + X^{[2]}Y^{[3]} + X^{[2]}Z^{[2]} + 2Y^{[3]}Z + Z^{[3]}$, defining the ideal $I=(xz-x^3,\,yz-2xy^2,\,z^2-x^2z,\,xyz,\,xz^2-xy^3,\, 2z^3-y^3z,\,y^4)$, also has $H(A_f)=(1,3,3,4,2,1,1)$, but is not in the subfamily we describe.} In Part B of the Proposition we show that any AG algebra of this Hilbert function is isomorphic to one in the large subfamily.  Recall that for $L=(aX+bY)$, $L^{[j]}=\sum_{i=0}^j a^ib^{j-i}X^{[i]}Y^{[j-i]}$. We denote by $R={\kk}\{x,y\}$, $\mathfrak{D}={\kk}_{DP}[X,Y]$, $S={\kk}\{x,y,z\}$ and $ \mathfrak{E}={\kk}_{DP}[X,Y,Z]$.
\index{curvilinear algebra}%
\index{parametrizing AG algebras!of given Hilbert function}%
\index{dual generator of A@dual generator of $A$!use to construct given $\mathcal{D}(A)$}%
\index{GorHR@$\Gor_H(R)$!irreducible family}%
\begin{proposition}
\label{prop1.24}
\begin{enumerate}[(a)]
\item All algebras $A=R/\Ann G$ constructed in the following manner are AG algebras of Hilbert function $H(A)=(1,3,3,4,2,1,1)$. Begin with a graded curvilinear Gorenstein quotient $C=R/J$ of Hilbert function $H(C)=(1,1,\ldots ,1)$: then $C$ is determined by a dual generator $f_6=L^{[6]}\in\mathfrak{D}$ for some non-zero linear form $L=aX+bY$. Then choose a relatively compressed modification $f=f_6+f_5+f_4\in \mathfrak{D}$: this requires $f_5$ to be general enough in $\mathfrak{D}_5$. The algebra $B=R/\Ann f$ satisfies  
\begin{equation}\label{HBeq}
\mathcal{D}(B)=\bigl(H_B(0)=(1,1,1,1,1,1,1) \text { and } H_B(1)=(0,1,2,2,1,0)\bigr).
\end{equation}
Now let $F=f+ Zh$ with $f$ given as above and $h\in \mathfrak{D}_{\le 3}$. There is an open dense subset $U_B\subset \mathfrak{D}_3$ such that if the top degree form $h_3\in  U_B$, then the algebra $A'=S/\Ann F$ satisfies
\begin{equation}\label{DAex1.6}
\mathcal{D}(A')=\bigl(H_B(0),H_B(1),H_A(2)=(0,1,0,1,0)\bigr) \text { and } H(A')=(1,3,3,4,2,1,1).
\end{equation}
Also $Q_{A'}^\vee(2)\cong \langle Z,h\rangle$ is non-cyclic.\par
Finally, let $G=F+G_{\le 3}$ where $G_{\le 3}\in \mathfrak{E}$ is arbitrary and set $A=S/\Ann G$.
Then $\mathcal{D}(A)=\mathcal{D}(A')$ and $H(A)=H(A')$.
\item Each AG algebra $A$ with $H(A)=(1,3,3,4,2,1,1)$ is isomorphic to an algebra that may be constructed as in part (A).
\item The AG algebras quotients of $S$ having Hilbert function  $(1,3,3,4,2,1,1)$ form an irreducible family having dimension 28.
\item  The algebra $S/\Ann G_0$ where $ G_0=X^{[6]}+X^{[2]}Y^{[3]}+ZY^{[3]}$ has Hilbert function $H_{G_0}=(1,3,3,4,2,1,1)$ and is defined by the ideal $ I= (xz,\,yz-x^2y,\,z^2,\,y^4,\,xy^3-x^5)$.
\item  An open dense family of AG algebras having Hilbert function  $(1,3,3,4,2,1,1)$ are CI's. 
\end{enumerate}
\end{proposition}
\begin{proof}[Proof of (a)] 
It is readily seen from the Symmetric Decomposition Theorem \ref{mainoldthm} that the above decomposition $\mathcal{D}$ of Equation  \eqref{DAex1.6} is the unique one possible for $H(A)=(1,3,3,4,2,1,1)$.
The statement about the dual generator for the algebra $C$ is elementary and that about the algebra $B$ in \eqref{HBeq} follows from Proposition \ref{maxprop}.   Taking the $G$ above and $k=3$, $a=2$ in Proposition~\ref{noncyclic1prop} we have that $Q_A(2)$ is non-cyclic  with $H_A(2)=(0,1,0,1,0)$.  Since by Corollary \ref{partialdecompcor} $G_{\le 3}=G_{\le 6-3}$ can only affect $H\bigl(Q(a)\bigr)$ for $a\ge 3$ and $H_A(3)$ must have center of symmetry $3/2$, while $H(A')_1=3$ already we have that $H_A(a)=0$ for $a\ge 3$, in $r=3$ variables, so $\mathcal{D}(A)=\mathcal{D}(A')$ and $H(A)=H(A')$.\vskip 0.2cm\noindent
\textit{Proof of (b)}. Let $A=S/I$, $I=\Ann G$ be an AG algebra of Hilbert function $H=(1,3,3,4,2,1,1)$.  By the Normal Form Theorem (see Theorem 5.3 of \cite{I1} and Theorem \ref{normalform5thm} below) $A$ is isomorphic to an AG algebra $\sigma(A)=S/\Ann G$, $G=g_6+g_5+\cdots $ such that $g_6,g_5\in \mathfrak{D}$ (two variables), determines $B=R/\Ann f$, $f=g_6+g_5$ of decomposition that of \eqref{HBeq}. Since by Macaulay's theorem the first differences of a HF of an AG height two algebra are at most 1, when we regard $f\in S$, we still have that $f$ determines the same $H_B(0)$, $H_B(1)$ from \eqref{HBeq} and $H_B(a)=0$ for $a>1$.  Applying Lemma \ref{linearZlem} below, we have that $G=g_6+g_5+Zh+G_{\le 3}$, with $h\in \mathfrak{D}_3 $ and $G_{\le 3}\in \mathfrak{E}=S^\vee$. Since $G_{\le 3}$ can only influence $H_A(a)$, $a\ge 3$ and each such $H_A(a)$ is zero in embedding dimension three, we have $G_{\le 3}=g_3+g_2\in \mathfrak{E}$ is arbitrary.   \par\noindent
\textit{Proof of (c)}. First assume that $G$ is constructed as in part (A). The CI quotients of $R={\kk}\{x,y\}$ having a given Hilbert function form an irreducible family ${\Gor}_{H_f}(R)$ by \cite{Bri,I8}; then those  $F=f+Zh$ are parametrized by $h$ in an open dense in an affine space
$\mathfrak{D}_3$; to this we can add arbitrary elements in $\mathfrak{E}_{\le 3}$, which mod ${R\circ (f+Zh)}$ are parameters: hence we have an irreducible family parametrizing those AG algebras constructed as in part A. However, when $f_6=X^{[6]}$ then there may be exotic terms of the form $cX^{[4]}Z$ in degree five determining $c^2X^{[2]}Z^{[2]}$ in degree 4, this adds a one-parameter fibre, yielding yet again an irreducible family.\par
We resume this to give a dimension count: choose a $2$-dimensional subspace for $\mathfrak{D}_1$ in $\mathfrak{E}_1=\langle X,Y,Z\rangle$ (this is the choice of an element of projective plane $\mathbb P^2)$: then the CI's of HF $(1,2,3,3,2,1,1)$ form a $10$-dimensional family in $R$ (The dimension formula from \cite[Theorem 2.12]{I8} is $n-\sum(e_i)(e_i+1)/2=n-d-\sum(e_i)(e_i-1)/2$ where $e_i=t_{i-1}-t_{i}$ and $d$ is the order of the defining ideal).  Choosing $hZ$ with $ h$ in an open in $ \mathfrak{D}_3$ gives a 6-dimensional fibre (constant multiple matters). Then adding on $G_{\le 3}\in \mathfrak{E}_{\le 3}$ mod what we already have in ${R\circ (g_6+g_5+hZ)}$ of Hilbert function $(1,3,3,4, \ldots)$ is a $(6-3+10-4)=9$ dimensional fibre. We add one for the exotic term of degree 5:  This gives a $(2+10+6+9+1)=28$ dimensional irreducible family of AG quotients of $S$ having Hilbert function $H(A)$.\par\noindent
\textit{Proof of (d)}. This can be checked by hand.\par\noindent
\textit{Proof of (e)}. By \cite[Example 1.6,4,7]{I1}, there is a CI of this Hilbert function. But being a CI is an open condition on algebras of a fixed Hilbert function (the number of generators of the defining ideal $I$ is semi-continuous). This completes the proof of Proposition \ref{prop1.24}.
\end{proof}\par
\noindent
\textbf{Question.}  What conditions on $f$ assure that for $s=1$ and a generic $h$, the $G$ constructed in Proposition \ref{prop1.24} defines a complete intersection ${S/\Ann G}$, as in the Example 1.6 of \cite{I1}?
\subsection{The associated graded algebra $A^\ast$ does not determine $\mathcal{D}(A)$.}\label{AGA2Hfsec}
\begin{remark}
\label{diffdecomprem} 
Does $A^\ast$ determine the symmetric Hilbert function decomposition $\mathcal{D}(A)$?\par 
\textbf{Yes} when $r=2$.
When ${r=2}$ a stronger result is true: the associated graded algebra $A^*$ of $A$ determines the stratification  $C(a)$, $0\le a\le j-1$ and hence the symmetric decomposition components $Q(a)$. This is shown in the proof of Lemma 2.3, and Theorem 2.6 of \cite{I1}: the latter shows that the graded algebra $A^\ast$ has the least number of generators possible given $H(A^\ast)$, they have different degrees, and the former shows that the unique generator $h_a$ of $A^\ast$ generates $C(a)$. \par
\textbf{No} when $r\ge 3$. When ${r\ge 3}$ the filtration of $A^\ast$ by the ideals $C(a)$ in general contains additional information that is not present in $A^\ast$ itself: that is, two AG algebras $A$, $B$ may have the same associated graded algebra $A^\ast=B^\ast$, but there may be two different Hilbert function decompositions $\mathcal{D}(A)\not= \mathcal{D}(B)$. The first example of this phenomenon was given for embedding dimension  $r=4$  \cite[Example~4]{I5}. There $R={\kk}\{x,y,z,w\}$ and $A^\ast=R/(x^3V,y^2V,z^4V,w^3V)$ where $V=R_1=\langle  x,y,z,w\rangle $.\footnote{The two height four AG algebras $A,B$ having different Hilbert function decompositions but the same associated graded algebra $A^\ast=B^\ast=R/(x^2,y^2,w^3,z^4)V$, from \cite[Example~4]{I5} are $A=R/\Ann F$,  $F=X^{[2]}YZ^{[3]}W^{[2]}+Y^{[2]}Z^{[4]}+X^{[3]}W^{[3]}$ and $B= R/\Ann G$, $G=X^{[2]}YZ^{[3]}W^{[2]}+X^{[3]}Z^{[4]}+Y^{[2]}W^{[3]}$. Here $Q_A(0)=Q_B(0)=R/\Ann X^{[2]}YZ^{[3]}W^{[2]}=R/(x^3,y^2,z^4,w^3)$ and $H_A(1)=0$, $Q_A^\vee(2)=\langle Y^{[2]},X^{[3]},W^{[3]}, Z^{[4]}\rangle$ and $ H_A(2)=(0,0,1,2,1)$. For the other AG algebra $Q_B^\vee(1)=\langle X^{[3]},Z^{[4]}\rangle$, $H_B(1)=(0,0,0,1,1)$ and $Q_B^\vee(3)=\langle Y^{[2]},W^{[3]}\rangle$, $H_B(3)=(0,0,1,1)$. The Hilbert function $H(A)=H(B)=(1,4,10,16,17,14,9,1)$.} We give several codimension three examples below (Examples~\ref{AG=ex} and \ref{simple=AG}). We also give the corresponding $Q^\vee(a)$ decompositions from Lemma \ref{Qdualdescription}.
\index{associated graded algebra!with two symmetric decompositions}%
\end{remark}
\index{symmetric decomposition of Hilbert function!not determined by $A^\ast$ when $r\ge 3$}
\begin{example}[Associated graded algebra with two symmetric decompositions]
\label{AG=ex} 
Let $R={\kk}\{ x,y,z\}$ and $\mathfrak{D}={\kk}_{DP}[X,Y,Z]$. Consider $H=(1,3,3,2,2,1)$. Then 
\begin{equation*}
f=X^{[5]}+Y^{[5]}+(X+Y)^{[4]}+Z^{[2]}\in \mathfrak{D}
\end{equation*}
defines an AG algebra $A=R/I$ with 
\begin{equation*}
I=\Ann f=(xz,\,yz,\,z^2-xy^3,\,x^2y-xy^2,\,xy^2+2xy^3-x^4-y^4)
\end{equation*}
having Hilbert function $H$ and HF decomposition $\mathcal{D}(A)=\mathcal{D}_1$,  where
\begin{equation*}
\mathcal{D}_1=\bigl( H_A(0)=(1,2,2,2,2,1),\, H_A(1)=(0,0,1,0,0),\, H_A(3)=(0,1,0) \bigr)
\end{equation*}
The associated graded algebra $A^\ast =R/I^\ast$ where  
\begin{equation}
I^\ast=(xz,\,yz,\,z^2,\,x^2y,\,xy^2,\,x^5-y^5,\,x^6,\,y^6).
\end{equation}
This can be seen readily by considering the vector space span of the leading (highest degree) terms of the elements of
$R\circ f$ 
\begin{equation*}
(R\circ f)^*=\langle X^{[5]}+Y^{[5]},\,X^{[4]},\,Y^{[4]},\,X^{[3]},\,Y^{[3]},\, X^{[2]},\,XY,\, Y^{[2]},\,X,\,Y,\,Z,\,1\rangle,
\end{equation*}
then calculating the annihilator $I$ in each degree: for example $I_2=\langle X^{[2]},\, XY,\,Y^{[2]}\rangle^\perp\cap R_2=\langle xz,\,yz,\, z^2\rangle$.\par
The dual to the $Q_A(a)$ decomposition of $A^\ast$ satisfies 
\begin{equation*}
Q_A^\vee(0)=R\circ (X^{[5]}+Y^{[5]}),\ Q_A^\vee (1)=\langle  XY\rangle \text { and }Q_A^\vee(3)=\langle Z\rangle.
\end{equation*}
The element $g=X^{[5]}+Y^{[5]}+ZXY\in \mathfrak{D}$  defines an AG algebra $B$ with $J=\Ann g=(xz-y^4,\,yz-x^4,\,z^2,\,x^2y,\,xy^2)$, having the same associated graded algebra $B^\ast=A^\ast$, but here the Hilbert function $H$ has decomposition $\mathcal{D}(B)=\mathcal{D}_2$, where
\begin{equation*}
\mathcal{D}(B)=\bigl( H_B(0)=(1,2,2,2,2,1),\, H_B(2)=(0,1,1,0) \bigr).
\end{equation*}
The dual to the $Q(a)$ decomposition of $B^\ast$ is 
\begin{equation*}
Q_B^\vee(0)=R\circ (X^{[5]}+Y^{[5]}) \text { and } Q_B^\vee (2)=\langle  XY, Z\rangle.
\end{equation*} (See Lemma \ref{Qdualdescription} and Example \ref{stdex} for how these decompositions are constructed). A consequence of the semicontinuity/deformation results of \cite[\S 4.1]{I1} is that no family of AG algebras having decomposition $\mathcal{D}_2$ can have a specialization to an algebra having decomposition $\mathcal{D}_1$.\footnote{This Hilbert function  $H=(1,3,3,2,2,1)$ appeared in \cite[Example 3.13]{I1} without the information that the associated graded algebras were the same for $A$ and $B$. There it was emphasized that (change of notation) $H(B)=H_B(0)+H_B(2)$ is a $2$-RCM of $H_B(0)$, but the decomposition of $H$ is not necessarily unique. Note, \cite[Example 3.13]{I1} has a typo in a generator for the ideal $J$: ``$xy-x^4$'' there should be ``$yz-x^4$''.}
\end{example}
We next give a class of examples of AG algebras $A,B$ in codimension three having the same associated graded algebra, but different Hilbert function decompositions.
\index{associated graded algebra!with two symmetric decompositions}%
\begin{example}
\label{simple=AG}
Let ${k\ge 3}$, ${H=(1,3,3,4,\ldots k,k+1,k,\ldots, 3,2,1_{2k})}$. Let $F=X^{[k+1]}Y^{[k-1]}+ZY^{[k]}$ defining the AG algebra $A=R/I$, $ I=(xz,\,yz-x^{k+1},\,z^2,\,xy^k,\, y^{k+1})$ of Hilbert function $H(A)=H$.  Then the Hilbert function decomposition for $A$ is 
\begin{equation*}
{\mathcal{D}}(A): \bigl( H_A(0)=(1,2,3,\ldots, k-1,k,k_k,k,k-1,\ldots ,3,2,1),\, H_A(k-1)=(0,1,0, \ldots, 0,1_k,0)\bigr).
\end{equation*}
Here $Q_A(k-1)$ is not cyclic, nor generated in degree 1. The duals to $Q_A(a)$ satisfy
\begin{equation*} 
Q_A^\vee(0)=R\circ X^{[k+1]}Y^{[k-1]}, \ Q_A^\vee(k-1)=\langle Z, Y^{[k]}\rangle.
\end{equation*}
Now let $G=X^{[k]}Y^{[k]}+Z^{[2]}$ defining $B=R/J$, $J=(xz,\,yz,\,z^2-x^ky^k,\,x^{k+1},\,y^{k+1})$. Then the Hilbert function decomposition for $B$ is 
\begin{equation*}
{\mathcal{D}}(B):\bigl( H_B(0)=(1,2,3\ldots, k-1,k,(k+1)_k,k,\ldots,2,1),\, H_B(2k-2)=(0,1,0)\bigr),
\end{equation*}
and we have $Q_B^\vee(0)=R\circ X^{[k]}Y^{[k]}$ and $ Q_B^\vee(2k-2)=\langle Z\rangle,$ which is, of course, cyclic.\par
In each case, the associated graded algebra is $A^\ast=B^\ast=R/(xz,\,yz,\,z^2,\,x^{k+1},\,y^{k+1}).$\par
Since $H_A(0)_k=k<H_B(0)_k=k+1$ no AG algebra of decomposition $\mathcal{D}(A)$ can specialize to one of decomposition $\mathcal{D}(B)$.
\end{example}

\subsection{AG algebras whose dual generator is linear in some variables.}\label{restrictsymdecompsec}
For this section we set  $R={\kk}\{x_1,\ldots ,x_r\}$, ${\mathfrak{D}}={\kk}_{DP}[X_1,\ldots ,X_r]$ and $S=R\{z_1,\ldots ,z_s\}$ ${\mathfrak{E}}={\kk}_{DP}[X_1\ldots, X_r; Z_1,\ldots ,Z_s]$.
Our main result here states that when the dual generator $F$ of an AG algebra has the form $F=f+\sum_{i=1}^s Z_ih_i\in \mathfrak{E}$, with $\{f, h_1,\ldots ,h_s\}\subset \mathfrak{D}$  then only certain specified $Q(a)$ can be non-zero, those corresponding to pairwise sums of $j-d_i$ where $d_i=\deg h_i$ (Theorem \ref{DArestrictthm}). This result will be useful in our later construction of AG algebras having ``designed'' symmetric decompositions,  where we specify the set of $a$ for which $Q^\vee (a)\not=0$.  We use it also to restrict the possible symmetric decompositions $\mathcal{D}$ in the proof of Proposition~\ref{prop1.24}.\par
We first show that dual generators having the form $F=f_j+f_{j-1}+\cdots +f_{j-a}+\cdots $ with $f_j,\ldots, f_{j+1-a}\in \mathfrak{D}$ and $f_{j-a}\in \mathfrak{E}$ must have $f_{j-a}$
linear in the variables $Z_1,\ldots, Z_s$, in order for $H_F(a)$ to have interior zeroes.
\index{linearity of dual generator in $Z$}%
\index{H(a) has interior zeroes@$H(a)$ has interior zeroes!and linearity of $f$ in $Z$}%
\index{Q(a)dual@$Q_A^\vee (a)$, dual to $Q_A(a)$}%
\index{dual generator of A@dual generator of $A$!linear in some variables}%
\index{Q(a)decomposition@$Q(a)$ decomposition!which $Q(a)$ can be non-zero!$F=f+\sum Z_ih_i$}%
\begin{lemma}[\textsc{Interior zeroes of $H_A(a)$ and linearity of terms in $F$}]\label{linearZlem} 
Assume that $F=f_j+f_{j-1}+\cdots +f_{j-a}+\cdots$ where $f_i\in \mathfrak{D}$ for $j\ge i>j-a$, and $f_{j-a}\in \mathfrak{E}$, let $A=S/\Ann F$, and suppose that  $H(a)_1=H(a)_{j-a-1}\neq0$ and that $H_A(a)_u =0$ for some $u\in[2,j-a-2]$. Then $f_{j-a}$ is linear in $Z_1,\ldots, Z_s$.
\end{lemma}
\begin{proof}
The assumption $H(a)_1\not=0$ and $F=f+f_{j-a}+\cdots\in\mathfrak{E}$
with $f=f_j+\cdots f_{j-a+1}\in \mathfrak{D}$ implies that $f_{j-a}\in
\mathfrak{E}\setminus\mathfrak{D}$. Suppose by way of contradiction that
$f_{j-a}$ has a term $\mu=\mu'\cdot \mu''$ where $\mu'\in
{{\kk}_{DP}}[Z_1,\ldots ,Z_s]$ is non-linear, and $\mu''\in
\mathfrak{D}$, and suppose that $\mu'$ has the maximum possible degree
$d\ge 2$ among such terms.  Write $\mu'=\mu'_1\cdots \mu'_d$ where
$\mu'_i\in \{Z_1,\ldots ,Z_s\}$ and let $z'_i$  be the corresponding
variable among $\{z_1,\ldots ,z_s\}$ and set $z'=z'_1\cdots z'_{d-1}$.
Then, since the top-degree terms of elements in
${Q^\vee(0),\ldots,Q^\vee(a-1)}$ all lie in $\mathfrak{D}$, we have that
\begin{equation}
\label{stringforQ(a)eq}
\begin{aligned}
&z'_1\circ \mu,\, (z'_1z'_2)\circ\mu,\ldots,\, (z'_1z'_2\cdots
z'_{d-1})\circ \mu, \text { and }\\
&(z'R_1)\circ \mu,\, (z'R_2)\circ \mu,\ldots , (z'R_{j-a-d})\circ \mu,
\end{aligned}
\end{equation}
being terms in $\mathfrak{E}\setminus\mathfrak{D}$, occur as non-zero
terms in elements of $Q^\vee(a)$ (the action of $z^\prime$ on
$\mu^\prime$ leaves a degree-one element in the $Z$ variables). Since
the degrees of the terms in the first line of Equation
\eqref{stringforQ(a)eq} range from ${j-a-1}$ to ${j-a-d+1}$ and in the
second line from ${j-a-d}$ to $1$, we get that ${Q(a)_i\ne0}$ for ${1\le
i\le j-a-1}$, meaning that $H(a)$ has no zero gaps. This proves the Lemma.
\end{proof}\par
\vskip 0.2cm
The next theorem studies the more general class of AG algebras whose dual generators are $F=f_j+\cdots$ where
$f_j\in \mathfrak{D}$ and $F_{<j}$ is linear in the variables $Z_1,\ldots, Z_s$. We write ${F=f_j+\sum_{t=1}^s h_t\cdot Z_t}$, and we can say which modules $Q(a)$ may be non-zero. In fact, considering integers ${a_t=j-(\deg h_t+1)}$, we see that each $Q(a_t)$ may be non-zero, thanks to partials of $h_t$, and modules $Q(a_{t_1}+a_{t_2})$ may also be non-zero, thanks to a cancelling between partials of $f_j$ and partials of $h_{t_2}$, yielding new partials of $h_{t_1}$.

\index{H(u) that can be non-zero@$H(u)$ that can be nonzero!for $f$ linear in $Z_1,\ldots,Z_s$}%
\index{dual generator of A@dual generator of $A$!linear in some variables}%
\index{Q(a)decomposition@$Q(a)$ decomposition!which $Q(a)$ can be non-zero!$F=f+\sum Z_ih_i$}
\begin{theorem}[\textsc{Specifying which $H(u)$ can be non-zero, for dual generators linear in some variables}]
\label{DArestrictthm}
Let ${f\in \mathfrak{D}_j}$ be a homogeneous polynomial of degree $j$. Let ${k_1,\ldots,k_s}$ be integers satisfying ${j-2\ge k_1\ge\cdots\ge k_s\ge1}$ and for ${1\le t\le s}$ choose homogeneous polynomials ${h_t\in \mathfrak{D}_{k_t}}$. Let ${a_t=j-(k_t+1)}$ and consider 
\begin{equation*}
{F}=f+\sum_{t=1}^s h_t\cdot Z_t \in \mathfrak{E}.
\end{equation*}
\begin{enumerate}[(a)]
\item \label{DArestrictthmNonZero} Then  $Q(u)=Q^\vee(u)=0$ for ${u\notin\{0,a_1,\ldots,a_s\}\cup\{a_{t_1}+a_{t_2}\mid 1\le t_1\le t_2\le s\}}$. 
\item \label{DArestrictthmInclusions} Moreover, if we set ${C_t=\{\varphi\in R \mid \varphi\circ f\in R\circ \langle h_1,\ldots,h_t\rangle\}}$ (and ${C_0=\Ann_R f}$), and consider the modules
\begin{align*}
B_t&=\frac{R\circ \langle f,h_1,\ldots,h_t\rangle + (\Ann_R \langle f,h_1,\ldots,h_{t-1}\rangle)
\circ (h_t Z_t+\cdots+h_s Z_s)}
{R\circ \langle f,h_1,\ldots,h_{t-1}\rangle + (\Ann_R \langle f,h_1,\ldots,h_t\rangle)
\circ (h_{t+1} Z_{t+1}+\cdots+h_s Z_s)},\\
B_{t_1,t_2}&=\frac{\Bigl(\bigl(C_{t_2}\setminus \Ann_R h_{t_1}\bigr)
\cap\Ann_R \langle h_1,\ldots,h_{t_1-1}\rangle\Bigr)\circ 
(h_{t_1} Z_{t_1}+\cdots+h_s Z_s)}
{\bigl(C_{t_2-1}
+\Ann_R \langle f,h_1,\ldots,h_{t_1-1}\rangle\bigr)
\circ (h_{t_1} Z_{t_1}+\cdots+h_s Z_s)}, 
\end{align*}
the modules $Q^\vee(a_t)$ and ${Q^\vee(a_{t_1}+a_{t_2})}$ satisfy
\begin{equation}
\label{DArestrictlemmodules}
B_t  \subseteq Q^\vee(a_t)\quad\text{and}\quad
B_{t_1,t_2} \subseteq Q^\vee(a_{t_1}+a_{t_2}).
\end{equation}
\item For each ${u>0}$,
\label{DArestrictlemmodulessum}
\begin{equation}
{\textstyle Q^\vee(u)=\bigl(\bigoplus_{a_t=u}B_t\bigr) 
\oplus \bigl(\bigoplus_{a_{t_1}+a_{t_2}=u}B_{t_1,t_2}\bigr)}
\end{equation}
\end{enumerate}
\end{theorem}
\begin{proof}
Recall from Lemma \ref{Qdualdescription} that ${Q^\vee(u)_i}$ is the quotient of partials of $F$ of degree at most $i$ and order at least ${j-(u+i)}$, i.e.\ the set ${\bigl(\maxA^{\,j-(u+i)}\circ F\bigr)_{\le i}}$, by partials of lower degree -- ${\bigl(\maxA^{\,j-(u+i)}\circ F\bigr)_{\le i-1}}$ -- or higher order -- ${\bigl(\maxA^{\,j-(u+i)+1}\circ F\bigr)_{\le i}}$. \vskip 0.2cm\par
Proof of \eqref{DArestrictthmInclusions}. Let $g$ represent a non-zero element in $B_t$, for some ${1\le t\le s}$. Then we can write ${g=g_1+g_2}$, where
\[{g_1\in R\circ \langle f,h_1,\ldots,h_t\rangle} \text{ and } 
{g_2\in (\Ann_R \langle f,h_1,\ldots,h_{t-1}\rangle)
\circ (h_t Z_t+\cdots+h_s Z_s)}.\]
If $g_1$ is non-zero in $B_t$, i.e.\ ${g_1\notin R\circ \langle f,h_1,\ldots,h_{t-1}\rangle}$, we can assume ${g_1\in R\circ h_t}$ and write ${g_1=\beta\circ h_t}$, for some ${\beta\in R}$. Since $h_t$ is homogeneous, we have that the order of $g_1$ as a partial of $h_t$ equals ${\deg h_t-\deg g_1}$, so we can assume that ${\ord\beta=\deg h_t-\deg g_1}$. But we can see that ${g_1=(z_t\beta)\circ F}$, so $g_1$ is a partial of $F$ of order
\[{\ord\beta+1=k_t-\deg g_1 +1=j-a_t-\deg g_1}\]
and therefore ${g_1\in Q^\vee(a_t)}$. Now if $g_2$ is non-zero, we can write ${g_2=\theta\circ (h_t Z_t+\cdots+h_s Z_s)}$, with
\[{\theta\in \Ann_R \langle f,h_1,\ldots,h_{t-1}\rangle \setminus\Ann_R h_t},\]
and $\theta$ can be chosen such that it has the highest possible order. But then ${g_2=\theta\circ F}$ and its order as a partial of $F$ equals ${\ord \theta}$. Since $h_t$ is homogeneous, we get
\[\ord g_2=\ord \theta = k_t-\deg g_2+1=j-a_t-\deg g_2,\]
so ${g_2\in Q^\vee(a_t)}$.

Suppose now that $g$ represents a non-zero element in $B_{t_1,t_2}$, for some ${1\le t_1,t_2\le s}$. Then we can write ${g=\theta\circ(h_{t_1} Z_{t_1}+\cdots+h_s Z_s)}$, with ${\theta\in\bigl(C_{t_2}\setminus \Ann_R h_{t_1}\bigr)\cap\Ann_R \langle f,h_1,\ldots,h_{t_1-1}\rangle}$. So there are ${\eta_1,\ldots, \eta_{t_2}\in R}$ such that ${\theta\circ f=\eta_1\circ h_1+\cdots+\eta_{t_2}\circ h_{t_2}}$. Therefore
\[
\theta\circ F=\eta_1\circ h_1+\cdots+\eta_{t_2}\circ h_{t_2}+g
\]
i.e.\ ${g=(\theta-z_1\eta_1-\cdots-z_{t_2}\eta_{t_2})\circ F}$. Therefore,
\begin{multline*}
\ord g= \ord \eta_{t_2}+1=k_{t_2}-\deg(\theta\circ f)+1=k_{t_2}-(j-\ord\theta)+1\\
=k_{t_2}-(j-k_{t_1}-1+\deg g)+1=j-a_{t_1}-a_{t_2}-\deg g.
\end{multline*}
So ${g\in Q^\vee(a_{t_1}+a_{t_2})}$. 
\vskip 0.2cm\par
Proof of \eqref{DArestrictthmNonZero} and \eqref{DArestrictlemmodulessum}. Fix ${u,i>0}$ such that ${Q^\vee(u)_i\ne 0}$. Then there is a non-zero partial ${g\in R\circ {F}}$ of degree $i$ and order ${j-(u+i)}$. We wish to use arguments on the degree and on the order of $g$ to show that ${u=a_t}$ or ${u=a_{t_1}+a_{t_2}}$, for some $t$, or $t_1$ and $t_2$. Note that if $g'$ is another partial of $F$, also belonging to ${\bigl(\maxA^{\,j-(u+i)}\circ F\bigr)_{\le i}}$, but with ${\ord g'>\ord g}$ or ${\deg g'<\deg g}$, then $g$ and ${g-g'}$ represent the same element in $Q^\vee(u)_i$. We will use this fact several times along the proof.

Let ${\varphi\in S={\kk}[x_1,\ldots,x_r,z_1,\ldots,z_s]}$ be an element of order ${j-(u+i)}$ such that ${g=\varphi\circ {F}}$. Then since ${(z_1,\ldots,z_s)^2\circ {F}=0}$, we may assume that
\begin{equation*}
\varphi=\varphi_0+\sum_{t=1}^s\varphi_tz_t,
\end{equation*}
with ${\varphi_0,\varphi_1\ldots,\varphi_s\in R={\kk}[x_1,\ldots,x_r]}$. Therefore
\begin{equation*}  
g=\varphi\circ {F}=\varphi_0\circ f+\sum_{t=1}^s(\varphi_0\circ h_t)Z_t+
\sum_{t=1}^s\varphi_t\circ h_t
\end{equation*}
and a term of degree $l$ of $g$ is
\begin{equation*}  
g_l=(\varphi_0)_{j-l}\circ f+\sum_{t=1}^s\bigl((\varphi_0)_{k_t-l+1}\circ h_t\bigr)Z_t+
\sum_{t=1}^s(\varphi_t)_{k_t-l}\circ h_t.
\end{equation*}
If ${(\varphi_0)_{j-i}\circ f\ne0}$, consider the partial ${(\varphi_0)_{j-i}\circ F}$, which also has degree $i$, but has order ${j-i}$, greater than ${\ord g=j-(u+i)}$. So this partial satisfies
\[
(\varphi_0)_{j-i}\circ F\in \bigl(\maxA^{\,j-i}\circ F\bigr)_{\le i}
\subseteq  \bigl(\maxA^{\,j-(u+i)+1}\circ F\bigr)_{\le i}
\]
and we see that both partials $g$ and ${g-(\varphi_0)_{j-i}\circ F}$ represent the same class in ${Q^\vee(u)_i}$, so we can replace $g$ by ${g-(\varphi_0)_{j-i}\circ F}$ and thus assume that ${(\varphi_0)_{j-i}\circ f=0}$. Therefore the leading term of $g$ is
\begin{equation*}  
\lt(g)=g_i=\sum_{t=1}^s\bigl((\varphi_0)_{k_t-i+1}\circ h_t\bigr)Z_t+
\sum_{t=1}^s(\varphi_t)_{k_t-i}\circ h_t.
\end{equation*}

We wish to start by looking at these two sums in $g_i$ separately, so we consider the partials 
\[
{g'=\sum_{t=1}^s(\varphi_t)_{k_t-i}z_t\circ F}\quad \text{and} \quad {g''=g-g'=
\bigl({\textstyle \varphi-\sum_{t=1}^s(\varphi_t)_{k_t-i}z_t}\bigr)\circ F}.
\] 
Since both these partials are obtained by eliminating terms of $\varphi$, we see that ${\ord g',\ord g''\ge\ord g}$. On the other hand, they satisfy 
\begin{equation*}
g'_i=g'=\sum_{t=1}^s(\varphi_t)_{k_t-i}\circ h_t, \qquad g''_i=\sum_{t=1}^s\bigl((\varphi_0)_{k_t-i+1}\circ h_t\bigr)Z_t,
\end{equation*}
and ${g''_l=g_l}$ for every ${l\ne i}$. In particular, $g'$ is homogeneous of degree $i$ and $g''$ has degree at most $i$.

We shall distinguish three cases.

\paragraph{Case 1}
Suppose that ${g''_i=0}$ or that ${\ord g''>\ord g}$, i.e.\ we have ${g''\in \bigl(\maxA^{\,j-(u+i)}\circ F\bigr)_{\le i-1}}$ or ${g''\in \bigl(\maxA^{\,j-(u+i)+1}\circ F\bigr)_{\le i}}$. Then both  $g$ and ${g'=g-g''}$ represent the same class in ${Q^\vee(u)_i}$, so we can replace $g$ by ${g'}$ and assume that ${g=\sum_{t=1}^s(\varphi_t)_{k_t-i}\circ h_t}$.

Let $t_1$ be the maximum integer such that ${(\varphi_{t_1})_{k_{t_1}-i}\circ h_{t_1}\ne0}$. Then we can assume that ${\varphi_t=0}$ for ${t>t_1}$, so the order of $\varphi$ is ${\ord\bigl((\varphi_{t_1})_{k_{t_1}-i}\cdot z_{t_1}\bigr)=k_{t_1}-i+1}$, and ${u=a_{t_1}}$. 

Furthermore, we can take $t_0$ to be the minimum integer such that ${k_{t_0}=k_{t_1}}$, so we get ${k_{t_0}=k_{t_0+1}=\cdots=k_{t_1}}$, and we know that the partial ${\sum_{t<t_0}(\varphi_t)_{k_t-i}z_t\circ F}$ has order greater than ${k_{t_1}-i+1}$, so we can again replace $g$ by ${g-\sum_{t<t_0}(\varphi_t)_{k_t-i}z_t\circ F}$. Then it is easy to check that $g$ can be written as an element of ${B_{t_0}\oplus\cdots\oplus B_{t_1}}$.

\paragraph{Case 2}
Suppose that ${g'_i=0}$ or that ${\ord g'>\ord g}$. Then, as before, we can replace $g$ by $g''$ as a representant of their class in $Q^\vee(u)_i$. To know the order of $\varphi$, we will argue on which of its terms we can assume to vanish. To this end, we wish to pay special attention to the action of the following summands of $\varphi$:
\[
\eta(e):=(\varphi_0)_{e}+{\textstyle\sum_{t=1}^s(\varphi_t)_{k_t-j+e}z_t}.
\]

Consider the sets ${T=\big\{t\in\{1,\ldots,s\} \mid (\varphi_0)_{k_{t}-i+1}\circ h_{t}\ne0\big\}}$ and ${E=\{k_{t}-i+1 \mid t\in T \}}$, and let ${e\notin E}$. Let us show that we can assume that ${\eta(e)=0}$. If ${e<j-i}$, we must have
\[
0=g_{j-e}=(\varphi_0)_{e}\circ f+\sum_{t=1}^s\bigl((\varphi_0)_{k_t-j+e+1}\circ h_t\bigr)Z_t+
\sum_{t=1}^s(\varphi_t)_{k_t-j+e}\circ h_t.
\]
In particular, since the middle term is the only one involving the variables ${Z_1,\ldots,Z_s}$, we must have  ${(\varphi_0)_{e}\circ f+\sum_{t=1}^s(\varphi_t)_{k_t-j+e}\circ h_t=0}$. Therefore
\[
\eta(e)\circ F=\bigl((\varphi_0)_{e}+{\textstyle\sum_{t=1}^s(\varphi_t)_{k_t-j+e}z_t}\bigr)\circ F=
\sum_{t=1}^s\bigl((\varphi_0)_{e}\circ h_t\bigr)Z_t.
\] 
We can easily see that this is a polynomial of degree lower than $i$, since any non-zero term of degree ${l>i}$ would imply ${g_l\ne0}$, and if $\bigl((\varphi_0)_{e}\circ h_t\bigr)Z_t$ is non-zero and has degree $i$, for some $t$, we would have ${e=k_t-i+1}$, with ${t\in T}$, a contratdiction. So we can replace $g$ by ${g-\eta(e)\circ F}$ and thus assume that ${\eta(e)=0}$. If ${e=j-i}$, we already have from our prevoious assumtions that ${(\varphi_0)_e=(\varphi_0)_{j-i}=0}$ and for all $t$, ${(\varphi_t)_{k_t-j+e}=(\varphi_t)_{k_t-j}=0}$, i.e.\ ${\eta(j-i)=0}$. If ${e>j-i}$, we get
\[
\eta(e)\circ F=\bigl((\varphi_0)_{e}+{\textstyle\sum_{t=1}^s(\varphi_t)_{k_t-j+e}z_t}\bigr)\circ F=
(\varphi_0)_{e}\circ f+\sum_{t=1}^s\bigl((\varphi_0)_{e}\circ h_t\bigr)Z_t
+\sum_{t=1}^s(\varphi_t)_{k_t-j+e}\circ h_t.
\]
Since this partial has degree lower than $i$, we can again replace $g$ by ${g-\eta(e)\circ F}$ and thus assume that ${\eta(e)=0}$.

At this point, we can write
\[
g=\sum_{e\in E}\eta(e)\circ F=
\sum_{e\in E}\bigl((\varphi_0)_{e}+{\textstyle\sum_{t=1}^s(\varphi_t)_{k_t-j+e}z_t}\bigr)\circ F,
\]
so we have ${\ord g=e=k_{t_1}-i+1}$, or ${\ord g=k_{t_2}-j+e+1=k_{t_2}-j+k_{t_1}-i+2}$, for some ${e\in E}$ and some $t_1$ and $t_2$. In the first case, we get ${u=a_{t}}$, and in the latter we get ${u=a_{t_1}+a_{t_2}}$. 

Looking closer at each summand ${\eta(e)\circ F}$, with ${e\in E}$, we can see that since ${e<j-i}$, the terms ${(\varphi_0)_{e}\circ f}$ and ${\sum_{t=1}^s(\varphi_t)_{k_t-j+e}\circ h_t}$ must cancel. So if ${(\varphi_0)_{e}\circ f=0}$, we can see that 
\[
\eta(e)\circ F=\sum_{t=1}^s\bigl((\varphi_0)_{e}\circ h_t\bigr)Z_t\in B_{t_1},
\]
where $t_1$ is the minimum such that ${(\varphi_0)_{e}\circ h_{t_1}\ne0}$, and therefore ${\ord\eta=k_{t_1}-i+1}$. If ${\ord\eta(e)>\ord g}$, we can replace $g$ by ${g-\eta(e)\circ F}$, as before. If ${\ord\eta(e)=\ord g}$, we can check that ${u=a_{t_1}}$. On the other hand, if ${(\varphi_0)_{e}\circ f\ne0}$, we choose $t_2$ to be the maximum such that ${(\varphi_t)_{k_{t_2}-j+e}\circ h_{t_2}\ne0}$ and $t_1$ again as the minimum such that ${(\varphi_0)_{e}\circ h_{t_1}\ne0}$, and we have
\[
\eta(e)\circ F\in B_{t_1,t_2},
\]
and ${\ord\eta=k_{t_1}+k_{t_2}-j-i+2}$. If ${\ord\eta(e)>\ord g}$, we can once more ignore this partial, and if ${\ord\eta(e)=\ord g}$, we see that ${u=a_{t_1}+a_{t_2}}$.

\paragraph{Case 3}
Suppose that ${g'_i\ne0}$, ${g''_i\ne0}$, and ${\ord g'=\ord g''=\ord g}$. Then, looking at the previous two cases, we have that ${g'\in\bigoplus_{a_t=u}B_t}$ and ${g''\in\bigl(\bigoplus_{a_t=u}B_t\bigr) \oplus \bigl(\bigoplus_{a_{t_1}+a_{t_2}=u}B_{t_1,t_2}\bigr)}$, and this finishes the proof.
\end{proof}

\begin{remark}
In Theorem \ref{DArestrictthm} the modules $B_t$ and $B_{t_1,t_2}$ depend on the choice of variables in ${\langle Z_1,\ldots, Z_s\rangle}$. If ${k_l=\cdots=k_{l+m}}$, for some integers $l$ and $m$, i.e.\ the homogeneous polynomials  ${h_l,\ldots,h_{l+m}}$ are all of the same degree, then a linear change of variables in ${\langle Z_l,\ldots, Z_{l+m}\rangle}$ will yield different modules $B_t$ and $B_{t_1,t_2}$, but the modules $Q(u)$ remain unchanged.

Note also that in general ${B_{t_1,t_2}\ne B_{t_2,t_1}}$, as is the case in the following example.
\end{remark}

\begin{example}
Let ${f=X^{[4]}Y^{[7]}}$, ${h_1=X^{[5]}Y^{[3]}}$, and ${h_2=X^{[6]}+Y^{[6]}}$. If we consider the polynomial 
\[
F=f+h_1Z_1+h_2Z_2=X^{[4]}Y^{[7]} + X^{[5]}Y^{[3]}Z_1 + (X^{[6]}+Y^{[6]})Z_2
\]
the Hilbert function of ${{\kk}[x,y,z_1,z_2]/\Ann F}$ is ${H_F=(1, 4, 6, 7, 6, 6, 7, 6, 5, 3, 2, 1)}$ and has symmetric decomposition $\mathcal{D}_F$:
\begin{align*}
H_F(0)&=(1, 2, 3, 4, 5, 5, 5, 5, 4, 3, 2, 1)\\
H_F(2)&=(0, 1, 1, 1, 1, 1, 1, 1, 1, 0)\\
H_F(4)&=(0, 1, 0, 0, 0, 0, 1, 0)\\
H_F(6)&=(0, 0, 2, 2, 0, 0).
\end{align*}
Note that, using the notation of Theorem \ref{DArestrictthm}, we have ${a_1=2}$ and ${a_2=4}$, so ${a_1+a_2=6}$. Here we have
\begin{align*}
Q^\vee_F(2)&=\langle Z_1,\, YZ_1,\, Y^{[2]}Z_1,\, Y^{[3]}Z_1 + XZ_2,\, X^{[5]},\, X^{[5]}Y,\, 
X^{[5]}Y^{[2]},\, X^{[5]}Y^{[3]} \rangle,\\
Q^\vee_F(4)&=\langle Z_2,\, X^{[6]}+Y^{[6]}\rangle,\\
Q^\vee_F(6)&=\langle XZ_1,\, YZ_2,\, XYZ_1,\, Y^{[2]}Z_2 \rangle.
\end{align*}
Note that ${B_{1,2}=\langle XZ_1,\, XYZ_1 \rangle}$, ${B_{2,1}=\langle YZ_2,\, Y^{[2]}Z_2 \rangle}$, and ${B_{1,1}=B_{2,2}=0}$.
\end{example}

\begin{corollary}
Under the hypotheses of Theorem \ref{DArestrictthm}, we have, for each $t$, ${B_t=0}$ if and only if ${h_t\in R\circ\langle f,h_1,\ldots, h_{t-1} \rangle}$.
\end{corollary}
\begin{proof}
Suppose that ${h_t\in R\circ\langle f,h_1,\ldots, h_{t-1} \rangle}$. Then clearly 
\[
R\circ\langle f,h_1,\ldots, h_{t-1}, h_{t} \rangle = R\circ\langle f,h_1,\ldots, h_{t-1} \rangle,
\] 
so any element in ${\Ann \langle f,h_1,\ldots,h_{t-1}\rangle}$ is also in ${\Ann h_{t}}$, and we get ${B_t=0}$. On the other hand, if ${h_t\notin R\circ\langle f,h_1,\ldots, h_{t-1} \rangle}$, then $h_t$ itself represents a non-zero element of $B_t$.
\end{proof}

\begin{example}
Under the hypotheses of Theorem \ref{DArestrictthm}, it is possible to have ${Q(a_1+a_2)\ne0}$, even if ${Q(a_1)=Q(a_2)=0}$. Consider ${f=X^{[3]}Y^{[3]}}$, ${h_1=X^{[3]}Y}$, and ${h_2=XY^{[2]}}$. Let
\[
F=f+h_1Z_1+h_2Z_2 = X^{[3]}Y^{[3]} + X^{[3]}YZ_1 + XY^{[2]}Z_2.
\]
The Hilbert function of $A={{\kk}\{x,y,z_1,z_2\}/\Ann F}$ is ${H_A=(1, 4, 5, 4, 3, 2, 1)}$. Here ${a_1=1}$, and ${a_2=2}$ and $H_A$ has symmetric decomposition nonzero for $a=0,2+1$:
\begin{align*}
\mathcal{D}_A: H_F(0)&=(1, 2, 3, 4, 3, 2, 1)\\
H_A(3)&=(0, 2, 2, 0).
\end{align*}
We can easily verify that ${Q^\vee_F(3)=\langle Z_1, Z_2, XZ_1, XZ_2 \rangle}$, so $Q^\vee_F(3)$ is generated in degree one. The reason for the vanishing of ${Q(a_1)}$ and ${Q(a_2)}$ is that $h_1$ and $h_2$ are partials of $f$: this case is related to exotic summands, a subject that we will address in Section \ref{exoticsumsec}.\footnote{The terms $h_1Z_1$ and $h_2Z_2$ are ``exotic summands'': the variables $Z_1,Z_2$ are not seen, as might otherwise be expected, in $Q(1)_1, Q(2)_1$ (Definition \ref{exoticdef}).}
\end{example}

We give now another construction yielding AG algebras with $H_A(a)=(0,0,s,0,\ldots,0,s,0,0)$.
\begin{lemma}
\label{simpledeformlem}
Let ${\mathfrak{D}={\kk}_{DP}[X_1,\ldots ,X_r]}$, ${R={\kk}\{x_1,\ldots ,x_r\}}$, and let ${f\in \mathfrak{D}_j}$ be a homogeneous polynomial of degree $j$ such that the ideal  ${\Ann_R f}$ has order $k$ satisfying $3\le k\le j-3$.  Let ${h\in \mathfrak{D}_{k+1}}$ be a homogeneous polynomial of degree ${k+1}$, such that $h\notin R_{j-(k+1) }\circ f$.  Consider the polynomial
\begin{equation*}
{F}=f+Z^{[j]} + Z\cdot h
\end{equation*}
in ${\mathfrak{E}={\kk}_{DP}[X_1,\ldots ,X_r,Z]}$, and let ${S={\kk}\{x_1,\ldots ,x_r,z\}}$. Then the Hilbert function $H_F$ of ${S/\Ann_{S} F}$ satisfies
$H(u)=0$ for $u\notin\{0,a\}$, where $a=j-k-2$; and we have
\begin{align}
H_F(0)&=H_f(0)+(0,1,1,\ldots,1,1,0),\label{HFCS1eq}\\
H_F(a)&=(0,0,s,0,\ldots,0,s,0,0):\quad H_F(a)_2=H_F(a)_{k}=s,\label{HFHa1eq}
\end{align}
where 
\[
s=\dim_{\kk}\bigl((R\circ h)^\ast_k + (R\circ f)^\ast_k\bigr)/(R\circ f)^\ast_k =\dim_{\kk}\bigl((R_1\circ h) + (R_{j-k}\circ f)\bigr)/(R_{j-k}\circ f).
\]
Furthermore, given $f$ we may choose $h$ so that $s$ has any value satisfying $1\le s \le \min\{r, \dim_{\kk}(\Ann f)_k\}$.
\end{lemma}
\begin{proof}
Since  by choice of $k$ the socle degree term of $F$ is ${F_j=f+Z^{[j]}}$, it is immediate that $H_F(0)$ is the Hilbert function of the connected sum algebra ${S/\Ann F_j}$: the Hilbert function is given by \eqref{HFCS1eq}  (Lemma~\ref{connsumlem} in Section~\ref{connectedsumsec}.) \par
Now since ${zx_i\in \Ann F_j}$, for ${1\le i\le r}$, we see that $zR_1\circ F=R_1\circ h$; but $(R_{j-k}\circ F)_k=  R_{j-k}\circ f $, so $zR_1\circ F$ yields an $s$-dimensional space
supplementing $Q^\vee(0)_{k}$, so it is in $Q(a)_k $ (elements of order $2$ acting on $F$ yield elements of degree $k$ in place of the expected $j-2$, so $a=j-k-2$). This implies that $H(a)_k=s$, and by symmetry of $H(a)$ implies that $H(a)_2=s$. For $u\ge 1$ we have $zR_{1+u}\circ h\subset \mathfrak{D}_{k-u}\subset R_{j-(k-u)}\circ f$ (as order of $\Ann f=k$): thus we have $H(a)_u$ has no further elements arising from $z R\circ F$. The other way to obtain elements of $Q^\vee(a)_u$ is from $(\Ann f)_{k+2-u}\circ Zh\subset Z\mathfrak{D}_{u-1}$. This works for $u=2$ only and yields the $s$-dimensional space 
$Q^\vee(a)_2$.  When $u\ge 3 $ then $(\Ann f)_{k+2-u}=0$ since $\Ann f$ has order $k$; for $u=1$ the image $Z$ is already in $Q(0)$. Thus we have  $Q^\vee (a)_2=Z(\Ann f)_k\circ h$, confirming \eqref{HFHa1eq}.\par
Since ${z^u\circ F=Z^{[j-u]}}$ for $u\ge 2$, since ${z^2\cdot R_{\ge 1}\circ F=0}$, and we have accounted for $R\circ F$, all other $Q(u)=0$ for $u\notin\{0,a\}$.
\end{proof}

\subsection{Partial non-ubiquity in codimension four.}
\label{nonubiquitycod4sec}
We give an example pertaining to caution (h) of Section \ref{cautionsec}, that we term partial non-ubiquity of a symmetric decomposition: here we cannot attain the symmetric Hilbert function  decomposition $\mathcal{D}(A)_{\le 1}$ as the full decomposition $\mathcal{D}(B)$ for any AG algebra $B=R/\Ann G$ satisfying $G=G_j+G_{j-1}$ (Theorem \ref{nonubiq4thm}). This is \emph{partial} non-ubiquity for $A$ as we here restrict the algebras $B$ considered. \par
\index{ubiquity!partial non-ubiquity}%
We let  $R={\kk}\{x,y\}$, ${\mathfrak{D}}={\kk}_{DP}[X,Y]$ and $S={\kk}\{x,y,z,w,\}$, ${\mathfrak{E}}={\kk}_{DP}[X,Y,Z,W]$. We will need the invariant $\tau(V)$ of a vector space of forms $V\subset R_j$ to prove our example.
\begin{definition}
Given a vector space 
$V\subset R_j$ we denote by $R_{-i}V\subset R_{j-i}$ the vector space 
\begin{equation}
R_{-i}V=V:R_i=\{ g\in R_{j-i}\mid R_i\cdot g\subset V\}.
\end{equation}
When $R={\kk}\{x,y\}$ we let 
\begin{align}
\tau (V)&=\dim_{\kk} R_1V-\dim_{\kk}V \notag\\
&=\dim_{\kk}V-\dim_{\kk} R_{-1}V.\label{taueq}
\end{align}
The \emph{ancestor ideal}
$\overline {V}$ of a vector space $V\subset R_j$ of forms is 
\begin{equation}\label{ancestoreq}
\overline {V}=\left(\oplus_{i=1}^j V:R_i\right)\oplus (V).
\end{equation}
\end{definition}
Here the invariant $\tau (V)$ is the number of generators of the ancestor ideal $\overline{V}$ \cite{Ianc}.
For example, if $V=\langle x^4,x^3y,y^4\rangle\subset R_4$, $R={\kk}[x,y]$ we have $\overline{V}=(x^3,y^4)$ and $\tau (V)=2$.\par
We next give the example of an algebra satisfying partial non-ubiquity, which we prove in the succeeding Theorem \ref{nonubiq4thm},
\begin{example}
\label{nonubiq4ex} 
Let $F=X^{[7]}Y^{[7]}+Z\cdot (X+Y)^{[6]}\cdot(X-Y)^{[6]}+W\cdot (X+2Y)^{[6]}\cdot(X-2Y)^{[6]}$ and let $A=S/\Ann F$. Then  we have
\begin{equation}
\label{nonubiq4eq}
\begin{aligned}
H_A(0)&=(1,2,3,4,5,6,7,8,7,6,5,4,3,2,1)\\
H_A(1)&=(0,2,4,6,4,2,0,0,2,4,6,4,2,0)\\
H_A(2)&=(0,0,0,0,0,3,6,3,0,0,0,0,0)\\
H(A)&=(1,4,7,10,9,11,13,11,9,10,11,8,5,2,1)
\end{aligned}
\end{equation}
Here, writing $F=F_{14}+F_{13}=F_{14}+Za+Wb$, in the example $F_{14}$ and the terms $a,b$ of $F_{12}$ are powers of distinct linear forms in $\mathfrak{D}$.
$\tau \bigl( \Ann (F_{14},a,b)\cap R_{12} \bigr)=4$, and also $R/\Ann(F_{14},a,b)$ is compressed, corresponding to the drop $11\to 8\to 5$ in the Hilbert function, in degrees $10$, $11$ and $12$.  Also $R/\Ann (a,b)$ is compressed.
\end{example}
\index{ubiquity!partial non-ubiquity}%
\begin{theorem}[\textsc{Partial non-ubiquity in codimension four}]
\label{nonubiq4thm}
Let  $G=g+g_{13}$, $g\in \mathfrak{D}_{14}$, $g_{13}\in \mathfrak{E}$ define the algebra $B=S/\Ann G$, and assume that $H_B(0)=H_A(0)$, $H_B(1)=H_A(1)$ from \eqref{nonubiq4eq}. Then we have $H_B(2)_6\ge 4.$ 
\end{theorem}
\begin{proof} 
Write $g_{13}=Za+Wb$. The idea is that the high values  of $H_B(1)$ will imply that $\tau \bigl( \Ann (g,a,b)_{12} \bigr)$ is large:  we will show that a key homomorphism whose image is $Q^\vee(2)_6$ has zero kernel so non-zero image.
\begin{enumerate}[i.]
\item First, analogously to the proof of earlier results, it is easy to see that because of the symmetry properties of the decomposition, that $H(B)=H_A(0)+H_A(1)$ implies both $H_B(0)=H_A(0)$ and $ H_B(1)=H_A(1)$. Also, by Lemma \ref{linearZlem} the dual form $G=g+g_{13}$, where
$g_{13}$ is linear in $W,Z$.
\item The linear map $\kappa: R_7\to \mathfrak{D}_7:  \kappa (h)=h\circ g$ is an isomorphism. 
Let  $d=\dim_{\kk} L,$ where $L$ is the vector space 
\begin{equation}
L=R_5\circ  \langle a,b\rangle =\langle zR_5,wR_5\rangle\circ g_{13}\subset  { \mathfrak{D}}_7.
\end{equation}
So for each element $h_{\ell_1,\ell_2}=(w\ell_1+z\ell_2)\circ g_{13}\in  L\subset \mathfrak{D}_7$ with $\ell_1,\ell_2 \in R_5$ there is a unique $h'_{\ell_1,\ell_2}=\kappa^{-1}(h_{\ell_1,\ell_2})\in R_7$ such that $ h'_{\ell_1,\ell_2}\circ g=h_{\ell_1,\ell_2}.$  Furthermore $L'=\kappa^{-1} (L)$ is a $d$-dimensional linear subspace of $R_7$. Then  we have
\begin{equation}
\label{keyeqn3}
(w\ell_1+z\ell_2-h'_{\ell_1,\ell_2})\circ G=-h'_{\ell_1,\ell_2}\circ g_{13}\in \langle Z,W\rangle\cdot {\mathfrak{D}}_5,
\end{equation}
and it follows that $L'\circ g_{13}\subset Q^\vee(2)_6$.\par
\item The kernel of the map: $\alpha: R_7\to R_7\circ g_{13}:   \alpha (h)=h\circ g_{13}\in \langle Z,W\rangle \cdot \mathfrak{D}_5$ satisfies 
\begin{align*}
\ker (\alpha)=\Ann (a,b)_7
\end{align*}
By the assumption on $H_B(1)$ we have that $\dim_{\kk} R_2\circ\langle a,b\rangle = 6$; from the properties of the invariant $\tau$ (\cite{Ianc}) this implies that 
\begin{equation}
\dim_{\kk}R_7\circ \langle a,b\rangle\ge \min\{6, \dim_{\kk}\mathfrak{D}_5\}=6.
\end{equation}
It follows that
\begin{equation}
\begin{aligned}
\dim_{\kk}\ker\alpha&\le 8-6=2 \text { and }\\
\dim_{\kk}L'\circ g_{13}&\ge 6-2=4.
\end{aligned}
\end{equation}
\end{enumerate}
This completes the proof of the Theorem.
\end{proof}\par\noindent
\section{A standard form for the dual generator $F_{{A}}\in \mathfrak{D}$, exotic summands, and modifications.}\label{tail sec}
The question guiding this section is Question 2 of Section \ref{intro1sec}:  Is there a normal or canonical form for the dual generator of an AG algebra $A$, up to isomorphism?  In section \ref{normalformsec} we recall a standard form for the dual generator given in \cite{I1}: a consequence is that $A$ is isomorphic to an algebra whose dual generator $f$ has no  ``exotic summands" (Section \ref{exoticsumsec}). In Section \ref{connectedsumsec} we apply this to a problem of writing certain Artinian Gorenstein algebras as connected sums. In Section \ref{questionsec} we pose some open problems. 
\subsection{Standard form for a dual generator.}\label{normalformsec}
We first discuss the correspondence between algebra isomorphisms of $R$ and the adjoint linear transformation of $\mathfrak{D}$. Then we recall a normal form theorem
from \cite[Theorem 5.3]{I1}, and prove it using the adjoint linear transformation.  Although this result is stated and its proof is outlined in \cite[Theorem 5.3 A, B]{I1}, we include here some further detail and explanation, in particular concerning the adjoint linear map. 
\subsubsection{Adjoint linear map of $\mathfrak{D}$ to an automorphism of $R$.}
We recall the adjoint linear map on $\mathfrak{D}$ corresponding to a given algebra isomorphism $\sigma$ of $R$, from
\cite{Mac1,Em}.   What they term $F\circ g (0)$ is our $g\circ F (0)$,
where $g\circ F(0)$ is the action of $g$ as contraction on $F$.
Let $\sigma $ be a ring automorphism of $R$, and $\xi=\sigma^\vee$ the corresponding adjoint linear map\index{adjoint linear map} on $\mathfrak{D}$.  Then for all $F\in \mathfrak{D}, g\in R$ we define $\xi (F)\in \mathfrak{D}$ by
\begin{equation}\label{equivcontracteq}
\bigl(\sigma (g)\circ \xi (F)\bigr) (0)=(g\circ F) (0).
\end{equation}
\begin{lemma}
\label{adjointlem}
The adjoint linear map $\xi=\sigma^\vee$ on $\mathfrak{D}$ to the automorphism $\sigma$ of $R$ satisfies, for $h\in R$
\begin{equation}\label{adjointcontracteq}
\xi (h\circ F)=\sigma(h)\circ \xi (F).
\end{equation}
and
\begin{equation}\label{adjointboundeq}
\xi : \mathfrak{D}_{\le i} \to  \mathfrak{D}_{\le i}.
\end{equation}
\end{lemma}
\begin{proof} 
Let $h\in R$. By definition and \eqref{equivcontracteq} we have, for any ${g\in R}$, 
\begin{multline*}
\bigl(\sigma(g) \circ \xi(h\circ F)\bigr) (0) = \bigl(g \circ (h\circ F)\bigr) (0) 
= \bigl( (g h) \circ F \bigr) (0)\\
= \bigl(\sigma(gh) \circ \xi(F)\bigr) (0) 
= \Bigl(\sigma(g) \circ \bigl(\sigma(h)\circ \xi( F )\bigr)\Bigr) (0)
\end{multline*}
which implies \eqref{adjointcontracteq}. Now by \eqref{equivcontracteq} we have $\bigl(g \circ \xi (F)\bigr)  (0)=\bigl(\sigma^{-1}(g)\circ F\bigr) (0)$.
So
\begin{equation*}
g\in \maxR^{\,i+1} \Rightarrow \sigma^{-1}(g)\in \maxR^{\,i+1} \text { and then } \deg F\le i\Rightarrow  \bigl(\sigma^{-1}(g)\circ F\bigr) (0)=0\Rightarrow \bigl(g\circ \xi (F)\bigr) (0)=0,
\end{equation*}
which implies \eqref{adjointboundeq}.
\end{proof}\par
\begin{example}
\label{adjointex} 
Let $R={\kk}[x,y]$, $\mathfrak{D}={\kk}_{DP}[X,Y]$, take 
\begin{equation}
\sigma: R\to R,\,  \sigma(y)=y,\, \sigma (x)=x-y^2,
\end{equation}
and let $\xi=\sigma^\vee : \mathfrak{D}\to \mathfrak{D}$.
Then we have $\xi (X^{[n]})=X^{[n]}$ and
\begin{equation}
\begin{array}{c||c|c|c|c|c}
F&X&Y&Y^{[2]}&Y^{[3]}&Y^{[4]}\\\hline
\xi (F)&X&Y&Y^{[2]}-X&Y^{[3]}-XY&Y^{[4]}-XY^{[2]}+X^{[2]}.
\end{array}
\end{equation} 
Also 
\begin{equation}
\begin{array}{c||c|c|c}
F&X^{[k]}Y^{[2]}&X^{[k]}Y^{[3]}&X^{[k]}Y^{[4]}\\\hline
\xi (F)&X^{[k]}Y^{[2]}& X^{[k]}Y^{[3]}& X^{[k]}Y^{[4]}- (k+1)X^{[k+1]}Y^{[2]}\\
&- (k+1)X^{[k+1]} & - (k+1)X^{[k+1]}Y & +{\binom{k+2}{2}X^{[k+2]}}.
\end{array}
\end{equation}
We have in general for $\epsilon\in \{0,1\}$,
\begin{equation}\label{adjointexeq}
\xi(X^{[k]}Y^{[2v+\epsilon]})=\sum_{i=0}^v (-1)^i{\tbinom {k+i}{i}}X^{[k+i]}Y^{[2v+\epsilon-2i]}.
\end{equation}
To verify \eqref{adjointexeq} for $\epsilon =0$ we write 
\begin{align*}
x^{k+i}y^{2v-2i}\circ \xi(X^{[k]}Y^{[2v]})(0)&=\sigma (x^{k+i}y^{2v-2i})\circ X^{[k+i]}Y^{[2v-2i]}(0)\\
&=(x-y^2)^{k+i}y^{2v-2i}\circ X^{[k+i]}Y^{[2v-2i]}(0)\\
&=(-1)^i{\tbinom{k+i}{i}},
\end{align*}
and note that all similar evaluations on other terms $x^sy^t$ are zero. The argument for $\epsilon=1$ is similar.
\end{example}

\subsection{Exotic summands.}
\label{exoticsumsec}

Recall from \cite{BJMR} the notion of \emph{exotic summand} of $F$: this
is one where there are terms of $F$ that involve more variables than
might naively be expected given the sequence of embedding codimensons
$\big\{n_a=H\bigl(R/C(a+1)\bigr)_1$ for $ 0\le a\le j-2\big\}$ of the graded
algebras ${A^*/C(a+1)}$.  This notion is due to A. Bernardi and K.
Ranestad, who defined it in a preprint leading up to \cite{BR}. It depends on an appropriate choice of variables for $\mathfrak{D}$. 
\begin{definition}[Exotic summand] 
\label{exoticdef}
Let $f\in \mathfrak{D}$ have degree $j$, let ${A=R/I}$, ${I=\Ann f}$  and let ${\mathcal{D}(A)=\bigl(H(0),H(1),\ldots\bigr)}$. Set $n_f=(n_0,n_1,\ldots )$ where
\begin{equation} 
n_a=\sum_{i=0}^a H(i)_1.
\end{equation}
Note that $n_a$ is the codimension (embedding dimension) of the algebra ${A^*/C(a+1)}$. Fix a basis $X_1,\ldots ,X_r$ for $\mathfrak{D}_1$ such that for each $a$, with ${0\le a\le j-2}$, ${X_{n_{a-1}+1},\ldots ,X_{n_a}}$ are linear partials of
$f$ of order $j-a-1$:  that is, they lie in $\maxA^{\,j-a-1}\circ f$, but not in $\maxA^{\,j-a}\circ f$. Then, writing $f=\sum f_i$ we define an \emph{exotic summand}\index{exotic summand} of degree $j-a$ of $f$ (for this fixed basis $X_1,\ldots ,X_r$) as an element $e_{j-a}\in \mathfrak{D}_{j-a}$ satisfying 
\begin{equation}
e_{j-a}\in \langle  X_{n_a+1},\ldots ,X_r\rangle {\kk}[X_1,\ldots ,X_r],
\end{equation}
and such that
\begin{equation}
f_{j-a}=h_{j-a}+e_{j-a},  \text { with } h_{j-a}\in  {\kk}[X_1,\ldots ,X_{n_a}].
\end{equation}
In short, after settting a basis of dual variables $X_i$ in the order of their appearance in ${C(a+1)^\perp\cong (A^\ast/C(a+1))^\vee}$, a term of degree $j-a$ in the dual generator $f$ of $A$ is exotic if it involves a variable outside of the first $n_a$ variables.
\end{definition}
The presence or absence of exotic summands can be important in issues of parametrization of non-homogeneous AG algebras up to isomorphism. \par
\begin{example}
\label{blindexotics}
\begin{enumerate}[(a)] 
\item\label{blexi} Let ${f=X^{[3]}+XY\in \mathfrak{D}={\kk}_{DP}[X,Y]}$. In this case the monomial $XY$ is exotic because the Hilbert function of ${A=R/\Ann f}$ is ${H=(1,1,1,1)}$, so ${(n_0,n_1)=(1,1)}$: the codimension one of $A$ itself is lower that the number of variables that occur in $XY$. \par
\item\label{blexii} Consider $w=f'_{\ge 4}= X^{[5]} + X^{[3]}Z + X^{[2]}Y^{[2]} + XY^{[3]}+Y^{[4]}$ from Example \ref{genericmod1ex}(b). Here $\mathcal{D}_w=\mathcal{D}_f$ from \eqref{Dfeq}, and $(n_0,n_1,n_2)=(1,2,3)$ but the term $X^{[3]}Z$ of $w_{j-1}$ involves the third variable, not just the first $n_1=2$ variables $X,Y$ that occur. However, defining $w^\prime$ by replacing this term by $XZ^{[3]}$, we would have $\mathcal{D}_{w^\prime}=\bigl(H_{w^\prime}(0)=(1,1,1,1,1,1),\, H_{w^\prime}(1)=(0,2,3,2,0)\bigr)$, so $n_{w^\prime}=(1,3,3\ldots)$ and $w^\prime$ has no exotic terms. Here $H_w=(1,3,3,2,1,1)$, and $H_{w^\prime}=(1,3,4,3,1,1)$. The simpler dual generator $w^{\prime\prime}=X^{[5]} + X^{[3]}Z + XY^{[3]}$ has $\mathcal{D}_{w^{\prime\prime}}=\bigl(H_{w^{\prime\prime}}(0)=(1,1,1,1,1,1),\, H_{w^{\prime\prime}}(1)=(0,2,4,2,0)\bigr)$, with $n_{w^{\prime\prime}}=n_{w^\prime}$ and also no exotic terms.
\end{enumerate}
\end{example}

\begin{remark}
The Example \ref{blindexotics}\eqref{blexi} illustrates the only way that a polynomial of degree three may have an exotic summand. The reason for this is that in a polynomial of degree $j$ we can always use a linear change of variables to write the top degree term $f_j$ in $n_0$ variables (recall that ${n_0=H(0)_1}$). So the highest degree where an exotic summand can occur is ${j-1}$. But this means that in a polynomial of degree three, a quadratic exotic summand involves variables that do not belong to ${\kk}[X_1,\ldots ,X_{n_0}]$. When degree $f=3$, the integer $n_1$ is just the codimension of $A$.

In general, for a polynomial $f$ of degree $j$, if the apparent number of variables is the same as the codimension of ${A=R/\Ann f}$, exotic summands can only occur if ${H(a)_1\ne0}$, for some ${a>1}$.
\end{remark}

\begin{example}
\label{moreexotics}
Consider ${f=X^{[6]}+X^{[4]}Y+X^{[3]}Z+XYZ\in \mathfrak{D}={\kk}_{DP}[X,Y,Z]}$. Then the ring ${A=R/\Ann f}$ has Hilbert function ${H=(1,3,2,2,1,1,1)}$, with ${H(0)=(1,1,1,1,1,1,1)}$, ${H(2)=(0,1,1,1,0)}$, and ${H(4)=(0,1,0)}$, giving ${(n_0,n_1,n_2,n_3,n_4)=(1,1,2,2,3)}$. Note that ${-x^2(y-x^2)\circ f=Y}$, so ${Y\in(x,y,z)^3\circ f}$. We can check that ${Y\notin(x,y,z)^4\circ f}$, therefore $Y$ is a partial of order $3$. Also ${-(z-x^3)\circ f=Z}$ and ${Z\notin(x,y,z)^2\circ f}$, so $Z$ is a partial of order~$1$. Then ${X,Y,Z}$ is a basis for $\mathfrak{D}_1$ satisfying the conditions in Definition \ref{exoticdef}. Here $X^{[4]}Y$ is an exotic summand because it has degree $5$ and involves $Y$, a linear partial of order $3$ (that is, $\{Q^\vee(a)_i, a\le 1\}$ does not involve $Y$ in degree $i$), and $X^{[3]}Z$ and $XYZ$ are both exotic because they involve $Z$ but $\{Q^\vee(a)_i, a\le 3\}$ does not involve $Z$ in degree $i$. For short, these terms of $f$ are exotic as the number of variables involved in
$f_{j-k}$ for $k=(0,1,2,3,4)$ is $(1,2,3,3,3)$ corresponding to $\bigl(X,(X,Y),(X,Y,Z),(X,Y,Z),(X,Y,Z)\bigr)$: but the variable $Y$ appears late in $Q^\vee(2)$, not in $Q^\vee(1)$ as expected when $n_1=1$, and the variable $Z$ appears late in $Q^\vee(4)$ ($n_2=n_3=2)$.\par

In \cite{BJMR}, the authors give a description of how exotic summands occur, showing that they arise from ``attaching'' a partial of a polynomial to a new variable. In this case, we can start with ${g=X^{[6]}}$ and consider the element $x^2$ adding ${(x^2\circ g)\cdot Y=X^{[4]}Y}$ to $g$, to obtain ${h=X^{[6]}+X^{[4]}Y}$. Next, to make a new exotic summand, we consider the element $x^3$ and add ${(x^3\circ h)\cdot Z=(X^{[3]}+XY)Z}$ to the polynomial $h$, obtaining $f$. 
\end{example}

\subsubsection{Exotic summands can be removed, up to isomorphism.}
\index{dual generator of A@dual generator of $A$!removing exotic terms}%
\index{exotic summand!removing}%

The first statement of Theorem \ref{normalform5thm} is Theorem 5.3 from \cite{I1}. The second statement, concerning the absence of exotic summands after a suitable change of variables, is an immediate consequence of the first, as was pointed out by  J. Jelisiejew, who asked us to confirm his reading of Theorem~5.3 in \cite{I1}.
We give a more explicit rendition of the proof of \cite[Theorem 5.3]{I1} that was sketched there. Then we give an example.
Here $R={\kk}\{x_1,\ldots , x_r\}$, $\mathfrak{D}={\kk}_{DP}[X_1,\ldots ,X_r]$.
\begin{theorem}[\textsc{Normal form for dual generator: removing exotic summands}]
\label{normalform5thm} 
Let $A$ be an AG quotient ${A=R/I}$, ${I=\Ann f}$, ${f=f_j+f_{j-1}+\cdots} $, and define $n_a$ as above. Then there is a change of variables ${\sigma \in \Aut (R)}$, ${ \sigma (w_i)=x_i}$,  such that under the corresponding adjoint linear map $\xi$ on $\mathfrak{D}$, the image ${g=\xi (f)}$ satisfies
\begin{equation}\label{exotic2eq}
g_{j-a}\in {\kk}_{DP}[X_1,\ldots  X_{n_a}].
\end{equation}
The algebra ${A'=R/\Ann g}$ is isomorphic to $A$ and $g$ has no exotic summands with respect to the basis ${X_1,\ldots ,X_r}$ for $\mathfrak{D}_1$.  Also,  ${\mathcal{D}(A)_{\le a}= \mathcal{D} \bigl(B(a)\bigr)_{\le a}}$ for the Artinian algebra\linebreak ${B(a)={\kk}\{x_1,\ldots ,x_{n_a}\}/J(a)}$ where ${J(a)=\Ann(g_j+\cdots +g_{j-a})\cap {\kk}_{DP}[X_1,\ldots ,X_{n_a}]}$.
\end{theorem}
\begin{proof}[Proof of Theorem \ref{normalform5thm}] 
By the definition of $n_a$ we may find a set of local parameters ${w_1,\ldots ,w_r}$ in $\maxR$ such that for ${0\le a\le j-2}$ the classes of ${w_{n_{a-1}+1},\ldots ,w_{n_{a}}}$ span $Q(a)_1$. In particular, the classes of ${w_{n_{a-1}+1},\ldots ,w_{n_{a}}}$ in $A$ lie in ${(0:\maxA^{\,j-a})}$ and therefore
\[
\langle w_{n_{a-1}+1},\ldots ,w_{n_{a}}\rangle\circ f \subseteq \mathfrak{D}_{\le j-a-1}.
\]
If ${n_{j-2}<r}$, we may choose ${w_{n_{j-2}+1},\ldots ,w_{r}}$ such that the initial forms of ${w_1,\ldots ,w_r}$ span ${\maxR/\maxR^{\,2}}$. Consider the ring automorphism $\sigma$ of ${R}$ given by ${\sigma (w_i)=x_i}$, for ${ 1\le i\le r}$ and the adjoint linear map $\xi$ of $\mathfrak{D}$, and let ${g=\xi (f)}$. 
(Warning: $\xi$ is not a ring homomorphism.) By \eqref{adjointboundeq} of Lemma \ref{adjointlem} $g$ has degree $j$. Let ${a\ge0}$ and ${i\ge n_{a-1}+1}$,  (we are working under the convention ${n_{-1}=0}$). Let ${h\in\maxR^{\,j-a}}$. Then
\[
\bigl(h\circ (x_i\circ g)\bigr)(0) = \bigl(\sigma^{-1}(h)\circ (w_i\circ f)\bigr)(0) = 0,
\]
because ${\sigma^{-1}(h) \in \maxR^{\,j-a} }$ and ${w_i\circ f\in \mathfrak{D}_{\le j-a-1} }$. But this evidently implies that
\begin{equation}
g_{j-a}\in {\kk}_{DP}[X_1,\ldots ,X_{n_a}],
\end{equation}
which is equivalent to $g$ having no exotic summands with respect to the basis $X_1,\ldots ,X_r$. For the last statement, we set $x_{n_a+1},\ldots  ,x_r$ equal to zero.
\end{proof}

\begin{example}
\label{3.10connsumex}
Consider ${f=Y^{[4]}+Y^{[2]}X\in \mathfrak{D}={\kk}_{DP}[X,Y]}$, and note that ${(-x+y^2)\circ f=X}$ and ${X\notin(x,y)^2\circ f}$, so $X$ is a partial of order $1$. Since ${y^3\circ f=Y}$, $Y$ is a partial of order $3$. Therefore the basis ${Y,X}$ satisfies the conditions of Definition \ref{exoticdef}, and we can see that the term $Y^{[2]}X$  is exotic, as it is a term of degree $3$ involving $X$. The ideal ${I=\Ann f\subset R={\kk}\{x,y\}}$ satisfies ${I=(x^2,\,xy-y^3)}$; the algebra ${A=R/I}$ has Hilbert function ${H({A})=(1,2,1,1,1)}$ and is stretched in the language of \cite{Sa,ACLY2,CN}. We have $\mathcal{D}({A})= \bigl(H(0)=(1,1,1,1,1),\, H(2)=(0,1,0)\bigr)$ so ${(n_0,n_1,n_2)=(1,1,2)}$. Take ${a=1}$. Then
\begin{equation} 
\mathcal{C}(2)_1=\{\mathrm{in} (h)\mid h\in\maxA\setminus\maxA^{\,2}, \,h\circ f\in \mathfrak{D}_{\le 4-1-2}\}=\langle \mathrm{in}  (x-y^2)\rangle=\langle x\rangle,
\end{equation}
as ${(x-y^2)\circ f=-X}$.  Following the notation of Theorem \ref{normalform5thm}, we let ${w_1=y}$, ${w_2=x-y^2}$. Set ${x_1= \sigma (y)=y}$, ${x_2=\sigma(x-y^2)=x}$.  So ${\sigma(x)=x+y^2}$.  Then the adjoint map ${\xi: \mathfrak{D}\to\mathfrak{D}}$ is the one from Example \ref{adjointex}, and
\begin{equation}
\xi (f)=\xi(Y^{[4]}+Y^{[2]}X)=(Y^{[4]}-YX^{[2]}+X^{[2]})+(Y^{[2]}X-2X^{[2]})=Y^{[4]}-X^{[2]},
\end{equation}
Now $\xi(f)$ has no exotic summands.\par
Thus we have by Theorem \ref{normalform5thm} that 
\begin{equation*}
{{A}=R/I\cong {A}'= R/J} \text { where }{J=\Ann \xi (f)=(x^2+y^4,\,xy)}.
\end{equation*}
Note that this is a rewriting of $I$ using the new parameters as ${x^2\mapsto (x+y^2)^2}$, ${xy-y^3\mapsto xy}$, so ${I=(x^2,\,xy-y^3)\to J=\bigl((x+y^2)^2,\,xy\bigr)=
(x^2+y^4,xy)}$.
\end{example}

\index{curvilinear algebra}%
\subsection{Isomorphism class}
\label{isomclasssec}
\begin{remark}[Parametrization vs. isomorphism class]  
A length-$n$ quotient $A$ of $R={\kk}\{x,y\}$ or $R'={\kk}\{x_1,\ldots,x_r\}$ is called ``curvilinear'' if $H(A)=(1,1,\ldots ,1)$. Such an algebra satisfies $A\cong {\kk}\{x\}/(x^n)$, for discussion see \cite[Example 2.17]{ChI}, \cite{Bri,I9}. The Gorenstein algebra quotients $\mathcal{A}={\kk}\{x,y\}/I$ with $H=H(\mathcal{A})=(1,1,1,1)$ form a family  $Z_H$.  Their associated graded algebras are a subfamily of the variety $G_H$ that parametrizes graded algebra quotients of $R$ having Hilbert function $H$ (Definition \ref{basicdef}). A typical element  of $Z_H$ is determined by $I=I_{a_1,a_2,a_3} =(y+a_1x+a_2x^2+a_3x^3,\, x^4)$ or $J=(x+b_1y+b_2y^2+b_3y^3,\, y^4)$, and they have associated graded ideal
$I^\ast=(y+a_1x, \mathfrak{m}^4)$ or $J^\ast=(x+b_1y, \mathfrak{m}^4)$, an element of the projective line $\mathbb P^1$. Thus $\pi: Z_H\to G_H\cong \mathbb P^1$ is fibred by an affine plane,  and there is a section ${\sf s}: G_H\to Z_H$, but this is not a vector bundle \cite{I0}.  This is a common occurrence for the maps $Z_H\to G_H$ in two variables and different Hilbert functions $H$. The affine bundles that are not vector bundles have recently been further studied by W.~Haboush and D. Hyeon in connection with families of commuting nilpotent matrices \cite{HH}. 
\end{remark}

The AG algebras of Hilbert function $H=(1,2^a,1^b)$ are studied up to isomorphism by J.~Elias and M. Rossi \cite{ER1}. Although certain short - as all socle degree 3 and socle degree 4-compressed AG algebras over $\mathbb C$ have, strikingly, been shown to be canonically graded- isomorphic to their associated graded algebra \cite{ER1,ER,ER2},\footnote{The article \cite{ER2} shows that socle degree three compressed algebras over $\mathbb C$ are canonically graded; this extends to characteristic not 2 (\cite[Example 2.16]{Je3}). Then \cite[Theorem 3.1]{ER2} shows that socle degree four compressed algebras over $\mathbb C$ are canonically graded; \cite[Corollary 3.15]{Je3} shows this also in characteristics not $2$ or $3$.} it is easy through a dimension calculation to show that this cannot occur in general. In particular those AG algebras of Hilbert function $(1,n,n,1)$ are shown in \cite{ER} to be canonically graded, but they note an example of an AG algebra of Hilbert function $H=(1,2,2,2,1)$ that is not canonically graded.\par
J. Jelisiejew in \cite{Je2} classifies algebras with Hilbert function $H=(1, 3, 3, 3, 1)$, obtaining finitely many isomorphism types; he also classifies up to isomorphism those with Hilbert function $H=(1, 2, 2, 2, 1, 1, 1)$ over fields of characteristic zero or larger than $6$.  In $\cha \kk =0$ he uses new techniques related to Lie algebra, considering the orbits of the $\Aut (R)$ action on the family $Z_H$, and the tangent spaces to the orbits.\par

\subsubsection{Characteristic dependence of Hilbert function.}
We give an example showing that the Hilbert function $H(A)$ and hence the symmetric decomposition $\mathcal{D}(A)$ for $A=R/\Ann f$ for a fixed dual generator $f$ may depend upon the characteristic of $\kk$, even using the contraction action of $R$ on $\mathfrak{D}$.  It is open whether there are such examples where $H(A)$ and $f$ remain fixed but $\mathcal{D}(A)$ depends on the characteristic.\par
\index{characteristic dependence@characteristic dependence!of Hilbert function}%
Of course, even considering graded AG algebras, taking $f=X^{[4]}+nX^{[2]}Y^{[2]}+Y^{[4]}$ we have $H(A)$, $A=R/\Ann f$ satisfies $H(A)=(1,2,3,2,1)$ except in characteristic $p$ dividing the integer $n$, when $H(A)=(1,2,2,2,1)$. We propose the following more subtle example, where the dual generator does not so obviously change with $\cha {\kk}$.
\begin{example}[Characteristic dependent Hilbert function for $F=L^{[m]}+X^{[n]}Y^{[n]}$]
\label{chardependex} 
Consider first the special case $R={\kk}\{x,y\}$, $f= X^{[2]}Y^{[2]}\in \mathfrak{D}={\kk}_{DP}[X,Y]$ and take $ L=X+aY$. Then the ideal $\Ann L=(ax-y)$ and $A=R/\Ann f$ is compressed of Hilbert function $H_A=(1,2,3,2,1)$. Let $F=L^{[6]}+f$.  Writing the matrix for the basis $\{L^{[2]},\, f_1=(ax-y)y\circ f,\, f_2=(ax-y)x\circ f\}$ of $(R\circ F)_2$ in terms of the basis $X^{[2]},XY, Y^{[2]}$ of $\mathfrak{D}_2$ we obtain
\begin{equation}\label{matrix1eq}
\begin{array}{c|ccc}
&L^{[2]}&f_1&f_2\\
\hline\\[-1.8ex]
X^{[2]}&1&-1&0\\
XY&a&a&-1\\
Y^{[2]}&a^2&0&a\\
\end{array}
\end{equation}
whose determinant is $3a^2$. Thus for $\cha {\kk}\not=3$, for general enough $L$ we have $H_F=(1,2,3,2,1,1,1)$, where $H_F(0)=(1^7)$ and $H_F(2)=(0,1,2,1)$. But when $\cha {\kk}=3$, for every $L$ we have $H_F=(1,2,2,2,1,1,1)$, where $H_F(2)=(0,1,1,1)$. Here $X^{[2]}Y^{[2]}$ is never an exotic summand of $F$ as there is the expected number of variables involved in $Q(0)$ and in $Q(1)=R/\mathcal{C}(1)$. \par
Since this is a structural issue, the same exceptional behavior in characteristic 3 must occur for $F=L^{[6]}+B^{[2]}\cdot C^{[2]}$ where $B,C\in \mathfrak{D}_1$ are linear. However, the product $B^{[2]}\cdot C^{[2]}$ is in the divided power sense. Taking $B=X+bY, C=X+cY$ we have
\begin{align*}
B^{[2]}\cdot C^{[2]}&=(X^{[2]}+bXY+b^2Y^{[2]})\cdot (X^{[2]}+cXY+c^2Y^{[2]})\\
&={\tbinom{4}{2}}X^{[4]}+{\tbinom{3}{2}} (b+c)X^{[3]}Y+(b^2+c^2)X^{[2]}Y^{[2]}\\
&\equiv (b^2+c^2)X^{[2]}Y^{[2]} \text { when } \cha {\kk}=3.
\end{align*}
\normalsize In fact, $f= X^{[2]}Y^{[2]}$ or a multiple is the only exceptional degree four form in ${\kk}_{DP}[X,Y]$ for which there is a characteristic-dependent Hilbert function for $F=L^{[m]}+f$, $m>4$.\par
Note that $H_F=(1,2,3,2,1,1,1)$ is the Hilbert function of a relatively compressed $a=2$-modification in $R$ of $A_0=R/{\mathcal C}(0)$ (where $H(0)=(1,1,\ldots, 1)$). Proposition~\ref{maxprop}(b) concerning the existence of RCM's requires only that $\kk$ be an infinite field, so this is not an example where the set of possible Hilbert functions for fixed embedding dimension and socle degree depends on the characteristic of $\kk$: that is, we can achieve $H_G=H_F$ in any characteristic by taking $G=g+L^{[6]}$, with $g$ general enough of degree 4 (see Question \ref{Frobenius2ques}).
\vskip 0.2cm
Now consider more generally $f=X^{[n]}Y^{[n]}$ and $F=f+L^{[m]}$, $L=X+aY$ with $m>2n, n\ge 2$. Note that $x^{m-n}\circ F=x^{m-n}\circ L^{[m]}=\alpha L^{[n]}$ with $\alpha\in \kk$ and $\alpha\not=0$. We form a matrix $M_F$ whose $(n+1)$ rows correspond to the monomials of $\mathfrak D_n$, and whose columns correspond to the elements of $\mathfrak D_n$ spanning $(R\circ F)_n$.  The first column is $L^{[n]}$; there are $n$ further columns from applying $\Ann L=ax-y$ to a monomial basis of $R^{n-1}\circ f\subset \mathfrak{D}_{n+1}$, as $((ax-y)R_{n-1})\circ F=((ax-y)R_{n-1})\circ f\subset (R\circ f)_n$. That is, $M_F$ has columns
\begin{align*}
L^{[n]}, f_1&=(ax-y)\circ X^{[n]}Y, \ldots ,f_i=(ax-y)\circ X^{[n+1-i]}Y^{[i]},\ldots,f_n=(ax-y)\circ XY^{[n]}, \text { where }\\
f_i&=-X^{[n+1-i]}Y^{[i-1]}+aX^{[n-i]}Y^{[i]},\, 1\le i\le n.
\end{align*}
The matrix $M_F$  (see Figure \ref{matrixfig}) is readily seen to have determinant $(n+1)a^n$. 
\begin{figure}
\begin{equation}\label{matrix2eq}
M_F:\quad \begin{array}{c|ccccc}
&L^{[n]}&f_1&f_2&\ldots &f_n\\
\hline\\[-1.8ex]
X^{[n]}&1&-1&0&\ldots &0\\
X^{[n-1]}Y&a&a&-1&\ldots &0\\
X^{[n-2]}Y^{[2]}&a^2&0&a&\ldots &0\\
\cdots &&\ldots &&&\\
XY^{[n-1]}&a^{n-1}&0&\dots&a&-1\\
Y^{[n]}&a^n&0&\ldots &0&a\\
\end{array}
\end{equation}
\caption{Matrix $M_F$ of $(R\circ F)_n$ for $ F=X^{[n]}Y^{[n]}+L^{[m]}$.}\label{matrixfig}
\end{figure} 
Thus, we have, for general enough $L$,
\begin{align*}
H_F&=(1,\, 2,\ldots ,\, n-1,\, n,\, n+1,\, n,\, n-1,\ldots ,\, 2,\, 1_{2n},\, 1,\ldots ,\, 1_m),\\
\mathcal{D}_F&=\bigl( H(0)=(1,1,\ldots ,1_m),\, H(m-2n)=(0,1,2,3,\ldots, n-1,n,n-1,\ldots,3,2,1_{2n-1},0)\bigr),
\end{align*}
except when $\cha k$ divides $n+1$, in which case $H(m-2n)$ has maximum value $n-1$, and
\begin{equation*}
H_F=(1,\,2,\ldots,\, n-1,\, n,\, n,\, n,\, n-1,\ldots ,2,\, 1_{2n},\ldots , \,1_m).
\end{equation*}
\end{example}

\subsection{Connected sums.}\label{connectedsumsec}
The term ``connected sum" was introduced to topology apparently by John Milnor around 1960, then used for Artinian Gorenstein algebras over a field ${\kk}$ by C.-H. Sah and J. Lescot
\cite{Sah, Les}. Connected sums of graded Artinian Gorenstein algebras occur prominently in the topology-influenced book by M. Meyer and L. Smith \cite[p. 11]{MS}; their relation to topology is also studied in the article by L. Smith and R.E. Stong
\cite{SmSt1} on graded Poincar\'{e} duality algebras over $\cha \kk=2$. They have been more recently studied by H. Ananthnarayan, A.~Avramov and W.~F.~Moore, then by H.~Ananthnarayan, E. Celikbas, J. Laxmi, and Z.~Yang  \cite{AAM, ACLY1,ACLY2}. T. Matsumura and W.~F.~Moore studied connected sums of simplicial complexes and their Stanley-Reisner rings \cite{MaMo}. Other recent articles have focused on the graded case \cite{BBKT}.  \par The basic concept of an AG connected sum is simple: write the Macaulay dual generator as a sum of polynomials (divided powers) in two separate sets of variables. The concept occurs also in \cite[Theorem 5.5]{I1}, where the emphasis is in constructing AG algebras with certain given Hilbert functions, by writing the dual generator $f$ as a sum of specified powers of a sequence of generic linear forms (\cite[Theorem 5.8]{I1}). We work over an arbitrary infinite field ${\kk}$, recall $R={\kk}[x_1,\ldots, x_r]$.
\begin{definition}
\label{connectedsumdef} 
An AG algebra ${A=R/I}$, ${I=\Ann f}$ is a \emph{connected sum}\index{connected sum} over the field $\kk$ if, possibly after a coordinate change, there is a decomposition of the variables $L$ of $\mathfrak{D}$ into two disjoint subsets,
$L=L_1\cup L_2$ so that $f=f_1+f_2$ with each $f_i$ a polynomial in the variables $L_i$. We denote by $\mathfrak{D}(i)$, $i=1,2$ the corresponding divided power rings, by $R(i)\subset R$, $i=1,2$ the corresponding rings,  by $\mathcal{X}_1,\mathcal{X}_2\subset \{x_1,x_2,\ldots ,x_n\}$, the variables of $R(1)$, $R(2)$, respectively, so $R(i)={\kk}[\mathcal{X}_i]$; and we denote by $I(i)\subset R(i)$ the ideal $I(i)=(\Ann f_i)\cap R(i)$, $i=1,2$.
\end{definition}
To be more precise we say that ${A}$ is a connected sum in $\#L_1$ and $\# L_2$ variables, or that ${A}$ has a connected summand in $\# L_1$ variables. 
\begin{lemma}{\rm \cite{MS,ACLY1}}
\label{connsumlem}  
Let ${A}=R/I$ be a connected sum, and let ${A(i)=R(i)/I(i)}$ for $ {i=1,2}$.  Then ${I\supset \mathcal{X}_1\cdot \mathcal{X}_2}$ and $I$ has a minimal generator of order ${j_{\min} =\min\{ j_1, j_2\}}$ with $j_i$ the socle degree of  $A(i)$. 
The Hilbert function $H({A})$ satisfies
\begin{equation}\label{HFconnsumeq}
H(A)=H({A}_1)+H({A}_2)- (1,0,\ldots ,0, 1_{j_{\min}}).
\end{equation}
\end{lemma}
See \cite{ACLY1} for a discussion: they in fact show that connected sums (even up to isomorphism) are quite rare: see for example their Theorems 3.6 and 3.9. It is shown using just a dimension calculation in \cite[Proposition 4.4]{SmSt} that when $r\ge 4$ and $j\ge 3$ there are Gorenstein algebras that are not a connected sum. \par The following is a result that is due essentially to H. Ananthnarayan, E. Celikbas, and Z.~Yang  \cite[Theorem 34]{ACY}; our statement is slightly different,  and we give a different proof. The authors of \cite{ACY} introduce  ``graded Gorenstein up to almost linear socle'' and work entirely in the ring $R$: the Theorem 39 of \cite{ACY} generalizes the results of J. Sally \cite{Sa} on stretched Gorenstein algebras and one of the results in J. Elias and M. Rossi in \cite{ER} (see also \cite[Theorem 5.6]{ACLY2}). We arrived at the following slightly more general formulation after considering their result.
To prove this, we work in the dual ring, with the dual generator, and we apply the Normal Form Theorem \ref{normalform5thm}.\footnote{Since first writing this we learned that there is a proof of essentially the same statement, however, without reference to characteristic  (so, presumably, in characteristic zero) in \cite[Proposition~5.1]{Je}; this result is shown over an algebraically closed field of arbitrary characteristic not 2 or 3 in \cite[Proposition 4.5]{CJN}. The result is stated in somewhat different language and proven differently in \cite[Theorem 4.3]{CN}. We have referred to \cite{ACY}, which has been replaced by the later \cite{ACLY1} and \cite{ACLY2}, although not all results exactly correspond.}
\begin{theorem}{\rm \cite{ACY}}
\label{0s0thm} 
Let ${A}$ be an AG algebra of socle degree $j$, and suppose that $H(j-2)=(0,s,0)$. Let ${\kk}$ be an infinite field not having characteristic two. Then ${A}$ has a connected summand in $s$ variables.
\end{theorem}
\begin{proof}  
Let $f$ be the dual generator for ${A}$. We may assume ${A}$ has codimension $n$, let $R={\kk}\{x_1,\ldots ,x_n\}$, so we may take $f_1=0$. By assumption $n_{j-3}=n-s$, so by Theorem \ref{normalform5thm} we may assume that after a change of variables,  $f_j+\cdots +f_{3}\in {\kk}_{DP}[X_1,\ldots ,X_{n-s}]$. Let ${\mathfrak{D}} (1)= {\kk}_{DP}[X_1,\ldots X_{n-s]}$, ${\mathfrak{D}}(2)= {\kk}_{DP}[X_{n_s+1},\ldots ,X_n]$. We may write
\begin{equation}
f_2=g_2+\sum_{i=1}^s L_i X_{n-s+i}+h_2 \text { where }  g_2\in  { \mathfrak{D}}(2)_2,\, L_i\in  { \mathfrak{D}}(1)_1,\, h_2\in  { \mathfrak{D}}(1)_2.
\end{equation}
Since $\cha {\kk}\not=2$, we may diagonalize $g_2$, after a change of basis of ${\mathfrak{D}} (2)$, so we may write
\begin{equation}
g_2=\sum_{i=1}^s  a_iX_{n-s+i}^{\,[2]}.
\end{equation}
Here each $a_i\not=0$, else we claim $H(j-2)_1<s$.  For if $a_i=0$ then $X_i$ must appear in $R\circ f$ by the action of an linear element $\ell\in \langle x_1,\ldots ,x_{n-s}\rangle$ satisfying $\ell \circ L_i \not=0$. This can happen only if  $\ell\circ (f_j+\cdots +f_3)=0$, which contradicts the assumption that $\sum_{i=0}^{j-3} H(i)_0=n-s$ (since only terms of $f$ having degree at least three can contribute to that sum). This shows that $a_i\not=0$ for each $i$.
\par 
Now replacing  $X_{n-s+i}$ by $X_{n-s+1}+L_i/2$, and replacing $h_2$ by $h_2-\sum L_i^{\,[2]}/4$, we have $f_2=\sum X_{n-s+i}^{\,[2]}+h_2$ and have written $f$ as a connected sum of $(f_j+\ldots +f_{3}+h_2)\in  { \mathfrak{D}}(1)$ and $g_2\in  { \mathfrak{D}}(2)_2$, so
${A} $ has a connected summand in $s$ variables.
\end{proof}\par
The following example shows that the statement of Theorem \ref{0s0thm} does not extend to adding on $H(j-3)=(0,s,s,0)$.
\begin{example}[RCM of a curvilinear algebra may not be a connected sum]
\label{nonCS1ex}
Consider the form ${F=Z^{[4]}+YX^{[2]}+Y^{[2]}Z}$ in ${\mathfrak{D}={\kk}_{DP}[X,Y,Z]}$ and ${R={\kk}\{x,y,z\}}$,  ${A=R/I}$, where $I=\Ann F=(x^2-yz,\, xz,\, y^2-z^3)$.  Here $H(A)=(1,3,3,1,1)$ and $A$ is an RCM of the curvilinear algebra $A(0)\cong R=R/(x,y,z^5)$. We have  
\begin{equation*}
{\mathcal{D}(A)=\bigl(H_A(0)=(1,1,1,1,1),\, H_A(1)=(0,2,2,0)\bigr)}.
\end{equation*}
Since $A$ is a complete intersection, it cannot be a connected sum:  in order to be a connected sum by Lemma \ref{connsumlem} applied to $r=3$ variables, $I$ must satisfy, $I\supset \ell\cdot  V$ for some $\ell\in R_1$ and a space $V\subset R_1$ of dimension 2.\vskip 0.2cm
The algebra $A'=R/\Ann(F^\prime)$, ${F^\prime=Z^{[4]}+X^{[3]}+Y^{[2]}Z}$, ${I_{F'}=(xy,\,xz,\,y^2-z^3,\,x^3-y^2z,\,yz^2)}$ is evidently a connected sum, and is an RCM of $A'' =R/\Ann F''$, $F''=Z^{[4]}+YZ^{[2]}$ where $H_{A''}(1)=(0,1,1,0)$, and $I_{F''}=(x,\,yz-z^3,\,y^2)$.
\par  Each of $F$, $F'$ determines an AG algebra of decomposition $\mathcal{D}(A)$.
\end{example}
We don't know if there is a counterexample to the connected summand statement of Theorem \ref{0s0thm} when $\cha \kk =2$.
\begin{example} 
Consider again ${f=Y^{[4]}+Y^{[2]}X\in \mathfrak{D}={\kk}_{DP}[X,Y]}$, $A={\kk}\{x,y\}/\Ann f$ from Example \ref{3.10connsumex}, of Hilbert function ${H({A})=(1,2,1,1,1)}$.  Here ${w_1=y}$, ${w_2=x-y^2}$, ${x_1= \sigma (y)=y}$, ${x_2=\sigma(x-y^2)=x}$.  Then the adjoint map ${\xi: \mathfrak{D}\to\mathfrak{D}}$ is
\begin{equation}
\xi (f)=\xi(Y^{[4]}+Y^{[2]}X)=(Y^{[4]}-YX^{[2]}+X^{[2]})+(Y^{[2]}X-2X^{[2]})=Y^{[4]}-X^{[2]},
\end{equation}
which is a connected sum.\par
It can be shown, similar to \cite{I0,HH} that the variety $\mathfrak{Z}_H$ parametrizing AG quotients $A$ of $R$ having Hilbert function ${H({A})= H=(1,2,1,1,1)}$ is an $\mathbb{A}^3$ bundle over the projective line ${\mathbb{P}^1=\mathfrak{G}_H}$ that parametrizes the associated graded ideals ${I^\ast (\ell)=(\ell \cdot x,\ell \cdot y,\maxA^{\,5})}$, ${\ell\in R_1}$.  The fibre of $Z_H$ over $\mathbb P^1$ when $\ell= x$, as above, is
\begin{equation}
I_{a,b,c}=(x^2-cy^4,\,xy-ay^3-by^4), 
\end{equation}
with ${a\ne0}$, whose dual generator is $F_{a,b,c}=Y^{[4]}+aY^{[2]}X+bYX+cX^{[2]}$. By Theorem \ref{0s0thm} each such algebra $A$ has a connected summand. 
\end{example}
\subsection{Questions and open problems.}\label{questionsec}
Our questions indicate what we feel are some worthwhile open directions for exploration, most of which we have not discussed previously in the paper.

First, the structure theorem for codimension two AG algebras (see \cite[Theorem 2]{I5},\cite[Theorem 2.2]{I1}) shows that the symmetric decomposition $\mathcal{D}(A)$ of the Hilbert function in codimension two depends only on the Hilbert function $H(A)$.\footnote{In codimension two, also, the associated graded algebra $A^\ast$ of an AG algebra determines the ideals $C(a)$ and the subquotients $Q(a)$, which are graded complete intersections (Remark \ref{diffdecomprem} and \cite[Theorem 2.2]{I1}).} The sequences possible for $H_A(0)=H\bigl(Q_A(0)\bigr)$ (the graded Gorenstein sequences)  are known in codimension three, as a result of the D.~Buchsbaum-D. Eisenbud structure theorem \cite{BuEi}; that theorem also implies that the number of generators of a height three Gorenstein ideal is odd, first shown by J.~Watanabe \cite{W1}. However, the following question for codimension three and higher is quite open for $a>0$.
\begin{question}  
Are there distinguishing characteristics of symmetric decompositions for 
AG algebras that depend on the codimension? For example, are the modules $Q_A(a)$ or the set of Hilbert functions $H_A(a)$ simpler in codimension three than in codimension four? 
\end{question}
\index{parametrizing AG algebras of given symmetric decomposition}%
\begin{question}[Parametrization] 
Given a Gorenstein sequence $H$ and a symmetric decomposition $\mathcal{D}$ of $H$, can we parametrize the AG algebras having decomposition $\mathcal{D}$?  Can we bound the dimension of the family of quotients of $R$ having decomposition $\mathcal{D}$?  Answers are known when $r=2$ (see \cite[Section 2]{I1}). 
\end{question}
\index{deformation within ${\Gor}(H)$}%
\begin{question}[Deformation within a given Gorenstein sequence] 
Fix a Gorenstein sequence $H$. Describe the possible symmetric decompositions and the closure of the family of AG quotients of $R$ having a given symmetric decomposition in $\Gor_H(R)$. See \cite[\S 4]{I1}.
\end{question}
\index{GorHR@$\Gor_H(R)$!closure}%
\begin{question}[Closure of $\Gor_H(R)$]
Fix a Gorenstein sequence $H=(1,r,\ldots )$ of length $n$ ($H$ is usually non-symmetric). Determine the closure of $\Gor_H(R)$ in the family of all Artinian algebra quotients of $R$. See \cite{Bri,I8} where the case $r=2$ is studied. In particular J. Brian\c{c}on shows that every ideal of $R=\mathbb C\{x,y\}$ over the complexes has a
CI deformation: this is the first step in his proof that the fibre of the punctual Hilbert scheme $\mathrm{Hilb}^n \mathbb P^2$ over a point of $\mathbb P^2$ is irreducible.
\end{question} 
\index{elementary components!of punctual Hilbert scheme}%
\begin{question}[Elementary components] 
A \emph{generic} AG algebra $A$ is one such that any deformation of $A$ has the same Hilbert function and symmetric decomposition. An elementary component of the Hilbert scheme of length-$n$ schemes in $\mathbb A^r$ is one parametrizing local algebras concentrated at a single maximum ideal, so quotients of $R$.  An example of a generic AG algebra corresponding to an elementary component of $\Hilb^n(\mathbb A^r)$ is a general enough AG algebra $A=R/I$ with $H(A) =(1,r,r,1)$ for $r=6$ \cite{EmI,Je1} and for $8\le r\le 12$ (and, conjecturally, for all $r\ge 8$, and many more AG graded Hilbert functions of higher socle degrees \cite{EmI,IKa}.\footnote{In \cite{IKa}, proof of Lemma 6.21, it is stated that a small tangent space criterion (STC) of Equation  (6.31) there, has been checked for $H=(1,r,r,1)$ for $r=6$, and for $8\le r\le 12$ by computer calculation (there is a change of notation); the STC gives elementary components. In the proof of Corollaries 6.28, 6.29 there it is stated that Equation  (6.3.1) has been verified for the compressed Gorenstein Hilbert functions  $T(15,4)$ and $T(5,5)$. Here $j=15$ is the smallest socle degree in height 4 for which STC can work and $j=5$  is the smallest socle degree in height 5 for which STC can work.}    What are the generic AG algebras: in particular, are there any for which $H(A)\not=H_A(0)$, that is, whose Hilbert function is not symmetric?  It is remarkable that M. Huibregtse has determined elementary components of the Hilbert scheme in five or more variables that are essentially non-graded: but they are non-Gorenstein \cite{Hu}. For some recent results and further reading on elementary  components of the Hilbert scheme, see \cite{Hu,Je3} and the references there.
\end{question}
\index{isomorphism classes of AG algebras}%
\begin{question}[Isomorphism classes] 
Give parameters for isomorphism classes of AG algebras. That there are continuous families of isomorphism classes is evident for dimension reasons, and was shown in codimension two as early as J. Brian\c{c}on's thesis (see \cite{Bri}). This problem has been studied extensively for very small colengths or socle degree, and certain Hilbert functions \cite{Bri,CN1,CN, CJN, EH,ER1,ER,Is,Je,Je2}. 
\end{question}
\index{Frobenius structure on ${\Gor}(H)$!when $\cha \kk =p$}%
\begin{question}[Frobenius structure in $\cha \kk=p$]
\label{Frobenius2ques}  
When $\cha \kk =p$, finite, we have an additional Frobenius structure on Gorenstein algebras, that has been described by Larry Smith and coauthors and by representation theorists.   See \cite{MS}, section II.6 on Frobenius powers and Chapter III on ``Poincar\'{e} duality and the Steenrod algebra'', also \cite{SmSt1,SmSt}. It is an open question, whether the set of Gorenstein Hilbert functions, or their symmetric decompositions, when we fix the embedding dimension and socle degree, depend on the characteristic of $\kk$.
\end{question}
\index{characteristic dependence@characteristic dependence!of Hilbert function}%
\begin{question}[Higher dimension Gorenstein algebras] 
What are the implications of symmetric decomposition of the Hilbert functions of AG algebras, for Gorenstein algebras $A$ of higher dimension?
Of course, we can always mod out by a system of parameters. In \cite[Lemma 1.1]{EI} J. Elias and the first author  with L. Avramov  show that the Hilbert function decomposition of a quotient of $A$ by a general enough system of parameters $L$ is an invariant of $A$: that is the dimensions $\dim_{\kk} (0:\mathfrak{m}^i)\cap \mathfrak{m}^j$ for $A/(L)$ depend only on $A$ when the s.o.p $L$ is general enough. (See also the discussion in \cite[\S 5D.i.]{I1}). Higher dimension Gorenstein algebras are also studied in \cite{ER3,RV,ERV}.
\end{question}
\index{ubiquity!partial non-ubiquity}%
\begin{question}[Non-ubiquity] 
Pertaining to Question 1, we conjecture that for $r\ge 3$ there exist algebras $A$ with decompositions
$\mathcal{D}(A)=\left(H(A)_0,\ldots \right)$ such that $\mathcal{D}(A)_{\le a}$ cannot occur as the complete Hilbert function decomposition for some algebra $B$; such partial decompositions we term non-ubiquitous. In section \ref{nonubiquitycod4sec} we gave a partial result in codimension four, where we restrict the Macaulay dual generator of $B$ to have high enough order.
We have similar results in other codimensions at least three, but no example where we have proven non-ubiquity. 
\end{question}
\index{Jordan type and Lefschetz properties}%
\begin{question}[Jordan type, and Lefschetz Properties] 
The multiplication map $m_\ell$ for $\ell\in \mathfrak{m}\subset A$ on $A$ is nilpotent, so over any field $\\k$ it has a \emph{Jordan type} $P_\ell$, a partition of $n=\dim_{\kk}A$ given by the conjugacy class of $m_\ell$. These satisfy natural closure conditions and can be evaluated on both $A$ and $A^\ast$.  Weak Lefschetz and Lefschetz properties can be phrased in terms of the Jordan type of $\ell$: strong Lefschetz is equivalent to the Jordan type $P_\ell=H(A)^\vee$ the conjugate partition to $H(A)$ (regarded as a partition), and weak Lefschetz is that the number of parts of $P_\ell$ is the Sperner number of $A$, the maximum value of $H(A)_i$. See \cite[Chap 3]{H-W}. Although studied in codimension 2 (\cite[Chapter 2]{I1},\cite[Proposition~3.15]{H-W},\cite{AIK})
the study for  algebras of higher codimension is nascent.  There has been a past tendency to focus on strong and weak Lefschetz, which are special cases, but there is recent motion in the direction of understanding other Jordan types, particularly when $A$ is graded
\cite{AIK,AIKY,BI,CGo,Gon,IMM}. 
\index{GorHR@$\Gor_H(R)$!several irreducible components}%
In a work in progress \cite{IM} the authors show that semicontinuity properties of Jordan type, and the semicontinuity of symmetric decompositions \cite[\S 4.1]{I1} for a given Hilbert function can combine to show there is an infinite set of families $\Gor_H(R)$ having several irreducible components.\footnote{That $\Gor_H(R)$ for $H=(1,3,3,2,1,1)$ has several irreducible components is shown using only the semicontinuity of symmetric decomposition, and dimension counts for the symmetric strata in \cite[\S 4.1]{I1}.}
\end{question}
\index{central simple modules}%
\index{symmetric decomposition of Hilbert function!for given other filtrations}%
\begin{question}[Central simple modules, other filtrations] 
In place of intersecting the $\mathfrak{m}-$adic and Loewy $0:\mathfrak{m}^j$ filtrations, we may make similar constructions using other pairs, for example intersecting the $L^i$ and $ {0:L^j}$ filtrations where $L$ may be an element of $\mathfrak{m}$ or an ideal in $A$. Which results concerning the symmetric decompositions extend more broadly?  What can we say about tensor products? These questions are studied implicitly by T. Harima and J.~Watanabe in their concept of ``central simple modules''  \cite{HW,HW2,HW3}. They in \cite{HW2} discuss non-standard grading: their work uses both the action of multiplication by $z$ and the $\mathfrak{m}$-adic filtrations. See also \cite[{\S} 4.1]{H-W}. 
There is further work by many on multigradings, and filtrations by powers of ideal as \cite{Ba,RV,TV}; in the context of Gorenstein Artin algebras there is a special duality that can be explored \cite{BI}.
\end{question}
\begin{question}[Non-standard grading] 
Although we have worked in the paper with the standard grading on the
ring $R={\kk}\{x_1,\ldots,x_r\}$ where the variables have weights one, the concepts apply to any non-negative grading, see \cite{IMM,KK}.  An Artinian Gorenstein algebra $A$ of relative coinvariants is naturally homogenous for a non-standard grading on the variables; by ignoring the original weights one determines a local AG algebra whose associated graded algebra has a symmetric decomposition \cite{MCIM}.
\end{question}
\begin{ack} 
We appreciate conversations with Oana Veliche, Marilina  Rossi, and with Ela Celikbas, who communicated an early version of \cite{ACY}.  We appreciate comments of and discussion with Larry Smith and Joachim Jelisiejew. The first author appreciates his many conversations over the years with Jacques Emsalem, whose insight concerning Artinian Gorenstein local rings, and whose fundamental note \cite{Em} have been important in this work. We are greatly appreciative of comments by the referee, which led to many clarifications.
The second author was partially supported by CIMA -- Centro de Investiga\c{c}\~{a}o em Matem\'{a}tica e Aplica\c{c}\~{o}es, Universidade de \'{E}vora, project UIDB/04674/2020 (FCT -- Funda\c{c}\~{a}o para a Ci\^{e}ncia e Tecnologia), and by FCT project ``Comunidade Portuguesa de Geometria Alg\'{e}brica'', PTDC/MAT-GEO/0675/2012. Parts of this work were done while the second author was visiting math departments of Northeastern University, University of Campinas, KU~Leuven, and University of Connecticut. He thanks them for their hospitality.
\end{ack}

\printindex
\end{document}